\documentclass[letter,12pt]{article}
\usepackage[T1]{fontenc}
\usepackage{graphicx,color}
\usepackage{float}
\usepackage{subfigure}

\usepackage{verbatim,comment}
\usepackage{amsmath,amsthm}
\usepackage{xspace}
\usepackage{pifont}
\usepackage{amssymb}
\usepackage{epic, eepic}
\usepackage{dsfont}
\usepackage{amssymb}
\usepackage{makeidx}
\usepackage{mathrsfs,enumerate}

\usepackage{amsmath,afterpage}
\usepackage{epsf}

\usepackage{bm}

\newcommand{\zz}[1]{\mathbb{#1}}
\def\q{\quad}

\def\x{\mathbf{x}}
\def\y{\mathbf{y}}
\def\b{\mathbf{b}}

\def\z{\mathbf{z}}
\def\0{\mathbf{0}}

\def\vep{\varepsilon}
\def\HS{\mathcal{H}}
\def\SHS{\mathcal{Q}}
\def\K{\mathcal{K}}

\def\DI{\mathcal{I}}
\def\DR{\mathcal{R}}
\def\DS{\mathcal{S}}
\def\DD{\mathcal{D}}
\def\Dd{D}

\def\eps{\varepsilon}

\def \< {\langle}
\def \> {\rangle}

\def\ul{\underline}

\def\S1{\mathcal{R}^1}
\def\SN1{\mathcal{Q}^1}

\def\beqa{\begin{eqnarray}}
\def\eeqa{\end{eqnarray}}
\def\beqas{\begin{eqnarray*}}
\def\eeqas{\end{eqnarray*}}

\newtheorem{theorem}{Theorem}
\newtheorem{thm}{Theorem}

\newtheorem{lemma}[theorem]{Lemma}

\newtheorem{prop}[theorem]{Proposition}

\newtheorem{cor}[theorem]{Corollary}

\newtheorem{remark}[theorem]{Remark}

\newcommand{\hatd}[1]{{}}

\newcommand{\bd}{\begin{displaymath}}
\newcommand{\ed}{\end{displaymath}}
\newcommand{\be}{\begin{equation}}
\newcommand{\ee}{\end{equation}}
\newcommand{\bq}{\begin{eqnarray}}
\newcommand{\eq}{\end{eqnarray}}
\newcommand{\bn}{\begin{eqnarray*}}
\newcommand{\en}{\end{eqnarray*}}

\newcommand{\re}{\zz{R}}
\newcommand{\ze}{\zz{Z}}
\newcommand{\qe}{\zz{Q}}

\newcommand{\cL}{\mathcal{L}}
\newcommand{\PP}{\mathbb{P}}
\newcommand{\toop}{\stackrel{\PP}{\longrightarrow}}

\def\wt{\widetilde}

\newcommand{\rN}{\zz{Z}_{\geq 0}}

\graphicspath{{./figs/}}

\begin{document}
\title{Supercritical Spatial SIR Epidemics: Spreading Speed and Herd Immunity  }
\author{Xinghua Zheng\thanks{Department of ISOM, Hong Kong University of Science and Technology, Clear Water Bay, Kowloon, Hong Kong. Email: xhzheng@ust.hk}, Qingsan Zhu\thanks{HKUST Jockey Club Institute for Advanced Study, Hong Kong University of Science and Technology, Clear Water Bay, Kowloon, Hong Kong. Email: iaszhuqs@ust.hk}}

\date{\today}

\maketitle

\abstract{We study supercritical spatial SIR epidemics on $\mathbb{Z}^2\times \{1,2,\ldots, N\}$, where each site in $\mathbb{Z}^2$ represents a village and $N$ stands for the village size. We establish several key asymptotic results as $N\to\infty$. In particular, we derive the probability that the epidemic will last forever  if the epidemic is started by one infected individual. Moreover,
conditional on that the epidemic lasts forever, we show that the epidemic spreads out linearly in all directions and derive an explicit formula for the spreading speed. Furthermore, we prove that the ultimate proportion of infection converges to a number that is constant over space and find its explicit value. An important message is that if there is no vaccination, then the ultimate proportion of population who will be infected can be \emph{much higher} than the vaccination proportion that is needed in order to prevent sustained spread of the infection. }

\section{Introduction}
\subsection{SIR Model}
The Susceptible, Infected and Recovered (SIR) epidemic model is a fundamental model in epidemiology. In the usual SIR model, there is a fixed population, and at any time, individuals in the population fall in one of the three categories: susceptible, infected and recovered (\cite{KM27}). Infected individuals remain infected for one unit of time, then recover and gain immunity. The disease is spread from an infected individual to a susceptible individual with a fixed probability.

In the afore-mentioned model, spatial information is not considered. This is not suitable for various applications because many epidemics, including COVID-19, only transmit \emph{locally}.

In this paper, we focus on a spatial version of SIR model.

\subsection{Spatial SIR Model}
We consider the spatial SIR epidemic model introduced in \cite{lalley09}. The epidemic takes place on ${\zz{Z}^2}\times\{1,2,\ldots,N\}$. The space dimension can be more general, but in this paper we focus on the two dimensional case  because of its practical relevance.  In the model, each site  $\x\in\zz{Z}^2$ represents a village, which hosts $N$ individuals, which are labeled as $(\x,1), (\x,2),\ldots,(\x,N)$. The general rule is the same as the usual SIR model, except that the disease can only be spread from an infected individual to a susceptible individual who is either at the same site or in a nearest neighbor site. We assume that the infection probability is the same for all such pairs of infected and susceptible individuals, and the infection probability is given by
\begin{equation}\label{eq:infect_prob}
P_{N}^{\theta}= \frac{1+\theta}{5N}.
\end{equation}

From \eqref{eq:infect_prob}, we see that the basic reproduction number is $R_0=(5N-1)\cdot P_{N}^{\theta}$, which is approximately $1+\theta$ when the village size $N$ is large. The epidemic is hence supercritical, critical, or subcritical according to whether $\theta>0$, $\theta=0$, or $\theta <0$.

The critical case when $\theta = 0$, or more generally, the near-critical case when $\theta =\theta_0/N^{1/2}$ for some fixed constant $\theta_0$, has been studied in \cite{lz10} and \cite{LPZ14}. In the first paper, the authors prove that the process, suitably scaled, converges to a measure-valued super-process. In \cite{LPZ14}, the authors further establish a survival-extinction phase transition for the limiting process.

COVID-19, as well as many other epidemics, are supercritical. This is the case we focus on in this paper. Henceforth, we assume that $\theta$ is a fixed positive constant.

\subsection{Main Results}

We answer the following three  fundamental questions:
\begin{enumerate}
\item[Q1.] How big is the probability that the epidemic will last forever?
\item[Q2.] How fast does the epidemic spread out?
\item[Q3.] What is the ultimate proportion of individuals who will be infected?
\end{enumerate}

We start with Q1. Note that because the total population size is infinity, there \emph{is} a positive probability that the epidemic will last forever.

We first define related random variables:
\[
\left\{
\aligned
I_t(\x) &= I_t^N(\x)= \#\mbox{infected individuals at site } \x \mbox{ at time } t,\\
R_t(\x) &=R_t^N(\x) = \#\mbox{recovered individuals at site } \x \mbox{ at time } t, \mbox{and}\\
S_t(\x) &=S_t^N(\x) = \#\mbox{susceptible individuals at site } \x \mbox{ at time } t\\
& = N - I_t(\x) - R_t(\x).\\
\endaligned
\right.
\]
The evolution of the SIR process  can be described as the following:
\begin{equation}\label{eqn:X_R}
\aligned
I_{t+1}(\x)|(I_t(\x), R_t(\x)) &\stackrel{\mathrm{d}}{=} \mathrm{Bin}\left(S_t(\x),1-\left(1-P_{N}^{\theta}\right)^{\wt{I}_t(\x)}\right),\\
\mbox{ where }\wt{I}_t(\x) & = \sum_{||\y- \x||_1\leq 1} I_t(\y), \q \mbox{ and}\\
\mbox{  } R_{t+1}(\x) &= I_t(\x) + R_t(\x),\q \mbox{ for all } t\geq 0.
\endaligned
\end{equation}
Here, $\stackrel{\mathrm{d}}{=}$ means ``equal in distribution'', $\mathrm{Bin}(n,p)$ represents the binomial distribution with parameters $n$ and $p$, and $||\cdot||_1$ denotes the $\ell_1$ norm, namely, $||(a,b)||_1=|a| + |b|$ for any $(a,b)\in\zz{R}^2.$
The way we define the function $\wt{I}$ will be used throughout the rest of the paper, namely, for any function $\phi(\cdot)$ on $\zz{Z}^2$, $\wt{\phi}(\cdot)$ is a function defined as follows:
\[
  \wt{\phi}(\x)= \sum_{||\y- \x||_1\leq 1} \phi(\y),\q\mbox{for all } \x\in\zz{Z}^2.
\]

We focus on the following two initial conditions:
\begin{itemize}
\item IC1:
\begin{equation}\label{eq:IC_0}
I_0^N(\mathbf{0}) = 1, \mbox{ and } I_0^N(\x) = 0\mbox{ for all } \x\neq \mathbf{0};
\end{equation}
\item IC2: for some fixed $\gamma\in (0,1]$,
\begin{equation}\label{eq:IC_1}
I_0^N(\mathbf{0}) \sim \gamma N, \mbox{ and } I_0^N(\x) = 0\mbox{ for all } \x\neq \mathbf{0},
\end{equation}
where for any two sequences $(a_n)$ and $(b_n)$, we write $a_n\sim b_n$ if $a_n /b_n\to 1$.
\end{itemize}
{The default setting is that $R_0^N\equiv0$, unless otherwise stated.}

 IC1 corresponds to the situation where the epidemic starts with one infected individual, so-called ``patient zero''. IC2 can be considered as the situation where there is a virus outbreak, which causes an instant infection of a significant proportion of population at the  outbreak point.

We now state the result about the probability that the epidemic lasts forever. Let us recall that for a Galton-Watson process started by one particle and with offspring distribution Poisson$(1+\theta)$, its survival probability, denoted by $\iota$, is the unique solution to the following equation:
\begin{equation}\label{eq:iota}
1-\iota=\exp(-(1+\theta)\iota), \q \iota \in (0,1).
\end{equation}
More general results can be found in, e.g., \cite{athreya72}.
\begin{thm}\label{thm:surv_prob}
Let $q_N$ be the survival probability of the SIR process with village size $N$, namely, $q_N=P(I_t^N(\cdot)\not\equiv 0 \mbox{ for all } t>0)$. Then, we have
\[
       \lim_{N\rightarrow\infty}q_N=
       \left\{
       \aligned
       \iota, & \q \mbox{under the initial condition IC1};\\
       1,   &\q  \mbox{under the initial condition IC2}.
       \endaligned
       \right.
\]
\end{thm}

Theorem \ref{thm:surv_prob} states that as the village size gets larger, the survival probability of the SIR process under the initial condition IC1 approaches the survival probability of a supercritical Galton-Watson process. The intuition is the following. During the initial period, the population of people who are immune to the disease is small, hence if the village size $N$ is large, the epidemic evolves like a supercritical Galton-Watson process, causing the number of infected individuals to   blow up quickly. Afterwards, due to the fixed village size and the accumulation of people who are immune to the disease, the SIR process starts to evolve differently than the Galton-Watson process. However, because the number of infected individuals is already large, with a high probability, the epidemic will last forever. Whether the epidemic lasts forever or not hence depends mainly the initial period, during which time the process behaves similarly to a Galton-Watson process.

Next, we turn to Q2, the spreading speed of the epidemic.
Because the disease can only be spread between neighboring sites, the speed, under the $\ell_1$ norm, is at most one. Intuitively, the bigger the $\theta$, the more likely the disease spreads out, hence the higher the speed is. Our second main result states that there is a double phase transition, according to how $\theta$ compares with $1.5$ or $4$, and the speed can be strictly smaller than one in all directions, or  smaller than one in some but not all directions, or equal to one in all directions.

To give the explicit formula of the spreading speed, we need to define several functions. Let
\[
h(t)=t\log t+(1-t)\log(1-t),\q t\in [0,1],
\]
\[
a(\phi) = \frac{\min(|\sin(\phi)|,|\cos(\phi)|)}{|\sin(\phi)|+|\cos(\phi)|}, \q \phi\in[0,2\pi),
\]
and
\begin{equation}\label{eq_def_G}
\aligned
 G(v,\phi)&=\inf_{t\in[v,1]} h(t)+t\left(h\left(\frac{1}{2}-\frac{v}{2t}\right)+h\left(\frac{1}{2}-\frac{(1-2a(\phi))v}{2t}\right)\right),\\
 & \q\mbox{ for } v\in (0,1] \mbox{ and } \phi\in[0,2\pi).
 \endaligned
\end{equation}
The function $G$  is strictly decreasing in $v\in(0,1]$ with
$$
 \lim_{v\rightarrow 0}G(v,\phi)=\log \frac{1}{5},\q \mbox{and}\q G(1,\phi)=h(a(\phi)).
$$
It follows that for any $\theta\in (0,1.5)$, there exits a unique $\upsilon=\upsilon(\theta,\phi)\in(0,1)$ such that
\begin{equation}\label{eq_def_v}
    G(\upsilon,\phi)=\log \left(\frac{1+\theta}{5}\right).
\end{equation}
When $\theta\geq 1.5,$ it may occur that $h(a(\phi)) \leq \log( (1+\theta)/5)$, in which case we define
\begin{equation}\label{eq_def_v_1}
    \upsilon(\theta,\phi) = 1.
\end{equation}
Finally, we define for any $\0\neq (x,y)\in\zz{R}^2$, a function $\arg(x,y)\in[0,2\pi)$ to be the angle  from the positive real axis to the vector representing the complex number $x+\mathbf{i}y$.

\begin{thm}\label{thm:speed}
For any $\theta\in(0,\infty)$, either under the initial condition~IC1 and conditional on that the epidemic lasts forever, or under the initial condition~IC2, the following results hold:
\begin{enumerate}[(i)]
 \item for any $\eps>0$, any village size $N$, along any sequence $(i_k,j_k)_k \subset \zz{Z}^2$ satisfying $|i_k|+|j_k|\rightarrow \infty$ and $\arg(i_k,j_k)\rightarrow \phi$  for some $\phi\in[0,2\pi)$, we have
\[
 \lim_{k\to\infty}P(R_{\lfloor ({|i_k|+|j_k|})(\upsilon(\theta,\phi)^{-1}-\eps)\rfloor}(i_k,j_k)>0)=0;
\]
 \item
 For any $\eps>0$, when the village size $N$ is sufficiently large,  along any sequence $(i_k,j_k)_k \subset \zz{Z}^2$ satisfying $|i_k|+|j_k|\rightarrow \infty$ and  $\arg(i_k,j_k)\rightarrow \phi$ for some $\phi\in[0,2\pi)$, we have
\[
    \limsup_{k\to\infty}P\left({\left|\frac{R_{\lfloor ({|i_k|+|j_k|})(\upsilon(\theta,\phi)^{-1}+\eps)\rfloor}(i_k,j_k)}{N}-\iota\right|>\eps}\right)<\eps.
\]
\end{enumerate}
\end{thm}

Results like Theorem \ref{thm:speed} are known as ``Shape Theorem''. A number of such theorems have been proved for other models; see, e.g., \cite{Richardson73,DL81,CD88,ZY93,Lalley03}. However, in all these situations, the exact form of the limiting shape is unknown. To the best of our knowledge, Theorem \ref{thm:speed} is the first such result for which the exact shape is known.

Let us explain Theorem \ref{thm:speed} in words. For any direction $\phi\in[0,2\pi)$, any site~$(i,j)$ that is far away from the origin and  on this direction, namely, $(i,j)\approx (|i|+|j|)(\cos\phi,\sin\phi)$, before generation $\lfloor (|i| + |j|)(\upsilon(\theta,\phi)^{-1}-\eps)\rfloor$, with a high probability, there is no recovered individual, in other words, the epidemic has not reached the site. However, after a short period of time, at generation $\lfloor (|i| + |j|)(\upsilon(\theta,\phi)^{-1}+\eps)\rfloor$, conditional on that the epidemic lasts forever,  with a high probability, there is
roughly a constant proportion of recovered individuals. The constant proportion, as we shall see from Theorems \ref{thm:ult_inf_prop_2} and \ref{thm:ult_inf_prop_1} below, turns out to {be} the ultimate proportion of recovered individuals over the whole period. To sum up, the epidemic spreads at the speed of $\upsilon(\theta,\phi)$ along  direction~$\phi$, and most infections occur during a short period of time.

Let us visualize the statement about speed. Theorem \ref{thm:speed} asserts that  the ``frontier'' of the epidemic, after scaled by the generation, converges to the curve $(\upsilon(\theta,\phi): \phi \in[0,2\pi))$.
The following pictures plot the curve for three different values of $\theta:$~$ \theta = 1, 2$ and $5$.

\begin{figure}[H]	
\centering
\includegraphics[width=0.3\textwidth]{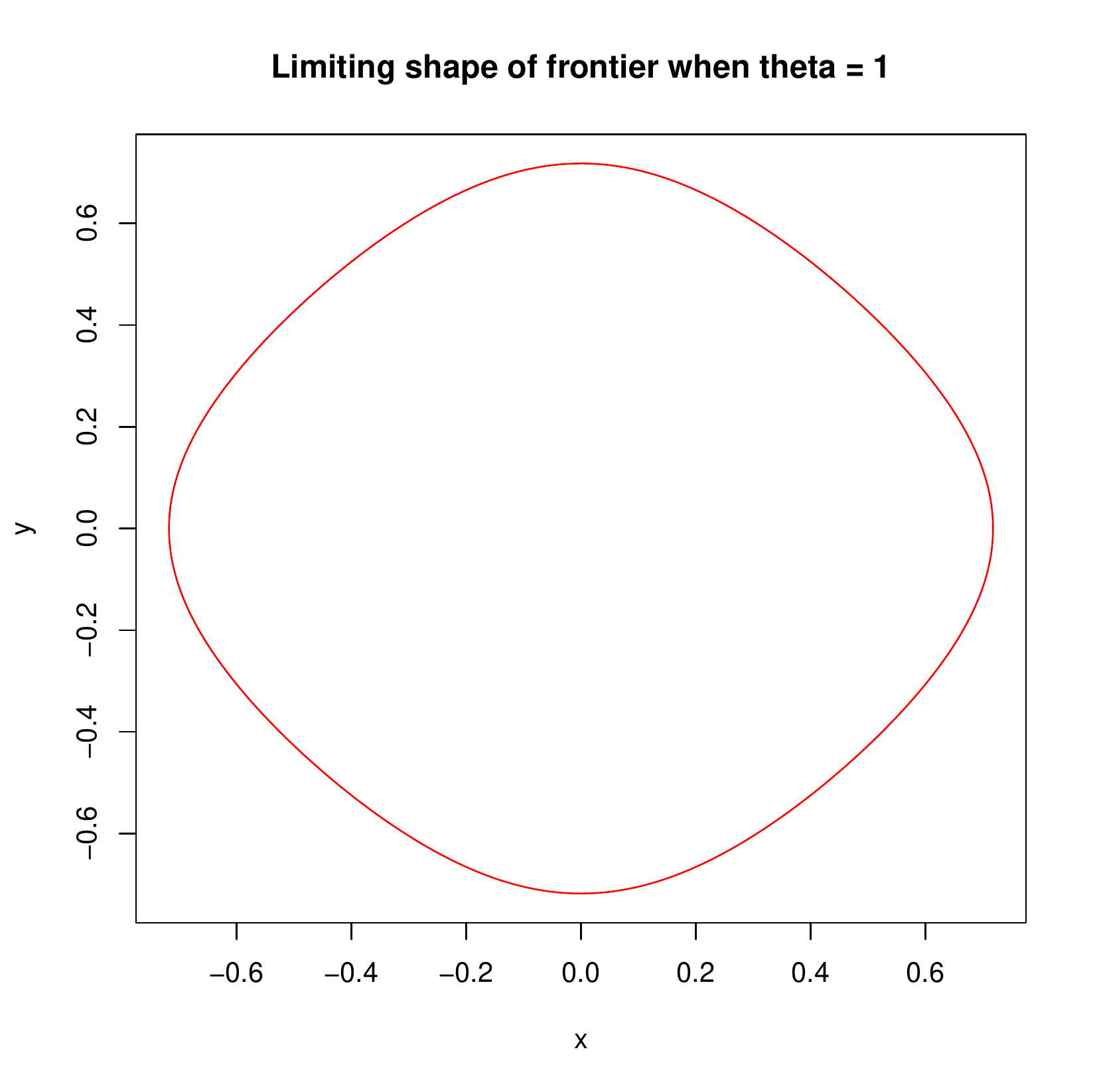}
\includegraphics[width=0.3\textwidth]{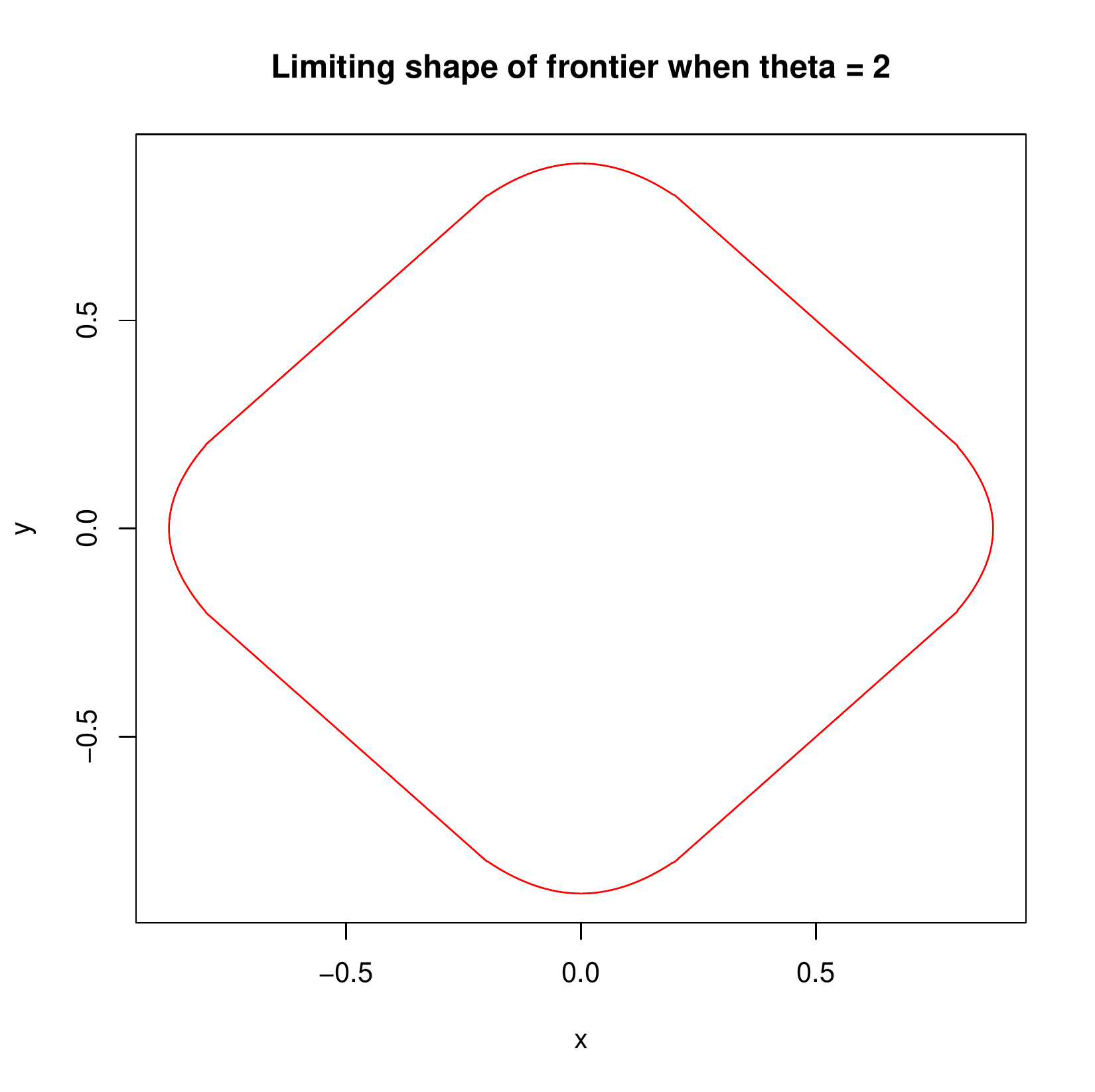}
\includegraphics[width=0.3\textwidth]{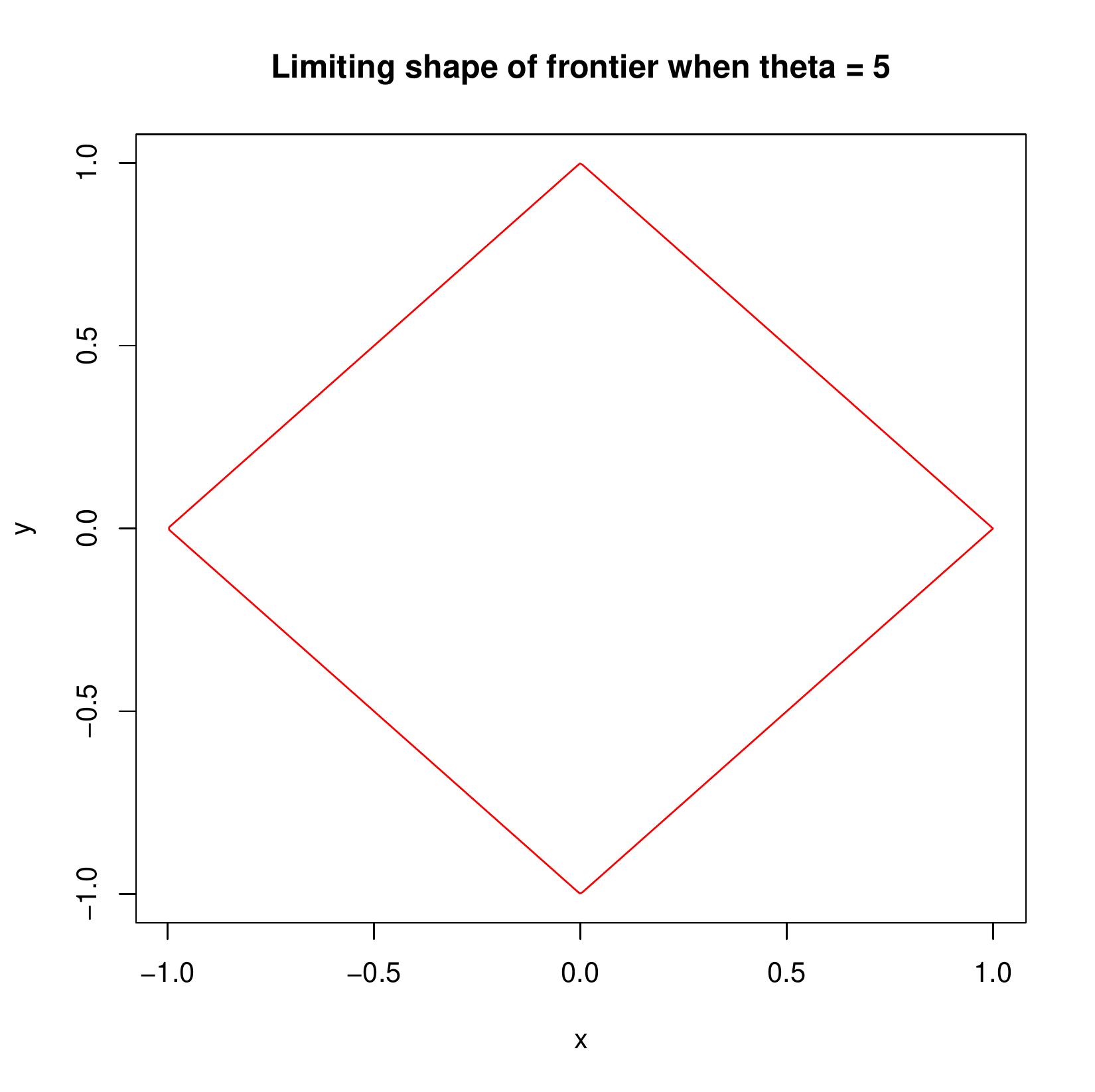}
\caption{Limiting shape of the frontier of the epidemic when $ \theta = 1, 2$ and $5$.}\label{fig:shape}
\end{figure}

We note that there is an interesting phase transition:
\begin{itemize}
  \item when $ \theta < 1.5$, the speed is smaller than one in all directions;
  \item when $ \theta \in [1.5,4)$, there are  four cones along the $x$ and $y$ axes, inside of which
the speed is smaller than one while the speed equals one outside;
 \item when $ \theta \geq 4$, the speed equals one in all directions.
\end{itemize}

To see how well the curve $(\upsilon(\theta,\phi): \phi \in[0,2\pi))$  describes the frontier, we run simulations for $\theta = 1, 2$ and $5$ with village size $N=1,000$ and up to generation $T=1,000$ under the initial condition IC1. The following pictures show the population of infected/recovered individuals.

\begin{figure}[H]	
\centering
\includegraphics[width=0.3\textwidth]{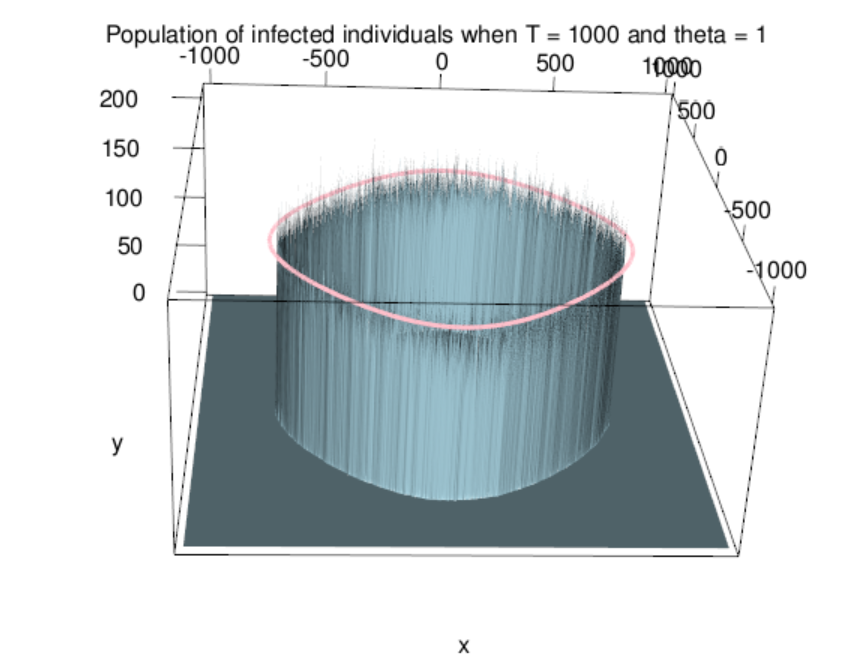}
\includegraphics[width=0.3\textwidth]{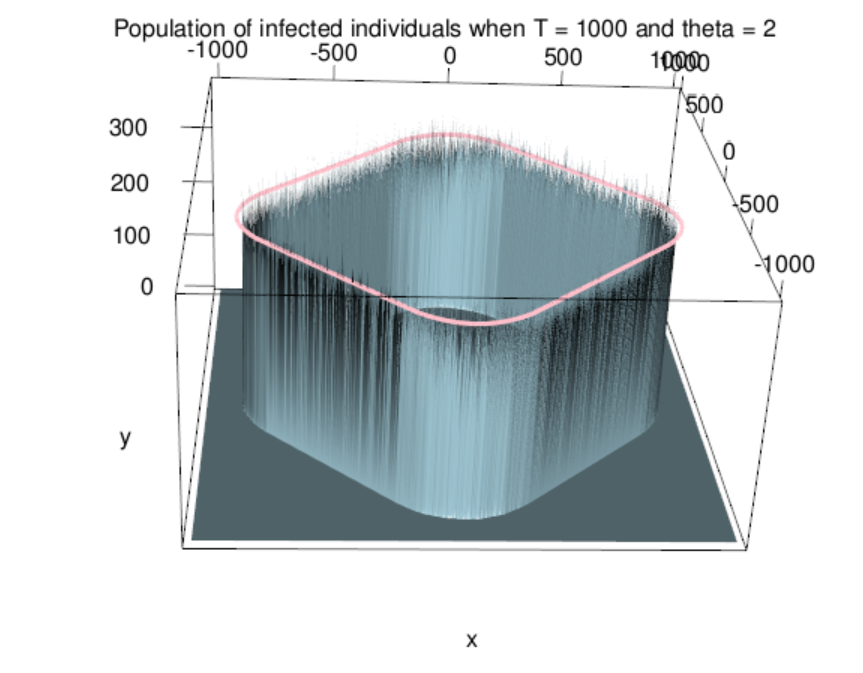}
\includegraphics[width=0.3\textwidth]{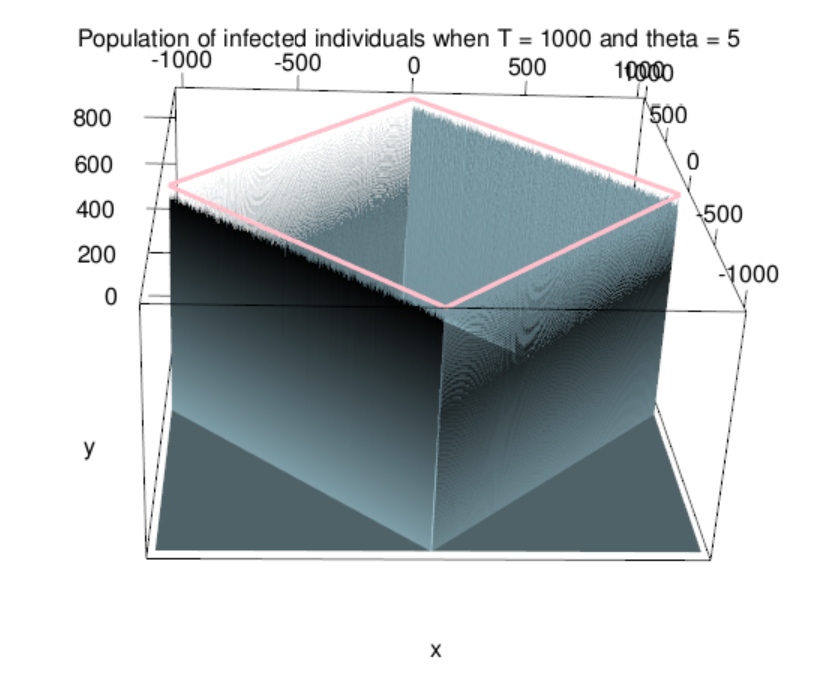}\\
\includegraphics[width=0.3\textwidth]{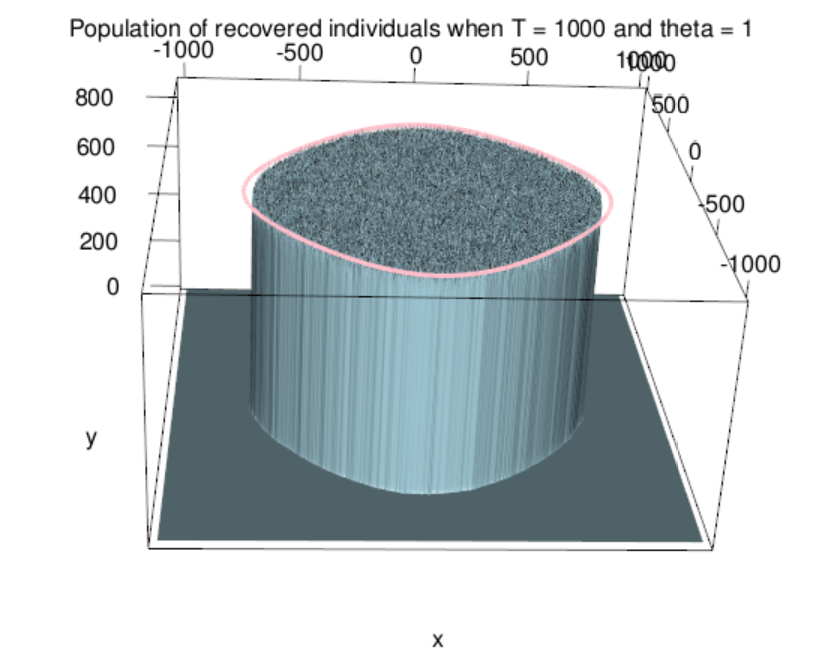}
\includegraphics[width=0.3\textwidth]{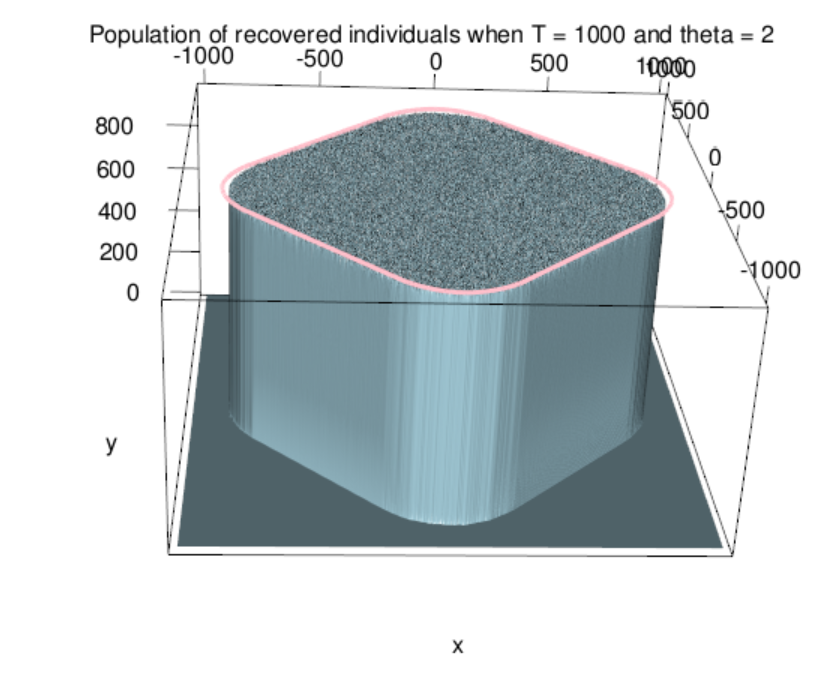}
\includegraphics[width=0.3\textwidth]{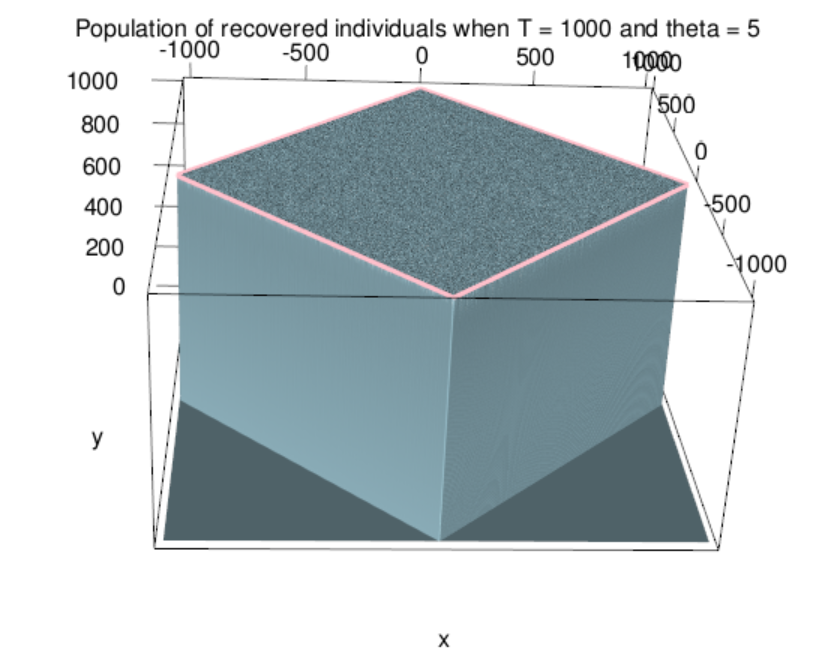}
\caption{Population of infected/recovered individuals from random simulations when $\theta = 1, 2$ and $5$ with village size $N=1,000$ and at generation $T=1,000$ under the initial condition  IC1. Above: infected population; bottom: recovered population. The pink curve is $(T\upsilon(\theta,\phi): \phi \in[0,2\pi))$.}\label{fig:shape_sim}
\end{figure}

We see that the curves provide remarkably accurate description of the frontier in all three cases.

Now we turn to Q3, the ultimate proportion of individuals who will be infected. In fact, the three pictures in the bottom panel of Figure~\ref{fig:shape_sim} already shed some light on the proportion. We see that the proportion is almost constant behind the frontier. Our next result confirms this observation and gives the explicit value for the constant. Recall that $\iota$ is defined in \eqref{eq:iota}. Note that for any fixed~$N$ and $\x$, $R^N_n(\x)$ is increasing in $n$, so we can define $R^N_\infty(\x)=\lim_{n\to\infty} R^N_n(\x)$.

\begin{thm}\label{thm:ult_inf_prop_2}
For any $\eps>0$, there exists  $N_0\in\zz{N}$ such that for all~$N\geq N_0$, under the initial condition~IC1, conditional on that the epidemic lasts forever, we have
    \begin{equation}\label{eq:R_final_IC1}
    \sup_{\x\in\ze^2} P\left(\left|\frac{R^N_\infty(\x)}{N}-\iota\right|>\eps\right)<\eps.
    \end{equation}
\end{thm}

Under the initial condition IC2, the ultimate proportion of infection  is not constant over the space, rather, it is specified by a difference equation. More precisely, let $\DR_\infty(\x):\ze^2\rightarrow [0,1]$ be the solution to the following difference equation:
\begin{equation}\label{eq:ult_prop_IC_2}
   {\aligned
       f(\x)&=1-\exp(-\frac{1+\theta}{5}\wt{f}(\x)),\, \mbox{ for all } \x\in \ze^2\setminus \{\0\}, \mbox{ and }\\
       f(\0)&=\gamma +(1-\gamma)\left(1-\exp(-\frac{1+\theta}{5}\wt{f}(\0))\right).
   \endaligned
   }
\end{equation}
We will show in Proposition~\ref{prop:uniq_sol_eq} that there is a unique solution to the above equation. Moreover, the function is symmetric and satisfies that $\DR_\infty(\x)\searrow\iota$ in the sense that $\DR_\infty(\x)\geq\max\{\DR_\infty(\x+(1,0)),\DR_\infty(\x+(0,1))\}$ for all $\x\in\rN\times \rN$, and $\DR_\infty(\x)\to\iota$ as $||\x||_1\rightarrow\infty$.

\begin{thm}\label{thm:ult_inf_prop_1}
Under the initial condition~IC2, we have
\[
   \frac{R^N_\infty(\x)}{N}\toop  \DR_\infty(\x),  \mbox{ for all } \x\in \ze^2 \mathrm{~as~} N\rightarrow \infty.
\]
Moreover, for any $\vep>0$, when $N$ is sufficiently large, we have
\[
  \limsup_{\x\rightarrow\infty}P\left(\left|\frac{R^N_\infty(\x)}{N}-\iota\right|>\vep\right)<\vep.
\]
\end{thm}

Theorems \ref{thm:ult_inf_prop_2} and  \ref{thm:ult_inf_prop_1} have a significant implication on  ``herd immunity''.  In general, for an epidemic with a basic reproduction number $R_0$, the minimum proportion of the population that needs to be  vaccinated  to prevent sustained spread of the infection is $1-1/R_0$. In our case, $R_0\sim 1 +\theta$, hence the required vaccination proportion is $\sim 1-1/(1+\theta)$. Theorem \ref{thm:ult_inf_prop_2}, however, states that under the initial condition~IC1, if there is no vaccination (hence the only way for an individual to gain immunity is via recovery from infection), then on the event that the the epidemic lasts forever, the proportion of individuals who will  be infected throughout the whole process is around~$\iota$. Under the  initial condition~IC2, the proportion of individuals who will  be infected on site far away from the origin is also around~$\iota$. A natural question is:
How do the two proportions compare with each other? The comparison is illustrated in Figure~\ref{fig:iota_vs_NC}.

\begin{figure}[H]	
\centering
\includegraphics[width=0.6\textwidth]{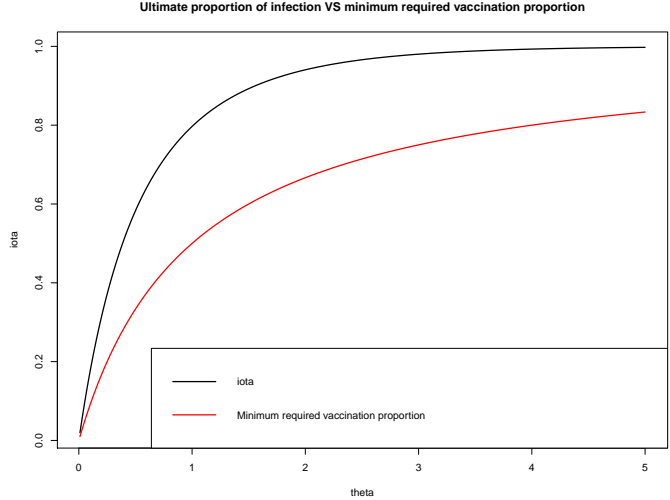}
\caption{Comparison between the ultimate proportion of infection and the minimum required vaccination proportion for different values of $\theta$.}\label{fig:iota_vs_NC}
\end{figure}

Figure \ref{fig:iota_vs_NC} reveals a sharp difference between ``natural immunization'', namely, recovery from infection, and vaccination. Via  natural immunization, the ultimate proportion of population who will be infected can be \emph{much higher} than the required vaccination proportion to prevent sustained spread of the infection. For an epidemic with a basic reproduction number of 3, the two proportions are 0.94 and~$2/3$, respectively. In other words, without vaccination, ultimately, there will be $94\%$ of population who will be infected! This is  substantially higher than the required vaccination level of $2/3.$ The difference between the two proportions highlights the huge benefit of vaccination.

Finally, we give a theorem to describe how  infection evolves inside the ``speed one'' cone.
Recall that by Theorem \ref{thm:speed}, in order to have such a cone,  $\theta$ needs to be larger than $1.5$. When $\theta\in(1.5,4)$, let $\kappa$ be  the unique solution in $(0,0.5]$ to
\begin{equation} \label{eq:kappa}
\kappa^\kappa(1-\kappa)^{1-\kappa}=\frac{1+\theta}{5}.
\end{equation}
When $\theta\geq4$, let $\kappa=0$. Also note that
$\upsilon(\theta,\phi)= 1$ if and only if $\kappa\leq|\cos(\phi)|/(|\cos(\phi)|+|\sin(\phi)|)\leq 1-\kappa$.

\begin{remark}\label{rmk_kappa}
The value $\kappa$ equals to the velocity of the leftmost position of a supercritical branching random walk. Specifically, consider a branching random walk started by one particle at the origin. At every generation, each particle produces a random number of offspring according to $\mathrm{Bin}(2N,(1+\theta)/(5N))$  and dies out; each offspring independently moves to the upper neighbour or the right neighbour with equal probability. It is easy to see that such defined process stochastically dominates $(I_n^N(m,n-m))_{m\in\zz{Z}}$. Write $(m(n), n-m(n))$ as the leftmost position of the branching random walk in the $n$-th generation. Then, the classical results in branching random walk (see, e.g., \cite{Hammersley74}, \cite{Kingman75}, \cite{Biggins76} and \cite{Bramson78}) state that, conditional on the event that the branching random walk survives forever, one has
$$
\frac{m(n)}{n}\rightarrow \kappa.
$$
\end{remark}

Before giving our next theorem, we define  a sequence  $\{\ell^{(k)}\}_{k\geq -1}$ recursively as follows:  $\ell^{(-1)}=\ell^{(0)}=0,$ and for $i\geq 1$,  $\ell^{(i)}$ is the unique solution in $(0,1)$ to
\begin{equation}\label{eq:ell_i}
\ell^{(i)} = \left(1-\sum_{1\leq j<i} \ell^{(j)}\right) \left(1-\exp\left(-\frac{1+\theta}{5}\cdot(2 \ell^{(i)} + \ell^{(i-1)}+ 2\ell^{(i-2)})\right)\right).
\end{equation}
We will show in
Lemma~\ref{lm:sum_ell} that $\iota=\sum_{i=1}^{\infty}\ell^{(i)}$ when $\theta>1.5$. Further define a sequence of random variables $\{\K(n)=\K^N(n)\}$ as
\[
\K(n)=\K^N(n)=\inf\left\{k:\sum_{(m,l)\in\ze^2:m+l=n}I^N_{n+k}(m,l)>0\right\},
\]
with the convention that $\inf\emptyset=\infty$. In words, $\K(n)$ measures after how many generations the epidemic reaches the line $\{(m,l)|m+l=n\}$. It is easy to see that $\K^N(n)$ is increasing, and we can define
\[
\K^N=\lim_{n\rightarrow \infty}\K^N(n).
\]

\begin{thm}\label{thm:all_layer_SIR}
Suppose that $\theta>1.5$, then the following results hold:
\begin{itemize}
  \item Under the initial condition~IC2,  for any $\vep>0$ and $i\geq 0$, when $N$ is sufficiently large, we have
\[
     \limsup_{n\rightarrow\infty}\sup_{m:(\kappa+\eps)n<m<(1-\kappa-\eps)n}P\left(\left|\frac{I_{n+i}^N(m,n-m)}{N}-\ell^{(i+1)}\right|>\eps\right)<\eps.
\]
    \item Under the initial condition~IC1, we have
\begin{enumerate}[(i)]
    \item For any $i\geq 0$,
      \[
         \lim_{N\rightarrow \infty } P(\K^N=i)=\ell^{(i+1)}.
      \]

    \item For any $i,j\geq 0$, $\vep>0$, when $N$ is sufficiently large, we have
       \begin{multline} \nonumber  
         \limsup_{n\rightarrow \infty} \sup_{m:(\kappa+\vep)n<m<(1-\kappa-\vep)n}\\
         P\left(\left|\frac{I^N_{n+j+i}(m,n-m)}{N}-\ell^{(i+1)}\right|>\vep\ |\ {\K^N}=j\right)<\vep.
        \end{multline}
\end{enumerate}
\end{itemize}
\end{thm}

\begin{remark}\label{rmk:frontier}
Theorem \ref{thm:all_layer_SIR} is stated for the first quadrant. Analogous results hold for the three other quadrants.
\end{remark}

Let us give some explanations to the above theorem. Under the initial condition IC2, when $N$ is sufficiently large, inside the ``speed one'' cone, Theorem~\ref{thm:all_layer_SIR} asserts that the epidemic spreads out exactly at speed one in the sense that as the generation $n$ increases, with a high probability, the epidemic would reach the diamond $\{(x,y):|x|+|y|=n\}$, and the infection proportion at sites on the diamond inside the cone is approximately $\ell^{(1)}$. Moreover, for any $i\geq 1$, at $i$ generations afterwards, the infection proportion would be approximately~$\ell^{(i+1)}$.  Under the initial condition IC1, similar conclusions hold except that  the frontier of the epidemic in the first quadrant is  $\K(n)$ layers behind  the line $\{(x,y):x+y=n\}$. Analogous results hold for other three quadrants.

Figure \ref{fig:ell_vs_sim} visualizes Theorem \ref{thm:all_layer_SIR}. We take $\theta=2$ with village size $N=1,000$ and run the process up to generation $T=1,000$ under IC1. The left picture in Figure \ref{fig:ell_vs_sim} plots the distribution of infected population at time $T=1,000$. We then compute the infection proportions on sites along the direction $y=x$  near the frontier. The results are shown as the black curve in the right picture in Figure \ref{fig:ell_vs_sim}. The red curve, on the other hand, plots the values of $\{\ell^{(i)}\}$. We can see that $\{\ell^{(i)}\}$ approximates the (random)  infection proportions well.

\begin{figure}[H]	
\centering
\includegraphics[width=0.45\textwidth]{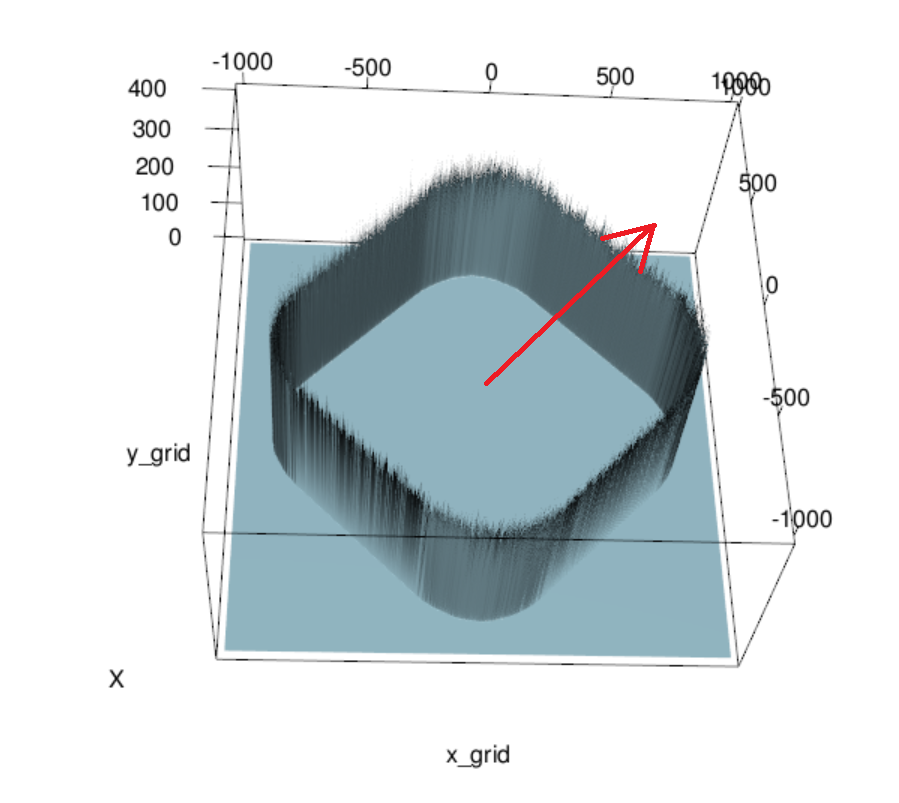}\includegraphics[width=0.45\textwidth] {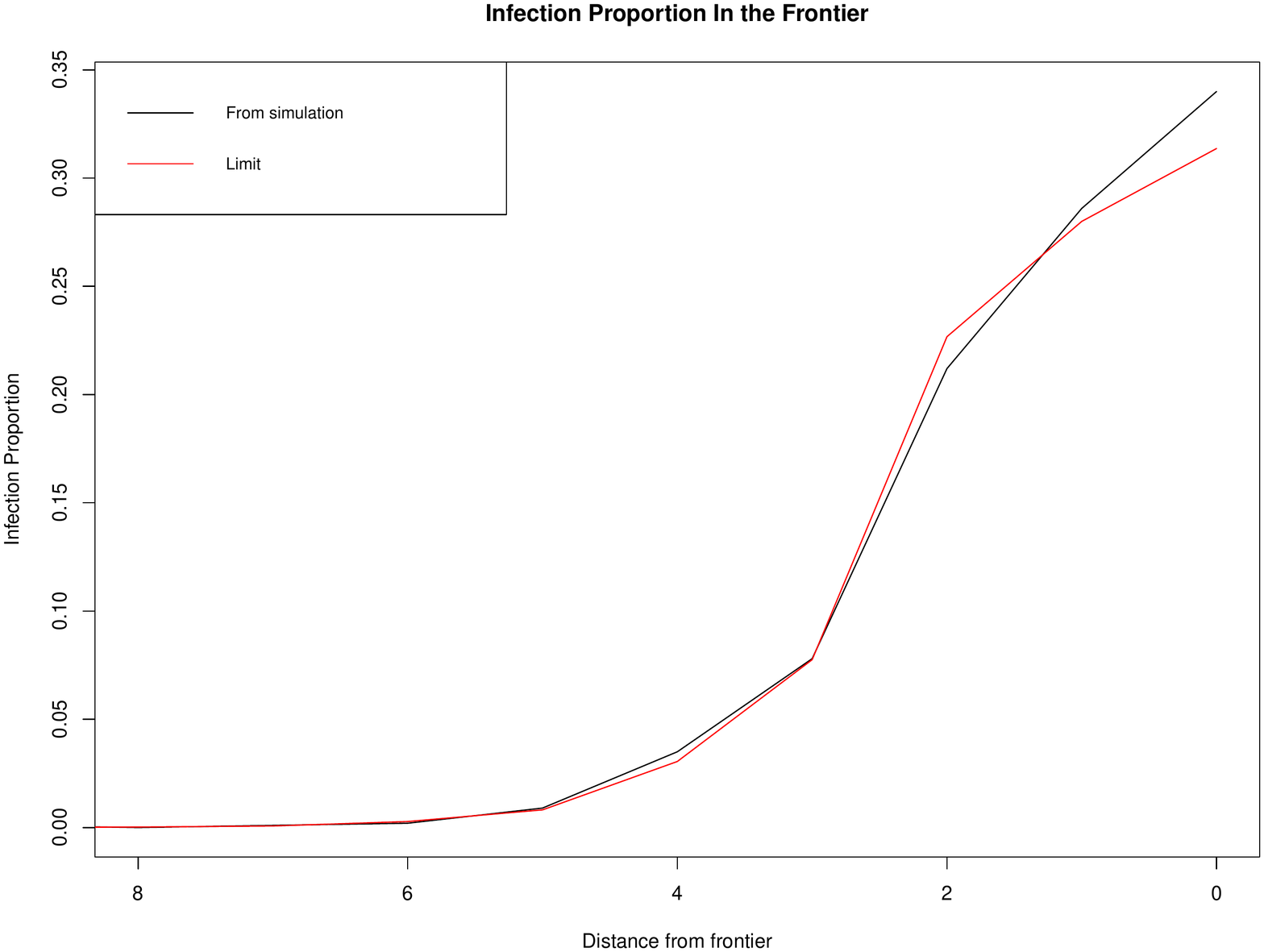}
\caption{Comparison between  $\{\ell^{(i)}\}$ and the infection proportions near the frontier from a random simulation when $\theta = 2$ with village size $N=1,000$ and at generation $T=1,000$ under the initial condition IC1. The random variable $\K^N(998)=2$ for this simulation.}\label{fig:ell_vs_sim}
\end{figure}

\subsection{Connection to percolation}\label{ssec:conn_perc}
Let us recall a connection  between our SIR process and a percolation model (\cite{lalley09}). We will call the percolation the $N$-percolation, which is defined as follows. Consider the graph with vertex set $\ze^2\times [N]$, where we write $[N]$ for $\{1,2\dots,N\}$, and edge set $\{(\x,i)-(\y,j): ||\x-\y||_1\in\{0,1\}\}$.  Each edge is open with probability $P_N^\theta$. For any initial condition $(I^N_0(\cdot), R^N_0(\cdot)\equiv0)$, let $A$ be any subset of $\ze^2\times [N]$ such that $\# (A\cap \x\times [N])=I^N_0(\x)$ for any $\x\in\ze^2$. Then, $I^N_n(\x)$ ($R^N_n(\x)$, respectively) equals the number of vertices at site $\x$ that are of graph distance $n$ to $A$ (of graph distance $\leq n-1$, respectively). Moreover, $R^N_\infty(\x)=\lim_{n\rightarrow \infty}R^N_n(\x)$ is just the number of  vertices at site $\x$ that are connected to $A$.  From this, we see that $R^N_t$(but not $I^N_t$) is increasing in the initial condition $I_0^N$. Based on such a connection, Theorems \ref{thm:surv_prob}-\ref{thm:all_layer_SIR} can be interpreted correspondingly for the percolation model.

\subsection{Generalization to other dimensions and distributions}
Our methods of proof are robust enough to extend to other dimensions and infection distributions. We state the assumptions and results below without giving detailed proofs.

Suppose that the dimension $d\geq 2$, and $p=\{p(\x)\}_{\x\in\ze^d}$ is a probability measure on $\ze^d$ that satisfies the following regularity assumptions:
\begin{enumerate}[(i)]
    \item $p$ has a finite support, namely, $\mathrm{Supp}(p):=\{\x:p(\x)>0\}$ is finite;
    \item $p$ is symmetric, i.e., $p(\x)=p(-\x)$ for all $\x$;
    \item $p$ is irreducible, i.e., $\mathrm{Supp}(p)$ is not contained in any strict subgroup of $\ze^d$;
    \item $p$ is aperiodic, i.e., there exist an odd $n$ and $\x_1,\dots,\x_n\in\ze^d$ such that $\sum_{i=1}^n\x_i=\0$ and $\prod_{i=1}^n p(\x_i)\neq 0$.
\end{enumerate}
The general SIR process on $\ze^d\times [N]$ evolves as follows:
\begin{equation*}
\aligned
I_{t+1}(\x)|(I_t(\x), R_t(\x)) &\stackrel{\mathrm{d}}{=} \mathrm{Bin}\left(S_t(\x),1-\prod_{\y\in\mathrm{Supp}(p)}\left(1-(1+\theta)\frac{p(\y)}{N}\right)^{I_t(\x+\y)}\right),\\
\mbox{  } R_{t+1}(\x) &= I_t(\x) + R_t(\x),\q \mbox{ for all } t\geq 0.
\endaligned
\end{equation*}
The initial conditions are still IC1 and IC2.

Under the above setting, Theorems \ref{thm:surv_prob} and \ref{thm:ult_inf_prop_2} still hold. Theorem \ref{thm:ult_inf_prop_1} hold as well except that $\DR_\infty(\cdot)$ is defined via the following difference equation:
\begin{equation*}
   {\aligned
       f(\x)&=1-\exp(-(1+\theta)\mathcal{A}f(\x)),\, \mbox{ for all } \x\in \ze^d\setminus \{\0\}, \mbox{ and }\\
       f(\0)&=\gamma+(1-\gamma)\left(1-\exp(-(1+\theta)\mathcal{A}f(\0))\right),
   \endaligned
   }
\end{equation*}
where $\mathcal{A}$ stands for the Markov operator with respect to $p$, namely, $\mathcal{A}f(\x)=\sum_{\y\in\mathrm{Supp}(p)}p(\y)f(\x+\y)$.

As to the spreading speed, we need to define several functions.
Write $\S1$ and $\SN1$ for the unit real and rational spheres:
\[
\S1:=\{\x\in \re^d|\,||\x||_1=1\},\quad
\SN1:=\S1\cap \{(x_1,\dots,x_d)\,|\, x_i\in \qe,\ i=1,\ldots,d \},
\]
where we write $\qe$ for the set of rational numbers. For any $\x\in \ze^d\setminus \{\0\},u\in (0,+\infty)\cap \qe$, for all $n\geq 1$ with $nu\in\ze$, define
\[
b_n=b_n(\x,u)=P(S_{n||\x||_1}=nu\x),
\]
where $(S_m)$ is a random walk starting at $\0$ with jump distribution $p$. Clearly, $b_{n+m}\geq b_nb_m$, hence the sequence $(-\log b_n)$ is subadditive.
Therefore, by Fekete's Lemma,  $b_n\approx \varrho^n$ for some $\varrho=\varrho(\x,u)\in[0,1]$. Note that $\varrho(\x,u)=\varrho(q\x,u)$ for any $q\in\qe\setminus \{0\}$ and for fixed $\x$, $\varrho(\x,u)$ is strictly decreasing in $u$ for $u\in\qe\cap (0,\sup\{u:b(\x,u)>0\})$.
{Therefore, $\varrho$ can be extended to be a function in $\SN1\times ((0,\infty)\cap \qe)$. }
For any fixed $\theta>0$ and $\x\in\SN1$, define
\begin{equation*}
\upsilon(\x)=\upsilon(\theta,\x):=\inf\left\{u:\varrho(\x,u)\leq \frac{1}{1+\theta}\right\}.
\end{equation*}
It can be easily verified that $\upsilon(\cdot)$ is a concave function on $\SN1$ and therefore can be extended to be a continuous concave function in $\S1$. {Moreover, it can be shown that $\upsilon(\theta,\x)>0$ for all $\theta>0$ and $\x\in\S1$.}

We now state the generalization of Theorem~\ref{thm:speed}.
\begin{thm}
For any $\theta\in(0,\infty)$, either under the initial condition~IC1 conditional on that the epidemic lasts forever, or under the initial condition~IC2, the following results hold:
\begin{enumerate}[(i)]
 \item for any $\eps>0$, any village size $N$, along any sequence $\x_k \subset \zz{Z}^d$ satisfying $||\x_k||_1\rightarrow \infty$ and $\x_k/||\x_k||_1\rightarrow \y$ for some $\y\in\S1$, we have
 \[
 \lim_{k\to\infty}P(R_{\lfloor ||\x_k||_1(\upsilon(\theta,\y)^{-1}-\eps)\rfloor}(\x_k)>0)=0;
 \]
 \item For any $\eps>0$, when the village size $N$ is sufficiently large, along any sequence $\x_k \subset \zz{Z}^d$ satisfying $||\x_k||_1\rightarrow \infty$ and $\x_k/||\x_k||_1\rightarrow \y$ for some $\y\in \S1$, we have

 \[
 \limsup_{k\to\infty}P\left({\left|\frac{R_{\lfloor ||\x_k||_1(\upsilon(\theta,\y)^{-1}+\eps)\rfloor}(\x_k)}{N}-\iota\right|>\eps}\right)<\eps.
\]
\end{enumerate}
\end{thm}

\subsection{Plan of the paper}
The remainder of the paper is organized  as follows.
In Section~\ref{sec:D_process}, we introduce a deterministic process, which arises as the limiting process of our SIR process as the village size $N\to\infty$.  In Section~\ref{Sec:Front_det}, we focus on the deterministic process and prove the analogous version of Theorem~\ref{thm:all_layer_SIR} for this deterministic process. In Section~\ref{Sec:Front_SIR},  we prove Theorem~\ref{thm:all_layer_SIR}. {In Section~\ref{Sec:spread_speed}, we study the spreading speed for both the deterministic process and the SIR process, and prove Theorem~\ref{thm:speed}(i) and a weak version of Theorem~\ref{thm:speed}(ii), see Theorem~\ref{thm_speed}.} In Section~\ref{Sec:prop_recover}, we consider the ultimate proportion of recovered particles and prove Theorems \ref{thm:surv_prob}, \ref{thm:ult_inf_prop_2}, \ref{thm:ult_inf_prop_1} and Theorem~\ref{thm:speed}(ii).

\section{The Deterministic Large $N$ Limit Process}\label{sec:D_process}
It is well known that the mean-field supercritical SIR model admits a deterministic limit process as the population size gets to infinity. Similarly, our spatial SIR model admits a deterministic large $N$ limit process. It will play a fundamental role in the analysis of the original SIR process.

The deterministic process $(\DS_t(\x),\DI_t(\x),\DR_t(\x))_{\{t\in \rN,\x\in \ze^2\}}$ is defined as follows.  For any  initial condition $\DI_0, \DR_0\in [0,1]^{\ze^2}$ with $\DS_0(\x)=1-\DI_0(\x)-\DR_0(\x)\geq~0$ for all $\x\in \ze^2$, the process evolves as
\begin{equation}\label{eq:lim_X}
\aligned
\DI_{t+1}(\x) & =\DS_t(\x)\left(1-\exp\left(-\frac{1+\theta}{5}\wt{\DI}_t(\x)\right)\right), \mbox{ where }\\
\DS_{t}(\x) &=1-\DI_{t}(\x)-\DR_{t}(\x)\q \mbox{ and}\\
\DR_{t+1}(\x) &= \DI_t(\x) + \DR_t(\x), \q \mbox{ for all } t\geq 0 \mbox{ and } \x\in \ze^2.
\endaligned
\end{equation}

The following proposition states that $(\DS_t(\x),\DI_t(\x),\DR_t(\x))$ is the scaling limit of the original SIR process as $N\to\infty$. Note that since $\DS+\DI+\DR\equiv 1$, we can drop $\DS$ and only consider  $(\DI,\DR)$.
\begin{prop}\label{prop:lim_N} Suppose that the initial conditions $(I_0^N(\x),R_0^N(\x))$ and $(\DI_0(\x),\DR_0(\x))$ satisfy
\begin{equation}\label{eq:ini_conv}
  \frac{I_0^N(\x)}{N}\rightarrow \DI_0(\x),  \frac{R_0^N(\x)}{N}\rightarrow \DR_0(\x), \mbox{ for all }\x\in \ze^2 \q \mbox{as }N\rightarrow \infty.
\end{equation}
Then the process $(I_t^N(\x)/N, R_t^N(\x)/N)$ converges, under finite dimensional distribution, to the process $(\DI_t(\x), \DR_t(\x))$.
\end{prop}
\begin{proof}
By the Markovian property, to prove finite dimensional convergence, it suffices to prove the convergence for {$t=1$}, which follows from \eqref{eqn:X_R} and the law of large numbers.
\end{proof}
\begin{remark}
The initial condition IC1, after scaling, converges to the trivial initial condition $\DI_0\equiv 0$. For this reason, it requires additional techniques to study the asymptotics under the initial condition IC1.
\end{remark}

\section{Distribution Near the Frontier: the Deterministic Case}\label{Sec:Front_det}

In this section, we derive the distribution on the frontier for the deterministic limit process. We only focus on the first quadrant; identical arguments apply to the other three quadrants.

We assume that  that $\DR_0\equiv0$ and $\DI_0$ is supported by the line $x+y=0$. Hence, the frontier at time $n$ is on the lines $x+y=\pm n$.
We remark that the frontier process is monotone, i.e., when the initial condition $\{\DI_0(-m,m)\}_{\{m\in \ze\}}$ gets larger, the frontier $\{\DI_n(n-m,m)\}_{\{m\in \ze\}}$ also becomes larger. This motivates us to consider the following two extreme initial conditions,
which serve as the upper and lower bounds: for some  $\gamma\in (0,1]$,
\begin{equation}\label{eq:det_ini_Upper}
\DI_0(m,-m)=\gamma \q\mbox{for all } m \in\zz{Z},\mbox{ and } \DI_0(\x)=0 \mbox{ otherwise, }
\end{equation}
or
\begin{equation} \label{eq:det_ini_Lower}
\DI_0(0,0)=\gamma, \mbox{ and } \DI_0(\x)=0 \mbox{ otherwise. }
\end{equation}

For notational ease, we write
\begin{equation}\label{eq:value_alpha}
\alpha=\exp((1+\theta)/5),
\end{equation}
and
\begin{equation}\label{eq:psi_function}
\psi(x)=1-\alpha^{-x}, \q\mbox{for } x\in [0,1].
\end{equation}

This function plays a fundamental role in the analysis below. Note that it is  increasing and concave . Moreover, when $\alpha\leq e$, namely, when $\theta\leq 4$,  $\psi(x)<x$ for $x\in(0,1]$ and it has no fixed point in $(0,1]$. On the other hand, when $\alpha >e$, namely, when $\theta > 4$, $\psi(\cdot)$ has a unique fixed point in $(0,1)$, denoted by {$\gamma_1$},  to the left (right, respectively) of which $\psi(x)>x$ ($\psi(x)<x$, respectively).

\subsection{Convergence on the frontier}\label{ssec:1st_det}

\begin{lemma}\label{lem:conv_front_const_IC}
Under the initial condition \eqref{eq:det_ini_Upper}, for all $n$, $\DI_n(n-k,k)$ is constant in $k$. Moreover, as $n\to\infty$, the constant converges to 0 when $\theta\leq 1.5$ and to the unique positive solution $\ell=\ell(\theta)=\ell^{(1)}$ to
\begin{equation}\label{eq:ell}
 \ell=1-\exp\left(-\frac{2\ell(1+\theta)}{5}\right)
\end{equation} when $\theta > 1.5$.
\end{lemma}

\begin{proof}
By symmetry, $\DI_n(n-k,k)$ is constant in $k$. Denote the constant value by $y_n$. By \eqref{eq:lim_X}, we have
\begin{equation}\label{eq:y_1}
y_0=\gamma,\quad y_{n+1}=1-\alpha^{-2y_n}.
\end{equation}

When $\theta\leq 1.5$, it is easy to see that $y_n$ is decreasing in $n$, and the limit will solve \eqref{eq:ell}. The only nonnegative solution to \eqref{eq:ell} is zero, so $\lim_{n\to\infty} y_n=0.$

On the other hand, when $\theta>1.5$, $y_n$ is
increasing when $\gamma\leq \ell(\theta)$ and
decreasing otherwise, so the limit again exists and solves \eqref{eq:ell}. The conclusion follows.
\end{proof}

Now we consider the initial condition {\eqref{eq:det_ini_Lower}}. In this case, there are two phase transitions: one at $\theta=1.5$ and the other at $\theta=4$. Precisely, we have

\begin{thm}\label{thm_LIM_front}
Under the initial condition {\eqref{eq:det_ini_Lower}}, the process $\{\DI_n(m,n-m)\}$ behaves differently depending on the value of $\theta$.
\begin{enumerate}[(i)]
    \item When $\theta\leq 1.5$, the process converges to zero uniformly:
    \[
    \lim_{n\rightarrow\infty}\sup_{m\in\{0,\dots,n\}}\DI_n(m,n-m)=0;
    \]
    \item When $\theta\in(1.5,4]$, there is a cone along the diagonal $y=x$ such that inside the cone the frontier converges to $\ell(\theta)$ while outside the cone the process converges to zero. Moreover, as $\theta$ increases, the cone grows continuously from the diagonal line to the first quadrant.
    \item when $\theta\in(4,\infty)$, the process restricted to the first quadrant converges to $\ell(\theta)$ uniformly:
    \[
    \lim_{n,k\rightarrow\infty:k\leq n/2}\sup_{m:k\leq m\leq  n-k}|\DI_n(m,n-m)-\ell(\theta)|=0.
    \]
\end{enumerate}
\end{thm}

Note that the result for the case when $\theta\in(0,1.5]$ is implied by Lemma~\ref{lem:conv_front_const_IC} and the monotonicity of  $\{\DI_n(m,n-m)\}$ with respect to the initial condition.

We first address the case when  $\theta\in(4,\infty)$, which is relatively easier. We will show a little bit more than the content of the above theorem. Recall that~$\gamma_1$ represents the unique fixed point of $\psi(\cdot)$ {defined in \eqref{eq:psi_function}}.

\begin{prop}\label{prop:conv_front_origin_IC}  Let $y_n(k)=\DI_n(n-k,k)$. When $\theta>4$, under the initial condition {\eqref{eq:det_ini_Lower}},
if $\gamma\leq \gamma_1$, then for any $k\in \rN$, we have
\begin{equation}\label{conv-lim}
\lim_{n\rightarrow\infty}y_n(k)=\ell_k,
\end{equation}
where $\ell_k$ is determined recursively by
\begin{equation}\label{eq:ell_delta_ini}
    \ell_0=\gamma_1,\quad \ell_{k+1}=1-\alpha^{-(\ell_{k}+\ell_{k+1})}.
\end{equation}
Moreover, we have
\begin{equation}\label{cts-front}
\lim_{\min\{m,n\}\rightarrow\infty}\DI_{m+n}(m,n)=\ell(\theta).
\end{equation}
\end{prop}
\begin{remark}
Note that when $\theta>4$ and $a\geq0$, the equation $x=1-\alpha^{-(x+a)}$ has a unique solution in $(0,1)$.
\end{remark}

\begin{proof}
Write
\[
Y_n=(y_n(0),y_n(1),\ldots).
\]
Then $Y_0=({\gamma},0,\dots,)$ and
\begin{equation}\label{eq_Y}
 y_{n+1}(k)=1-\alpha^{-(y_{n}(k)+y_n(k-1))}.
\end{equation}

Because $\gamma\leq \gamma_1$, we have $y_1(0)\geq y_0(0)$. Hence, $Y_1\geq Y_0$. By monotonicity  and induction, we get
$$
Y_{n+1}\geq Y_n.
$$
Hence, we can define
$$
 Y_\infty:=\lim_{n\rightarrow \infty}Y_n=:(\ell_0,\ell_1,\ldots).
$$
Letting $n\rightarrow\infty$ in \eqref{eq_Y}, we get that $\ell_n$ satisfies the recursive equation \eqref{eq:ell_delta_ini}. Note that $\ell_n$  increases to $\ell(\theta)$.

On the other hand, by induction, one can show  that $\DI_n(i,n-i)$ is increasing for $i\in[0,\dots,\lfloor n/2 \rfloor]$. Combining this with $\ell_n\nearrow \ell(\theta)$, we get the last assertion.
\end{proof}

\begin{remark}
By monotonicity, conclusion \eqref{cts-front} holds for all initial conditions $\DI_0=\gamma\bm{\delta}_{(0,0)}$ with $\gamma\in(\gamma_1,1]$.
\end{remark}

Now we address the case when $\theta \in (1.5,4]$.  Recall that $\kappa$ is defined in \eqref{eq:kappa}.
\begin{prop}\label{prop:middle_conv_front_origin_IC} Assume $\theta\in(1.5,4]$. Under the initial condition \eqref{eq:det_ini_Lower},
for any $\eps>0$, we have
\begin{equation}\label{eq:middle_outside_cone}
    \lim_{n\rightarrow\infty}\sup_{m:\,m<(\kappa-\eps)n\,\mathrm{or}\,m>(1-\kappa+\eps)n}\DI_{n}(m,n-m)=0;
\end{equation}
and
\begin{equation}\label{eq:middle_inside_cone}
    \lim_{n\rightarrow\infty}\sup_{m:\,(\kappa+\eps)n<m<(1-\kappa-\eps)n}|\DI_{n}(m,n-m)-\ell(\theta)|=0.
\end{equation}
\end{prop}

\begin{proof}
Write $Y:\rN\times\rN\rightarrow [0,1]$ be the value when the frontier arrives, i.e.,
\begin{equation}\label{eq:frontier_value}
    Y(m,n)=\DI_{m+n}(m,n).
\end{equation}
The key observation is the following, which can be easily proved by induction.

\ul{Claim~1}: For any $\mathbf{a}\in\rN\times\rN$ with
\[
Y(\mathbf{a})>Y(0,0)=\gamma,
\]
then, for any $\x\in \rN\times\rN$,
\[
Y(\x+\mathbf{a})>Y(\x).
\]

Inspired by this, we define
\[
g(\x)=\left.\frac{\partial Y}{\partial \gamma}\right|_{\gamma=0}(\x).
\]
Since $Y(m,n)=1-\alpha^{-(Y(m-1,n)+Y(m,n-1))}$, a simple application of the chain rule yields
\[
g(m,n)=(\log \alpha)(g(m-1,n)+g(m,n-1)).
\]
Now, by induction, one can easily show

\ul{Claim~2}:
\[
g(m,n)=(\log \alpha)^{m+n}\binom{m+n}{m}.
\]

By Stirling's formula, we get
\begin{equation}\label{eq_g_esti}
g(m,n)={\sqrt{\frac{m+n+1}{2\pi(m+1)(n+1)}}}\left(\frac{\log\alpha}{s^s(1-s)^{1-s}}\right)^{m+n}\left(1+O\left(\frac{1}{m+1}+\frac{1}{n+1}\right)\right),
\end{equation}
where $s=m/(m+n)$.

Now we have finished the preparation and are ready to show \eqref{eq:middle_outside_cone} and \eqref{eq:middle_inside_cone}.

For \eqref{eq:middle_outside_cone}, note that $1-\alpha^{-x}\leq x\log \alpha $. We can get (by induction)
\[
Y(
\x)\leq g(\x)\gamma.
\]
When $\x=(m,n)$ is outside the cone (meaning the condition inside the supreme in \eqref{eq:middle_outside_cone}), $s^s(1-s)^{1-s}>\kappa^\kappa(1-\kappa)^{1-\kappa}$ ($s=m/m+n$ as before). Therefore, $\log\alpha/s^s(1-s)^{1-s}<1-\eps_0$, for some $\eps_0$ depending on $\eps$. We see that $g(\x)$ converges to zero uniformly. Hence, $Y(\x)\leq g(\x)\gamma\rightarrow 0$ uniformly. The proof of  \eqref{eq:middle_outside_cone} is complete.

For \eqref{eq:middle_inside_cone}, we need more work. Since for the upper extreme case ($\DI_0(n,-n)\equiv 1$), \eqref{eq:middle_inside_cone} holds, we just need to show the lower extreme case, i.e., we can assume that $\gamma$ is small enough.  {For any $\eps>0$, by {\eqref{eq_g_esti}}, we can find some large $\mathbf{b},\mathbf{b}_1,\mathbf{b}_2\in\rN\times\rN$ in the cone $\{(b_1,b_2):{\kappa+0.5\eps<b_2/(b_1+b_2)<\kappa+\eps}\}$ such that $\mathbf{b}=\mathbf{b}_1+\mathbf{b}_2$, $\min\{g(\mathbf{b}),g(\mathbf{b}_1),g(\mathbf{b}_2)\}>1$, and $\{\mathbf{b}_1,\mathbf{b}_2\}$ generates $\ze^2$ in the sense that any $(x,y)\in\ze\times\ze$ can be written as $m\mathbf{b}_1+n\mathbf{b}_2$ for some $m,n\in\ze$. Letting~$\gamma$ be small enough, we can assume $\min\{Y(\mathbf{b}),Y(\mathbf{b}_1),Y(\mathbf{b}_2)\}>\gamma$.} We will show that
\begin{equation}\label{eq:side_1}
\lim_{k\rightarrow\infty}Y((x,y)+k\mathbf{b})=\ell(\theta),\;\mbox{ for all } x\in\{0,\ldots,b_1\}, y\in\{0,\ldots,b_2\}.
\end{equation}

Note that once we get \eqref{eq:side_1}, we have \eqref{eq:middle_inside_cone}. The reason is as follows. By induction, one can easily show that for any fixed $n$, $Y_n(n-m,m)$ is increasing in $m\in\{0,\dots,\lfloor n/2\rfloor\}$. Hence, for any $(m,n-m)$ with $(\kappa+\eps)n<m<(1-\kappa-\eps)n$, we can find some $(x,y)$ in the rectangle in \eqref{eq:side_1},$k\in\rN$, such that $Y(n-m,m)\geq Y((x,y)+k\mathbf{b})$. If the latter term goes to $\ell(\theta)$, then so does the former term (note that all $Y$'s are less than $\ell(\theta)$).

We are remaining to prove \eqref{eq:side_1}. By \ul{Claim~1}, we know that for any $(x,y)\in \rN\times \rN$,  $Y((x,y)+k\mathbf{b})$ is increasing in $k$ and hence that the limit does exist. We denote the limit by $h(x,y)$. Obviously,
\[
h(x,y)=h((x,y)+\mathbf{b}).
\]

By \ul{Claim~1}, we have $Y((x,y)+\mathbf{b}_i)>Y(x,y)$. Hence,
\[
h(x,y)\leq h((x,y)+\mathbf{b}_i)\leq h((x,y)+\mathbf{b}_1+\mathbf{b}_2)=h((x,y)+\mathbf{b})=h(x,y).
\]
Therefore, $h(x,y)=h((x,y)+\mathbf{b}_i)$. Since $\{\mathbf{b}_1,\mathbf{b}_2\}$ generates $\ze\times \ze$, we obtain that all $h(x,y)$'s are equal. Moreover since $Y(m,n)=1-\alpha^{-(Y(m-1,n)+Y(m,n-1))}$, by taking the limit, we get that $h(x,y)$ satisfies the same equation of $\ell(\theta)$. Finally, since $h(x,y)\geq \gamma>0$, we exclude the case $h(x,y)=0$ and reach the conclusion that $h(x,y)=\ell(\theta)$.
\end{proof}

\subsection{Convergence on all layers near the frontier}
We consider the proportion of infection in all layers near the frontier. As in the previous subsection,
we focus on  the initial conditions supported by the line $x+y=0$. By the $k$-th layer at time $n$ near the frontier,
we mean {$\{\DI_n(m,n+1-k-m)\}_{m\in\ze}$}. By symmetry, we only consider the frontier on the positive direction.

We start with the  initial condition \eqref{eq:det_ini_Upper}. Denote
$$
y_{n}^{(i)}=\DI_{n}(n+1-i,0) (=\DI_{n}(n+1-i-k,k),\, \mbox{for all }k), i=1,2,\ldots, n={i-1,i,\ldots}.
$$
When $i=1$, $(y_n^{(1)})$ satisfies the recursive equation \eqref{eq:y_1}.
For $i=2$, by \eqref{eq:lim_X} we have
\begin{equation}\label{eq:y_2}
\aligned
y_{1}^{(2)} &=(1-y^{(1)}_0)(1-\alpha^{-y^{(1)}_0}),\\
 y_{n+1}^{(2)}&= (1-y_{n}^{(1)}) (1-\alpha^{-(2 y_{n}^{(2)} + y_{n}^{(1)})}), \mbox{ for all } n\geq 1.
\endaligned
\end{equation}
In general, for $i\geq 3$,  we have
\begin{equation}\label{eq:y_i}
\aligned
y_{i-1}^{(i)}=&\left(1-\sum_{1\leq j<i} y_{j-1}^{(j)}\right)(1-\alpha^{-({2y_{i-2}^{(i-2)} + y_{i-2}^{(i-1)}+ 2y_{i-2}^{(i-2)}})});\\
y_{n+1}^{(i)}=& \left(1-\sum_{1\leq j<i} y_{n+1-i+j}^{(j)}\right) (1-\alpha^{-(2 y_{n}^{(i)} + y_{n}^{(i-1)}+ 2y_n^{(i-2)})}), \,\mbox{ for all } n\geq i-1.
\endaligned
\end{equation}

Let us now define the limiting constants. Recall the sequence $\ell^{(i)}=\ell^{(i)}(\theta)$ defined in \eqref{eq:ell_i}. {It is easy to verify that when $\theta\leq 1.5$, \eqref{eq:ell_i} has no solution in $(0,1]$, in which case we define $\ell^{(i)}\equiv 0$.} On the other hand, when $\theta > 1.5$, \eqref{eq:ell_i} has a unique solution in $(0,1)$, which is defined as $\ell^{(i)}$.

\begin{prop}\label{prop:conv_front_layers_const_IC}
Under the  initial condition \eqref{eq:det_ini_Upper}, for any $i\geq 1$, $y_n^{(i)}$ converges to $\ell^{(i)}=\ell^{(i)}(\theta)$.
\end{prop}
\begin{proof}
The case $i=1$ has been proved in Lemma \ref{lem:conv_front_const_IC}.
We shall only prove for the case when $i=2$; the case when $i\geq 3$ can be proved similarly.

We first consider the case when $\theta\neq 1.5$.
Define
\[
\aligned
f(x,y)&=(1-y) (1-\alpha^{-(2 x + y)}), \q x,y\in [0,1].
\endaligned
\]
We have $y_{n+1}^{(2)} = f(y_{n}^{(2)},y_{n}^{(1)})$, and $\ell^{(2)} = f(\ell^{(2)},\ell^{(1)})$.
Elementary calculus yields that
\begin{equation}\label{ineq_f}
  \aligned  |f(x_1,y_1)-&f(x_2,y_2)|\leq\\
 &|y_1-y_2|+(1-\min\{y_1,y_2\})\log(\alpha)(2|x_1-x_2|+|y_1-y_2|).
 \endaligned
\end{equation}
Therefore,
\[\aligned
  |y_{n+1}^{(2)}& - \ell^{(2)}| \leq \\ &|y_{n}^{(1)} - \ell^{(1)}| + 2(1-\min\{y_n^{(1)},\ell^{(1)}\})\log(\alpha)\cdot (\frac{1}{2}|y_{n}^{(1)} - \ell^{(1)}| + |y_{n}^{(2)} - \ell^{(2)}|).
\endaligned
\]
When $\theta\neq 1.5$, it is easy to show using equation \eqref{eq:ell}  that $2(1-\ell^{(1)})\log(\alpha) < 1$. It follows that
\begin{equation}\label{eq:y_2_recur}
 \aligned
 |y_{n+1}^{(2)} - \ell^{(2)}| \leq&(1.5-\frac{1}{2}\delta+\log(\alpha)|y_{n}^{(1)}- \ell^{(1)}|)|y_{n}^{(1)} - \ell^{(1)}|\\
 & + (1-\delta+2\log(\alpha)|y_{n}^{(1)}- \ell^{(1)}|) |y_{n}^{(2)} - \ell^{(2)}|,
 \endaligned
\end{equation}
where $\delta>0$ is a fixed constant. Because $y_{n}^{(1)} \to \ell^{(1)}$, taking limsup on both sides of~\eqref{eq:y_2_recur} yields that $y_{n}^{(2)} \to \ell^{(2)}$.

We now turn to the case when {$\theta=1.5$}. In this case, $\alpha=\sqrt{e}$ and
\[
f(x,y)=(1-y) (1-e^{-(2 x + y)/2}), \q x,y\in [0,1].
\]
It is easy to show that for any $\eps\in(0,1)$, there exist $\eps_0\in(0,\eps)$ and $\delta\in(0,1)$ such that for any $y\in[0,\delta]$, we have,
\[
f(x,y)<\left\{\begin{array}{cc}
     \eps-\eps_0, & \mathrm{~when~}x\in [0,\eps]; \\
     f(\eps,y)+\exp(-\eps)(x-\eps)<x-\eps_0, & \mathrm{~when~}x \in (\eps,1]. \\
\end{array}\right.
\]
Using such a fact and a similar argument to above, we get that $y_{n}^{(2)} \to 0$.
\end{proof}

Simple modifications of the proof of Proposition \ref{prop:conv_front_layers_const_IC} yield the following generalization.
\begin{cor}\label{cor:conv_front_layers_general_IC}
For any initial condition $\{\DI_0(m,-m)\}_{m\in \ze}$ such that
\begin{equation}\label{eq:assump_conv_front_layers_general_IC}
\lim_{n\rightarrow \infty}\sup_{m\in \ze}|\DI_n(m,n-m)-\ell(\theta)|=0,
\end{equation}
we have,
\[
\lim_{n\rightarrow \infty}\sup_{m\in \ze}|\DI_n(m,n+1-k-m)-\ell^{(k)}(\theta)|=0, \q \q \mbox{for all }k\geq 1.
\]
\end{cor}

\begin{remark}
Thanks to the results in Section~\ref{ssec:1st_det}, we can get some sufficient conditions to guarantee \eqref{eq:assump_conv_front_layers_general_IC}. For example, if there exists  $\eps>0$ such that $\DI_0(m,-m)>\eps$ for all $m\in\ze$, then \eqref{eq:assump_conv_front_layers_general_IC} holds. The  condition can be further weakened to be the following: If there exist  $\eps>0$ and $k\in \zz{N}$ such that for any $m\in \ze$, there exists $n\in \ze$ such that $|n-m|<k$ and $\DI_0(n,-n)>\eps$, then the frontier at time $2k$ is uniformly bounded from below and hence \eqref{eq:assump_conv_front_layers_general_IC} holds.
\end{remark}

Next result describes some asymptotics about the sequence $(\ell^{(i)})$.
\begin{lemma}\label{lm:bd_sum_ell}
If $\theta>1.5$, then,
\begin{equation}\label{eq:bd_sum_ell}
    \sum_{i=1}^\infty \ell^{(i)}>\frac{\theta}{1+\theta}.
\end{equation}
Moreover, there exist $A>0$ and $b\in(0,1)$ depending on $\theta$ such that
\begin{equation}\label{eq:decay_ell}
\ell^{(n)}\leq A b^n, \q \mbox{for all }n\geq 1.
\end{equation}
\end{lemma}

\begin{proof}
Write $s_n=\sum_{i=1}^n \ell^{(i)}$. By \eqref{eq:ell_i},   $\ell^{(i)}<1-s_{i-1}$. Therefore, $s_n\leq 1$ and $\ell^{(i)}\rightarrow 0$. Note that $1-\alpha^{-t}\sim t\log \alpha$ as $t\to 0$. Hence as $n\to\infty$,
\begin{equation}\label{eq:ell_r_asym}
\ell^{(n+1)}\sim (1-s_n)(2\ell^{(n+1)}+\ell^{(n)}+2\ell^{(n-1)}) \log \alpha .
\end{equation}

We will prove \eqref{eq:bd_sum_ell} by contradiction. Suppose that $s_\infty=\sum_{i=1}^\infty \ell^{(i)}\leq\frac{\theta}{1+\theta}$.

If $\sum_{i=1}^\infty \ell^{(i)}<\frac{\theta}{1+\theta}$, then by \eqref{eq:ell_r_asym}, we can find some $\delta>0 $ and $n_0$ such that when $n\geq n_0$, we have
\[
\ell^{(n+1)}\geq \left(\frac{1}{5}+\delta\right) (2\ell^{(n+1)}+\ell^{(n)}+2\ell^{(n-1)}).
\]
Then, one can use induction to show that $\ell^{(n)}\geq \min\{\ell^{(n_0)},\ell^{(n_0-1)}\}$. This is contradictory to that $\ell^{(n)}\rightarrow 0$.

Now assume $\sum_{i=1}^\infty \ell^{(i)}=\frac{\theta}{1+\theta}$. By \eqref{eq:ell_r_asym}, for any $\delta>0$, we can find some $n_0$ such that when $n\geq n_0$, we have
\[
\ell^{(n+2)}\geq (1-\delta)\frac{\ell^{(n+1)}+2\ell^{(n)}}{3}.
\]
By induction, one can show that when $n\geq n_0$ and $k\geq 2$,
\[
\ell^{(n+k)}\geq (1-\delta)^{k-1}\frac{3\ell^{(n+1)}+\ell^{(n)}}{9}.
\]
By summation, we get
\begin{equation}\label{eq_t1}
 \sum_{k\geq2}\ell^{(n+k)}\geq \frac{1-\delta}{\delta}\frac{3\ell^{(n+1)}+\ell^{(n)}}{9}.
\end{equation}
On the other hand, by Taylor expansion we have
\begin{equation}\label{eq:Taylor_2nd}
 1-\alpha^{-t}\geq t\log \alpha-\frac{t^2(\log \alpha)^2}{2}, \mathrm{~for~}t\in[0,5].
\end{equation}
Using the above estimates, we get that with $A_n=2\ell^{(n+1)}+\ell^{(n)}+2\ell^{(n-1)}$,
\begin{equation}\label{eq:ell_lb}
\aligned
\ell^{(n+1)}&=(1-s_n)(1-\alpha^{-A_n})
\geq (1-s_n)A_n\log \alpha -\frac{1}{2}A_n^2 (\log \alpha)^2\\
&=(1-s_{\infty})A_n\log \alpha+A_n\log \alpha\ \left(\sum_{k\geq 1}\ell^{(n+k)}-\frac{1}{2}A_n\log \alpha\right).
\endaligned
\end{equation}
Using \eqref{eq_t1} twice, we  get
\[
 \sum_{k\geq1}\ell^{(n+k)}\geq \frac{1-\delta}{2\delta}\frac{3\ell^{(n+1)}+\ell^{(n)}+3\ell^{(n)}+\ell^{(n-1)}}{9}.
\]
It follows that by choosing $\delta$ to be small enough, we have $\sum_{k\geq 1}\ell^{(n+k)}>{1}/{2}\cdot A_n\log \alpha$ for all $n\geq n_0(\delta)$. Therefore, by \eqref{eq:ell_lb}, we get that when $n\geq n_0$,
\[
\ell^{(n+1)}\geq(1-s_\infty)A_n\log \alpha=\frac{2\ell^{(n+1)}+\ell^{(n)}+2\ell^{(n-1)}}{5}.
\]
Therefore,
\[\inf_{n\geq n_0}\ell^{(n)}\geq\min\{\ell^{(n_0)},\ell^{(n_0+1)}\},
\]
which again contradicts to that $\ell^{(n)}\rightarrow 0$.

We now prove \eqref{eq:decay_ell}. By \eqref{eq:ell_r_asym} and \eqref{eq:bd_sum_ell}, there exists $c>0$ such that when $n$ is large enough, we have
\[
\ell^{(n+1)}\leq \left(\frac{1}{5}-c\right) (2\ell^{(n+1)}+\ell^{(n)}+2\ell^{(n-1)}).
\]
It follows that for a $c'>0$, when $n$ is large enough, {say, $n\geq n_1$},
\[
\ell^{(n+1)}\leq (1-c') \frac{\ell^{(n)}+2\ell^{(n-1)}}{3}.
\]
Let $b$ be the unique solution in $(0,1)$ to $x^2=(1-c')(x+2)/3$. Then one can show that if \eqref{eq:decay_ell} holds for $n=k$ and $k+1$, then it holds for $n=k+2$ as well. Pick an $A$ large enough such that \eqref{eq:decay_ell} holds for all $n\leq n_1+1$. The conclusion follows.
\end{proof}

\section{Distribution Near the Frontier: the Stochastic Case}\label{Sec:Front_SIR}
In this section, we prove analogous results for the SIR process. Again, we  focus on the first quadrant.


\subsection{When $\theta \leq 1.5$}

This case is relatively easy because the expectation is decreasing.
\begin{prop}\label{prop:front_theta_small}
Suppose that $\theta \leq 1.5$ and $I^N_0(\x)=0$ for all $\x\not\in\{(m,-m): m\in\zz{Z}\}$.
\begin{enumerate}[(i)]
\item If $\theta<1.5$, then for all $N$,
\[
  \sum_{n=0}^\infty \sup_{m\in\zz{Z}} E(I^N_{n}(n-m,m)) <\infty.
\]
Consequently, for any sequence $(m_n)\subseteq\zz{Z}$,
\[
I^N_{n}(n-m_n,m_n)\to 0  \mbox{   a.s.}.
\]
\item If $\theta=1.5$, then for any sequence $(m_n)\subseteq\zz{Z}$,
\[
I^N_{n}(n-m_n,m_n)\toop 0  \mbox{ }.
\]
\item Suppose further that $I^N_0(m,-m)$ is constant and $\theta < 1.5$, then for all $k\in\rN$,
\[
  \sum_{n=0}^\infty  E(I^N_{n}(0,n-k)) <\infty.
\]
Consequently, for any sequence $(m_n)\subseteq\zz{Z}$,
\[
I^N_{n}(n-m_n-k,m_n)\to 0  \mbox{   a.s.}.
\]
\end{enumerate}
\end{prop}
\begin{proof}
We first prove Part (i).
For any $n$ and $m$, it is easy to see that  $I^N_{n}(n-m,m)$ is increasing in the initial condition as long as the initial condition is supported by the diagonal line $\{(m,-m): m\in\zz{Z}\}$. Therefore, for Part (i) and~(ii), it suffices to prove the conclusions under the initial condition that
\begin{equation}\label{eq:lgest_IC}
I^N_0(m,-m)=N,\mbox{ and } I^N_0(\x)=0 \mbox{ for all } \x\not\in\{(m,-m): m\in\zz{Z}\}.
\end{equation}
Under such an initial condition, $E(I^N_{n}(n-m,m))$ is constant in $m$. Write  $E(I^N_{n}(n-m,m)) = y^N_n$ for any $m$. By \eqref{eqn:X_R}, we have
\[
y^N_{n+1} = N\cdot E\left(1 - \left(1-\frac{1+\theta}{5N}\right)^{\tilde{I}_n^N(n+1,0)}\right) =:N E(g_N(\tilde{I}_n^N(n+1,0))),
\]
where $g_N(x)=1-(1-(1+\theta)/(5N))^{x}$ and $\tilde{I}_n^N(n+1,0) = I_n^N(n,0) + I_n^N(n+1,-1) $.
 Elementary calculus yields
\begin{equation}\label{eq:gN_contract}
g_N(x)\leq \frac{1+\theta}{5N}x, \q\mbox{ for all } x\in\zz{N}.
\end{equation}
Hence, $y^N_{n+1}\leq 2(1+\theta)/5\cdot y^N_{n}$. When $\theta<1.5$, we see that $y^N_{n}$ decays to zero exponentially and the conclusion follows.

We turn to Part (ii). When $\theta=1.5$, we only have $y^N_{n+1}\leq y^N_{n}$. This guarantees that $\delta:=\lim_{t\to\infty}y^N_t$ exists. We want to show that $\delta=0$. By induction, it is easy to show that
\[
g_N(x)\leq \frac{1}{2N}x-\frac{1}{4N^2}(x-1),\q\mbox{ for all } x\in \zz{N}.
\]
It follows that
\[
y^N_{n+1}\leq y^N_{n}-\frac{1}{4N}P(I_n^N(n,0)\geq 2).
\]
Therefore, $P(I_n^N(n,0)\geq 2)\rightarrow 0$. On the other hand,
\[\aligned
 &P(I_n^N(n,0)\geq 2)\\
\geq &P(I_{n-1}^N(n-1,0)\geq 1)\cdot P(I_n^N(n,0)\geq 2| I_{n-1}^N(n-1,0) + I_{n-1}^N(n,-1)=1).
\endaligned
\]
The latter probability is constant in $n$, and so  $P(I_{n-1}^N(n-1,0)\geq 1)\to 0$.

Finally, we prove Part (iii). For $k\geq 1$, $E(I^N_{n}(n-k,0))$ is no longer increasing in the  initial condition. However, if $I^N_0(m,-m)$ is constant in $m$, then  $E(I^N_{n}(n-m-k,m))$ is constant in $m$ as well. Write  $E(I^N_{n}(n-m-k,m)) =  y^{N,k}_n$ for any~$m$. By \eqref{eqn:X_R} and \eqref{eq:gN_contract},
\[
y^{N,k}_{n+1} \leq \zeta (y^{N,k}_n + 1/2\cdot y^{N,k-1}_n+y^{N,k-2}_n ),
\]
where $\zeta=2(1+\theta)/{5}<1$.
Therefore,
\[
(1-\zeta)\sum_{n\geq 0}y^{N,k}_{n+1} \leq y^{N,k}_{0} +   1/2\sum_{n} y^{N,k-1}_n+ \sum_{n} y^{N,k-2}_n.
\]
It follows by induction that  $\sum_{n\geq 0}y^{N,k}_{n} <\infty $ for all $k$.
\end{proof}

\subsection{When $\theta > 1.5$}\label{ssec:conv_front_det}

\subsubsection{Convergence on the frontier}\label{ssec:1st_stoc}

\begin{thm}\label{thm:conv_SIR_front_const_IC}
Under the initial condition IC2, the  process $\{I^N_n(m,n-m)\}$ behaves differently depending on the value of $\theta$:
\begin{enumerate}[(i)]
    \item If $\theta\in(1.5,4]$, then for any $\eps>0$, for all $N$ sufficiently large,
    \begin{equation}\label{eq:conv_SIR_front_middle_outside}
     \lim_{n\rightarrow\infty}P\left(\sum_{m: m<(\kappa-\eps)n\,\mathrm{or}\,m>(1-\kappa+\eps)n}I_n^N(m,n-m)>0\right)=0,
    \end{equation}
    and
    \begin{equation}\label{eq:conv_SIR_front_middle_inside}
     \limsup_{n\rightarrow\infty}\max_{m:(\kappa+\eps)n<m<(1-\kappa-\eps)n}P\left(\left|\frac{I_n^N(m,n-m)}{N}-\ell(\theta)\right|>\eps\right)<\eps,
    \end{equation}
   where $\ell(\theta)$ and $\kappa=\kappa(\theta)$ are defined in \eqref{eq:ell} and \eqref{eq:kappa}, respectively.

    \item If $\theta\in(4,\infty)$, then for any $\eps>0$, for all $N$  sufficiently large, we have
    \begin{equation}\label{eq:conv_SIR_front_large}
     \limsup_{n\rightarrow\infty}\max_{m: \eps n<m<(1-\eps)n}P\left(\left|\frac{I_n^N(m,n-m)}{N}-\ell(\theta)\right|>\eps\right)<\eps.
    \end{equation}
\end{enumerate}
\end{thm}

The proof of Theorem \ref{thm:conv_SIR_front_const_IC} is divided into several parts. To simplify the notation, we write
\[
W(m,l)=W^N(m,l)=I_{m+l}^N(m,l).
\]
Note that similarly to the limiting  process, $W$ is monotone {in the initial condition as long as the initial condition is supported by the line $\{(x,y)|x+y=0\}$}, namely, if $\{I'_0(m,-m)\}_{m\in \ze}$ stochastically dominates $\{I''_0(m,-m)\}_{m\in \ze}$, then the corresponding $\{W'(n+m,-m)\}$ stochastically dominates $\{W''(n+m,-m)\}$.

{As we mentioned in Section \ref{ssec:conn_perc}, the SIR process has an equivalent description in terms of the  $N$-percolation.
Similarly, the ``frontier'' process $W$ can be described in terms of an  oriented percolation defined as follows, which we call the oriented $N$-percolation.
Consider the graph with vertex set $\{(x,y)\in\ze^2: x + y\geq 0\}\times [N]$. The edge set consists of all oriented edges pointing up or right, i.e., of the form $\{(\x,i)\rightarrow(\y,j): \y\in\{\x+(0,1),\x+(1,0)\}\}$.  Each directed edge is open with probability $P_N^\theta$. For any initial condition $I^N_0$ supported by the line $x+y=0$, let $A$ be any subset of $\ze^2\times [N]$ such that $\# (A\cap \x\times [N])=I^N_0(\x)$ for any $\x\in\ze^2$. Then,
for any site $\x=(x,y)$ satisfying $x+y\geq 0$, $W(\x)$ equals the number of vertices at site $\x$ that can be reached via open directed edges from~$A$.}

The conclusion in \eqref{eq:conv_SIR_front_middle_outside} follows from Remark~\ref{rmk_kappa}. To show the other assertions, we first give two lemmas, which are implied by the corresponding results of the limiting process.

\begin{lemma}
Let the initial condition be $I^N_0(m,-m)= N$ for all $m\in \ze$. Then, for any $\eps>0$, for all $N$ sufficiently large,
\begin{equation}\label{eq_l1}
    \lim_{n\rightarrow\infty}P\left(\frac{W^N(n,0)}{N}>\ell(\theta)+\eps\right) < \eps.
\end{equation}
\end{lemma}
\begin{proof}
By Lemma~\ref{lem:conv_front_const_IC},  for the limiting process of $W^N/N$, denoted by $\DI'$, we can find  $n_0$ such that $\DI'(n_0,0)<\ell(\theta)+\eps/2$. Because $W^N(\x)/N\toop\DI'(\x)$ for any $\x\in\ze^2$,  we can find an $N_0$ such that  $P(W^N(n_0,0)/N>\ell(\theta)+\eps)<\eps$ for $N>N_0$. By the monotonicity of $W$ with regard to the initial condition, we have that $W(n,0)$ is stochastically bounded by $W(n_0,0)$ for $n\geq n_0$. The conclusion follows.
\end{proof}

Similarly, one can show
\begin{lemma}\label{lem:front_lowbd}
For any $\eps>0 $ and $ a\in(0,1]$, there exist $n_0$ and $N_0$ such that when $N>N_0$ and the initial condition $I_0^N(0,0)\geq a N$, we have,
for any $m\in [(\kappa+\eps)n_0, (1-\kappa-\eps)n_0]$,
\begin{equation}\label{eq_l2}
    P\left(\frac{W(m,n_0-m)}{N}<\ell(\theta)-\eps\right)<\eps.
\end{equation}
\end{lemma}

The key step to prove Theorem \ref{thm:conv_SIR_front_const_IC} is to show that the survival probability for the oriented $N$-percolation is positive when $\theta> 1.5$. Our idea is to introduce constructions which will allow us to reduce questions about our oriented $N$-percolation to corresponding questions about the oriented $1$-percolation, a situation in which much is known.

\begin{proof}[Proof of \eqref{eq:conv_SIR_front_large}]
The upper bound is easy: by \eqref{eq_l1},  when $N$ is large enough,
$$
 \limsup_{n} \sup_{m}P\left(\frac{W(m,n-m)}{N}>\ell(\theta)+\eps\right)<\eps.
$$

We need to show the lower bound: when $N$ is large enough,
\begin{equation}\label{eq_eq2}
 \limsup_{n} \sup_{m:\eps n<m<(1-\eps n)}P\left(\frac{W(m,n-m)}{N}<\ell(\theta)-\eps\right)<\eps.
\end{equation}

We will define an oriented bond percolation process and describe its relationship to our original SIR process. Without loss of generality, assume that $\gamma <\ell(\theta)$. For any oriented edge $\vec{e}=(m,n)\rightarrow(m,n+1)\,\mathrm{or}\,(m,n)\rightarrow(m+1,n)$, define a $\{0,1\}$-valued random variable $\eta(\vec{e})$  as follows. Initially at the origin, there are $\lceil \gamma N\rceil$ infected particles, which we call red particles.
In general, suppose at time $t$ and at site $(m,n)$, there are $\lceil \gamma N\rceil$ red particles.
If at time $t+1$, there are at least $\lceil \gamma N\rceil$ in the endpoint of $\vec{e}$ infected by those red particles, then we declare $\eta(\vec{e})=1$, otherwise $\eta(\vec{e})=0$. For a site $(m,n)\neq(0,0)$, if there is an edge $\vec{e}$ pointing to $(m,n)$ with  $\eta(\vec{e})=1$, then there are at least $\lceil \gamma N\rceil$ particles infected by the red particles. In this case, we randomly choose $\lceil \gamma N\rceil$ infected particles and declare them red. On the other hand, if $\eta(\vec{e})=0$ for both edges  pointing to $(m,n)$, then we randomly choose  $\lceil \gamma N\rceil$ particles on the site $(m,n)$ and call them red particles.


Note that when $\theta>4$, $\psi(\gamma)=1-\exp(-\frac{1+\theta}{5}\gamma)>\gamma$ for $\gamma<\ell(\theta)$. It follows from \eqref{eqn:X_R} and Hoeffding's inequality that as $N$ gets large, $P(\eta(\vec{e})=1)$ can be arbitrarily close to $1$. Therefore, using the standard results on $2$-dimensional oriented percolation (see, e.g., (\cite{Du84})), we conclude that {for any fixed $\eps>0$, when $N$ is sufficiently large,}  the probability that a distant point  with its argument in $(\eps, \pi/2-\eps)$  is connected to the origin is close to $1$. In other words, for any $\eps>0$, when~$N$ is sufficiently large, we have
\begin{equation}\label{eq_l3}
    \limsup_{n} \sup_{m:\eps n<m<(1-\eps n)}P\left(\frac{W(m,n-m)}{N}<\gamma\right)<\eps.
\end{equation}

We are ready to prove \eqref{eq_eq2}.
By Lemma \ref{lem:front_lowbd}, letting $a=\gamma$, if we can find some $n_0$ such that
\begin{equation}\label{eq_eq1}
   W\left(m-\lfloor\frac{n_0}{2}\rfloor,l-n_0+\lfloor\frac{n_0}{2}\rfloor\right)\geq \gamma N,
\end{equation}
then we have
$$
P\left(\frac{W(m,l)}{N}<\ell(\theta)-\eps\right)<\eps.
$$
The event in \eqref{eq_eq1} holds with a high probability by \eqref{eq_l3}, therefore \eqref{eq_eq2} follows.
\end{proof}

\begin{proof}[Proof of \eqref{eq:conv_SIR_front_middle_inside}]
In order to make the description more inline with the usual oriented percolation, we will rotate the first quadrant counterclockwise by $\pi/4$ and work with the grid
$$
\cL = \{(m,n)\in \mathbb{Z}^{2} \, :\, m+n \textrm{ is even and } \, n\geq 0\}.
$$
After rotation, the cone $\{(m,l)\in\ze ^2:\kappa(m+l)<m<(1-\kappa)(m+l)\}$ in the first quadrant becomes $\{(x,y)\in\cL:|x|<(1-2\kappa)y\}$, and the edges in the oriented $N$-percolation point to either upper-right or upper-left . Write $\chi=(1-2\kappa)$.

We will define an oriented site percolation process as follows. For each $\mathbf{b} \in \cL$,
pick a $\beta < \chi$, a small $\delta<\beta/10$, a large $L\in \zz{N}$ satisfying $\beta L\in \rN$, and define for each $\mathbf{b}=(b_1,b_2)\in \cL$,
$$
C(\mathbf{b})=(\beta Lb_1,Lb_2){\in \cL}, \quad A(\mathbf{b})=C(\mathbf{b})+(-(\beta+\delta)L,(\beta+\delta)L)\times[0,L].
$$
Assume that there are $\lceil \gamma N\rceil$ red particles at $C(\mathbf{b})$. We set $\eta(\mathbf{b})=1$ if these red particles
are connected to
at least $\lceil \gamma N\rceil$ particles on both sites $C(b_1+1,b_2+1)$ and $C(b_1-1,b_2+1)$ via only the oriented edges inside $A(\mathbf{b})$, and $\eta(\mathbf{b})=0$ otherwise.

From this construction, it is easy to see that if  percolation occurs in the $\eta$-system, then percolation also occurs in the oriented $N$-percolation percolation. To see when the oriented site percolation occurs, we need the following lemma whose proof will be given later.
\begin{lemma}\label{lm_rect}
For any positive and rational $\beta<\chi$, $\delta<\beta/10$, we can find some small~$\gamma$ and large $L$ such that
$$
\lim_{N\rightarrow \infty}P(\eta(\mathbf{b})=1)=1.
$$
\end{lemma}

Note  that the rectangle $A(\mathbf{b})$ does not intersect with $A(\mathbf{b}+(0,2))$ or $A(\mathbf{b}+(4,0))$. This guarantees that the $\eta$-system is $4$-dependent, that is, when the (graph) distance between $B_1$ and $B_2$ is bigger than $4$, then $\{\eta(\mathbf{b}):b\in B_1\}$ and $\{\eta(\mathbf{b}):b\in B_2\}$ are independent. It follows from Lemma \ref{lm_rect} that the $\eta$-system percolates with a high probability for all $N$ sufficiently large; see, e.g., see \cite{LSS97}. Moreover, by  standard results in the two dimensional oriented percolation (\cite{Du84}), for any $\eps>0$,
$$
\lim_{N\rightarrow \infty}\max_{(x,y)\in\cL:|x|<(1-\eps)y}P_\eta((0,0)\rightarrow(x,y))=1,
$$
where $P_\eta$ is the law of the $\eta$-system and here $\x\rightarrow \y$ means the event that $\y$ is reached from $\x$ by open oriented edges.

Returning to {the oriented $N$-percolation, we get}
\begin{equation}\label{eq_l4}
    \lim_{N\rightarrow\infty} \max_{(x,y)\in\cL:|x|<(1-\eps)y}P\left(\frac{W(\beta Lx,Ly)}{N}\geq \gamma\right)=1.
\end{equation}

Similarly to the argument that \eqref{eq_l1}, \eqref{eq_l2} and \eqref{eq_l3} imply \eqref{eq:conv_SIR_front_large}, one can deduce~\eqref{eq:conv_SIR_front_middle_inside} from \eqref{eq_l1}, \eqref{eq_l2} and \eqref{eq_l4}.
\end{proof}

We now prove Lemma~\ref{lm_rect}.
\begin{proof}[Proof of Lemma~\ref{lm_rect}]
We will deduce the lemma from its limiting process. Consider the following deterministic system on $\cL$. Write $A=(-(\beta+\delta)L,(\beta+\delta)L)\times[0,L]$. Let $Y(0,0)=\gamma$ and define $Y:\cL\rightarrow \re $ recursively:
\[
\aligned
&Y(m,n)\\
=&\left\{ \begin{array}{cc}
    1-\exp(-\frac{1+\theta}{5}(Y(m-1,n-1)+Y(m+1,n-1))), & \mathrm{if}\,(m,n)\in A;\\
   0,  & \mbox{otherwise}.
\end{array}\right.
\endaligned
\]

As in the proof of Proposition~\ref{prop:middle_conv_front_origin_IC}, we will analyze
\[
g(m,n)=\left.\frac{\partial Y}{\partial \gamma}\right|_{\gamma=0}(m,n).
\]
As before, one can get $g(0,0)=1$ and for $(m,n)\in A$,
\[
g(m,n)=\frac{1+\theta}{5}(g(m-1,n-1)+g(m+1,n-1)).
\]
By induction, one can show that

\ul{Claim~1}:\, $g(m,n)=(\frac{1+\theta}{5})^n\times\#\mathrm{oriented~paths~from~}(0,0)\rightarrow(m,n)\mathrm{~in~}A$.

Moreover, by the reflection principle, we can do exact counting and get

\ul{Claim~2}:\, The number of oriented paths from $(0,0)$ to $(m,n)$ inside $(-k,k)\times[0,n]$ is
$$
\binom{n}{\frac{n+m}{2}}-\binom{n}{\frac{n+2k+m}{2}}-\binom{n}{\frac{n+2k-m}{2}}+\binom{n}{\frac{n+4k+m}{2}}+\binom{n}{\frac{n+4k-m}{2}}.
$$

Using Stirling's formula, one can show that the last four corresponding terms for $g(\beta L,L)$ are negligible compared with the first one, in other words,  the restriction that the path is  inside $A$ can be ignored. Therefore, we have that
$$
g(\beta L,L)\sim \sqrt{\frac{1}{2L\pi}} \left(\frac{1+\theta}{5}\right)^L h\left(\frac{1+\beta}{2}\right),
$$
where
\[
h(t)=\left(\frac{1}{t^t(1-t)^{(1-t)}}\right)^L\frac{1}{\sqrt{t(1-t)}}.
\]

By  assumption, $\beta<\chi=1-2\kappa$, hence $(1+\beta)/{2}<1-\kappa$. Therefore, $g(\beta L, L)\rightarrow\infty$ exponentially. In particular, we can find a large $L$  such that $g(\beta L,L)>1$. For such an $L$, we can find some small $\gamma$ satisfying $Y(\beta L,L)>\gamma$. Because $Y$ is the limit of $W/N$ as $N\rightarrow \infty$, the conclusion follows.
\end{proof}

\subsubsection{Convergence on all layers {near} the frontier}

We have shown in Theorem \ref{thm:conv_SIR_front_const_IC} that when $\theta>1.5$, the infection proportion on the first layer inside the cone is roughly $\ell(\theta)=\ell^{(1)}(\theta)$. In this subsection, we show that for the $i$-th layer in the cone, the infection proportion is roughly $\ell^{(i)}=\ell^{(i)}(\theta)$ defined in \eqref{eq:ell_i}.

\begin{thm}\label{thm:all_layer_SIR_2}
Suppose that $\theta>1.5$.  Consider the SIR process with initial condition IC2. For any $\vep>0$ and $i\in\zz{N}$, when $N$ is sufficiently large, we have
\[
     \limsup_{n\rightarrow\infty}\max_{m:(\kappa+\eps)n<m<(1-\kappa-\eps)n}P\left(\left|\frac{I_{n+i-1}^N(m,n-m)}{N}-\ell^{(i)}\right|>\eps\right)<\eps.
\]
\end{thm}
\begin{remark}\label{rmk:all_layer_SIR_2}
For future use, we note that the above conclusion also holds under the initial condition
\begin{equation}\label{eq_IC_adj}
    I_0^N(x,y)=\lfloor \gamma N\rfloor\bm{\delta}_{(0,0)}, \mbox{ and } R_0^N(x,y)=(N-I^N_0)\bm{\delta}_{x+y\leq 0}.
\end{equation}
\end{remark}

\begin{proof}
We shall only prove for the case when $i=2$; the case when $i\geq 3$ can be proved similarly. The idea is similar to the one used in the limiting process. Note that conditionally on the process  $(I_{n}^N(m,n-m))_{n\geq 0, m\in\zz{Z}}$, the process  $(I_{n+1}^N(m,n-m))_{n\geq 0, m\in\zz{Z}}$ is a Markov chain.

For notational ease, we write
\[\aligned
W(m,l)&=W^N(m,l)=\frac{I^N_{m+l}(m,l)}{N}, \mbox{ and}\\
Z(m,l)&=Z^N(m,l)=\frac{I^N_{m+l+1}(m,l)}{N},
\endaligned
\]
which represent the infection proportion on the first and second layers, respectively. For any $(m,l)\in\rN\times \rN$ and $j\leq k\in \zz{N}$, denote $T_k(m,l)$ and $L_k^j(m,l)$ by
\[\aligned
T_k(m,l)&=\{(x,y)\in \ze^2:x\leq m,\,y\leq l,\,x+y\geq m+l-k\};\\
L_k^j(m,l)&=\{(x,y)\in T_k(m,l):x+y=m+l-k+j\}.\\
\endaligned
\]
Note that the distribution of $Z(m,l)$ is determined by $\{W(x,y):(x,y)\in T_k(m,l)\}$ and $\{Z(x,y):(x,y)\in L_k^0(m,l)\}$. Note that because the ``frontier'' processes $W$ under the initial conditions IC2 and \eqref{eq_IC_adj} are the same, Theorem~\ref{thm:conv_SIR_front_const_IC} also holds for the initial condition \eqref{eq_IC_adj}. Therefore, to prove the theorem, it suffices to prove the following lemma.
\begin{lemma} \label{lm:SIR_second_layer}
For any $\vep>0$, there exist $k\in \zz{N}$ and $\delta>0$ such that as long as
\[
|W(\y)-\ell(\theta)|<\delta,\quad \mbox{for all } \y\in T_k(\x),
\]
 we have, when $N$ is sufficiently large,
\[
    P(|Z(\x)-\ell^{(2)}|>\vep)<\vep.
\]
\end{lemma}
\end{proof}

\begin{proof}[Proof of Lemma \ref{lm:SIR_second_layer}]
Write $M_i=\sup_{\y\in L_k^i(\x)}|Z(\y)-\ell^{(2)}|$. Obviously $M_0\leq 1$. We will use induction to show that with high probability, $M_i\leq 4\delta+(1-\vep_0)M_{i-1}$, for some $\vep_0$ depending only on $\theta$. Then we can get that $|Z(\x)-\ell^{(2)}|=M_k<4k\delta+(1-\vep_0)^k$ holds with a high probability, and the lemma follows.

By \eqref{eqn:X_R}, we have that, conditionally on $(W(\y),Z(\y-(0,1)), Z(\y-(1,0)))$,
\[
Z(\y)\stackrel{\mathrm{d}}{=}\frac{1}{N}\mathrm{Bin}\left(N(1-W(\y)),g_N\left(Z(\y-(0,1))+Z(\y-(1,0))+W(\y\right))\right),
\]
where $g_N(t)=1-(1-\frac{1+\theta}{5N})^{Nt}$. Note that $\ell(\theta) <1$, hence by letting $\delta$ be small enough, we have $1-W(\y)\geq \vep_1 $. Recall the  Chernoff bound for the binomial distribution: for $X \stackrel{\mathrm{d}}{=} \mbox{Bin}(n,p)$, we have $P(|X/n - p|\geq \delta) \leq \exp(-n\delta/(3p))\leq  \exp(- n\delta/3)$. It follows that
\begin{equation}\label{eq:CB}
 P(|Z(\y)-E(Z(\y))|>\delta_1)<a^{N},
\end{equation}
where $a\in (0,1)$ depends only on $\vep_1$ and $\delta_1$.

We need to estimate $E(Z(\y))$(conditioned on $(W(\y),Z(\y-(0,1)), Z(\y-(1,0)))$). Similarly to \eqref{eq:y_2_recur}, we can get that when $\delta$ is small enough, $N$ is sufficiently large,  and $|W(\y)-\ell^{(1)}|<\delta$,
$$
|E(Z(\y))-\ell^{(2)}|<3\delta+(1-\vep_0)\left|\frac{Z(\y-(0,1))+Z(\y-(1,0))}{2}-\ell^{(2)}\right|,
$$
where $\vep_0$ depends only on $\theta$.
When $\y\in L_k^i(\x)$, we get
$$
|E(Z(\y))-\ell^{(2)}|<3\delta+ (1-\vep_0)M_{i-1}.
$$
By induction and using \eqref{eq:CB}, we get that, with probability at least $1-k^2a^N$, $M_k<4k\delta+(1-\vep_0)^k$. Letting $k$ be large enough first and then $\delta$ be small enough, we finish the proof of the lemma.
\end{proof}

Similarly, one can show
\begin{prop}\label{prop:all_layer_SIR_IC_line}
Suppose that $\theta>1.5$. Consider the SIR process with initial condition $\lfloor\gamma N\rfloor \bm{\delta}_{x+y=0}$ for some fixed $\gamma\in(0,1]$. For any $\vep>0$ and $i\in \zz{N}$, when~$N$ is sufficiently large, we have
\[
     \limsup_{n\rightarrow\infty}P\left(\left|\frac{I_{n+i-1}^N(0,n)}{N}-\ell^{(i)}\right|>\eps\right)<\eps.
\]
\end{prop}

\subsection{The one-infected-particle initial condition}\label{ssec:front_1_infect}

In this subsection, we always assume that $\theta>1.5$ and the initial condition is IC1. Recall that
\[
\K(n)=\K^N(n)=\inf\left\{k:\sum_{(x,y)\in\ze^2:x+y=n}I^N_{n+k}(x,y)>0\right\}\mbox{ and }\K^N=\lim_{n\rightarrow \infty}\K^N(n).
\]
Note that the event $\K^N\leq k$ is equal to $\cap_n\{k: (\0,1)\stackrel{n+k}{\longleftrightarrow}\mathcal{L}(n)\}$, where $\mathcal{L}(n)=\mathcal{L}^N(n)=\{(x,y):x+y=n\}\times[N]$, and $(\0,1)\stackrel{n+k}{\longleftrightarrow}\mathcal{L}(n)$ denotes the event that $(\0,1)$ is connected to $\mathcal{L}(n)$ via $\leq n+k$ open edges in the $N$-percolation.

\begin{thm}
\begin{enumerate}[(i)]
    \item For any $i\in\rN$,
        \begin{equation}\label{eq:Delay_time}
         \lim_{N\rightarrow \infty }P(\K^N=i)=\ell^{(i+1)}.
        \end{equation}

    \item For any $k,j\in \rN$, $\vep>0$, when $N$ is sufficiently large, we have
        \begin{equation}\label{eq:layer_IC_0}
         \aligned
         &\limsup_{n\rightarrow \infty} \sup_{m:(\kappa+\vep)n<m<(1-\kappa-\vep)n}
         P\left(\left|\frac{I^N_{n+k+j}(m,n-m)}{N}-\ell^{(k+1)}\right|>\vep|\K^N=j\right)\\
         <&\vep.
         \endaligned
        \end{equation}
\end{enumerate}
\end{thm}

\begin{proof}
We start with \eqref{eq:Delay_time}.
It suffices to prove that for any $k\in\zz{N}$,
\begin{equation}\label{eq:Delay_time_k}
\lim_{N\rightarrow \infty}P(\K^N\leq k-1)= \sum_{i=1}^k\ell^{(i)}.
\end{equation}

{We will work mainly with $\sum_{i=0}^{k}I^N_{n+i}(\cdot,n-\cdot)$ instead of $I^N_{k}(\cdot,n-\cdot)$, because the former is increasing in the initial condition while the latter is not.} We first consider the upper bound. By
Proposition \ref{prop:conv_front_layers_const_IC}, for the deterministic limiting process $(\DI_n(\cdot))$, under the initial condition $\DI_0(x,y)=\bm{\delta}_{x+y=0}$,
for any $\vep>0$, we can find some $n\in\zz{N}$ such that
\[
\sum_{i=0}^{k-1}\DI_{n+i}(n,0)\leq \sum_{i=1}^k\ell^{(i)}+\vep/2.
\]
It follows from Proposition~\ref{prop:lim_N} that when $N$ is sufficiently large, under the initial condition $I_0(x,y)=N\bm{\delta}_{x+y=0}$,
\[
P\left(\sum_{i=0}^{k-1}I^N_{n+i}(n,0)/N\leq \sum_{i=1}^k\ell^{(i)}+\vep\right){>} 1-\vep.
\]
By symmetry, we get that
\[
P\left(\frac{\sum_{i=1}^N\mathbf{1}_{(\0,i)\stackrel{n+k-1}{\longleftrightarrow}\mathcal{L}(n)}}{N} \leq \sum_{i=1}^k\ell^{(i)}+\vep\right)>1-\vep.
\]
Therefore,
\[
E\left(\frac{\sum_{i=1}^N\mathbf{1}_{(\0,i)\stackrel{n+k-1}{\longleftrightarrow}\mathcal{L}(n)}}{N}\right)\leq \sum_{i=1}^k\ell^{(i)}+2\vep.
\]
In other words, $P((\0,{1})\stackrel{n+k-1}{\longleftrightarrow}\mathcal{L}(n))\leq \sum_{i=1}^k\ell^{(i)}+2\vep$, and we get the upper bound.

Now we turn to the lower bound. By Theorem~\ref{thm:all_layer_SIR} {for IC2}, under the initial condition $I_0^N=N\bm{\delta}_{(0,0)}$, for any $\vep>0$, when $N$ is sufficiently large, we have
\[
\limsup_{n\rightarrow \infty} \max_{m:(\kappa+\vep)n<m<(1-\kappa-\vep)n}P\left(\sum_{i=1}^k\frac{I_{n+i-1}^N(m,n-m)}{N}\leq \sum_{i=1}^k\ell^{(i)}-\vep\right)<\vep.
\]
Hence,
\[
\liminf_{n\rightarrow \infty}\min_{m:(\kappa+\vep)n<m<(1-\kappa-\vep)n}E\left(\sum_{i=1}^k\frac{I_{n+i-1}^N(m,n-m)}{N}\right)\geq(1-\vep)\left(\sum_{i=1}^k\ell^{(i)}-\vep\right).
\]
By monotonicity, we get that under the initial condition $I_0^N=N\bm{\delta}_{x+y=0}$, we also have
\[
\liminf_{n\rightarrow \infty}E\left(\sum_{i=1}^k\frac{I_{n+i-1}^N(n,0)}{N}\right)\geq(1-\vep)\left(\sum_{i=1}^k\ell^{(i)}-\vep\right).
\]
In other words,
\[
\liminf_{n\rightarrow \infty}P((\0,1)\stackrel{n+k-1}{\longleftrightarrow}\mathcal{L}(n))\geq (1-\vep)\left(\sum_{i=1}^k\ell^{(i)}-\vep\right).
\]
This is the lower bound we need.

We now prove \eqref{eq:layer_IC_0}. It suffices to show the following upper and lower bounds: For any $k,j\in\rN,\vep>0$, and $N$ sufficiently large,
\begin{equation}\label{eq:up_layer_IC_0}
\aligned
   &\limsup_{n,m\rightarrow \infty:(\kappa+\vep)n<m<(1-\kappa-\vep)n}P\left(\frac{\sum_{i=0}^{k-1}I^N_{n+i+j}(m,n-m)}{N}>\sum_{i=1}^k\ell^{(i)}+\vep,\K=j\right)\\
   <&\vep;
\endaligned
\end{equation}
and
\begin{equation}\label{eq:low_layer_IC_0}
\aligned
    &\limsup_{n,m\rightarrow \infty:(\kappa+\vep)n<m<(1-\kappa-\vep)n}P\left(\frac{\sum_{i=0}^{k-1}I^N_{n+i+j}(m,n-m)}{N}<\sum_{i=1}^k\ell^{(i)}-\vep,\K=j\right)\\
<&\vep.
\endaligned
\end{equation}

With Theorem~\ref{thm:all_layer_SIR} {for IC2} in mind, we aim to reduce our case to {IC2}. Introduce the following events:  fix a $\gamma_0\in(0,1]$, say, $0.1$, and define
\[
\aligned
B_n&:= \{\K^N(n+1)\geq j,\mbox{ and there exists }  m\in\ze \mbox{ such that } I^N_{n+j}(m,n-m)\geq \gamma_0 N\},\\
A_n&:= \cup_{k=n}^{2n-1}B_k.
\endaligned
\]
We will prove \eqref{eq:up_layer_IC_0} and \eqref{eq:low_layer_IC_0} based on the following lemma, whose proof will be given later.
\begin{lemma}\label{lm:reduction}
For any $N\in\zz{N}$ and $j\in\rN$, we have
\[
    \lim_{n\rightarrow\infty}P(A^{\mathrm{c}}_n, \K^N=j)=0.
\]
\end{lemma}

Now we prove \eqref{eq:up_layer_IC_0} and \eqref{eq:low_layer_IC_0}. For notational ease, we write $\Box^+$ ($\Box^-$, respectively) for the event $\{\sum_{i=0}^{k-1}I^N_{n+i+j}(m,n-m)/N-\sum_{i=1}^k\ell^{(i)}>\vep\}$ ($<-\vep$, respectively).
By Lemma \ref{lm:reduction}, we can find $n_0$ large enough such that
\[
\aligned
P(\Box^+, \K^N=j)&\leq  P(\Box^+, \K^N=j, A_{n_0})+P(\Box^+, \K^N=j, A^{\mathrm{c}}_{n_0})\\
&\leq P(\Box^+,A_{n_0})+P(\K^N=j,A^{\mathrm{c}}_{n_0})\\
&\leq P(\Box^+|A_{n_0})+\vep/2.
\endaligned
\]
By monotonicity, we have
\[
P(\Box^+|A_{n_0})\leq P(\Box^+|B'(2n_0-1)),
\]
where $B'(n)=\{\K^N(n+1)\geq j, \, I^N_{n+j}(m,n-m)= N \mbox{ for all } m\in\ze\}$.
By Proposition~\ref{prop:all_layer_SIR_IC_line}, we see that $P(\Box^+|B'(2n_0-1))$ is  small when both $N$ and $n-2n_0$ are large. The upper bound  \eqref{eq:up_layer_IC_0} follows.

For the lower bound \eqref{eq:low_layer_IC_0}, by Lemma \ref{lm:reduction} again, we can find $n_0$ large enough such that
\[
\aligned
P(\Box^-, \K^N=j)&\leq  P(\Box^-, \K=j, A_{n_0})+P(\Box^-, \K^N=j, A^{\mathrm{c}}_{n_0})\\
&\leq P(\Box^-,A_{n_0})+P(\K^N=j,A^{\mathrm{c}}_{n_0})\\
&\leq P(\Box^-|A_{n_0})+\vep/2.
\endaligned
\]
By monotonicity, we have
\[
P(\Box^-|A_{n_0})\leq \max_{\x\in T_{n_0}}P(\Box^-|B(\x)),
\]
where $T_{n_0}=\{(x,y)\in\ze^2: x,y\geq -j \mbox{ and } n_0\leq x+y\leq 2n_0-1 \}$,
and  $B(m,l)$ stands for that event that $I_{t_0}=\lceil\gamma_0 N\rceil\bm{\delta}_{(m,l)}$ and $R_{t_0}= (N-I_{t_0})\bm{\delta}_{x+y\leq m+l}$ {with $t_0=m+l+j$}.
By Remark~\ref{rmk:all_layer_SIR_2}, we see that $\max_{\x\in T_{n_0}}P(\Box^-|B(\x))$ is sufficiently small when both $N$ and $n-2n_0$ are large. This finishes the proof of \eqref{eq:low_layer_IC_0}.
\end{proof}

\begin{proof}[Proof of Lemma~\ref{lm:reduction}]
Note that $\K^N(n)\nearrow \K^N$. It follows that
\begin{equation}\label{eq_K_1}
 \lim_{n\rightarrow \infty}P(\K^N=j,\K^N(n)<j)=0.
\end{equation}

Write
\[
K_n=\{\K^N(n+1)\geq j,\mbox{ and there exists } m\in\ze  \mbox{ such that } I^N_{n+j}(m,n-m)\geq 1\}.
\]
It is easy to see that there exists $a=a(N,\theta)>0$ such that
\begin{multline} \label{eq_restart}
  \inf_{\vec{b}\in {\{0,\ldots,N\}}^\ze,\vec{b}\not\equiv\vec{0}}P(B_{n+1}|K_n\cap \{I^N_{n+j}(m,n-m)=\vec{b}(m)\mbox{ for all }m\})\\
  = P(B_{n+1}|K_n\cap \{I^N_{n+j}(m,n-m)=\mathbf{1}_{m=0} \mbox{ for all }m\})\geq a.
\end{multline}

Note that conditionally on the event $K_n$, the laws of $(I^N_{{n+1+j}}(m,n+1-m))_{m\in\ze}, \mathbf{1}_{K_{n+1}}$ and $\mathbf{1}_{B_{n+1}}$ depend only on $(I^N_{{n+j}}(m,n-m))_{m\in\ze}$. Therefore,
\[
\aligned
P&(\K^N=j,\K^N(n)=j,  A_n^\mathrm{c})\leq P(\cap _{k=n}^{2n-1}K_k,\cap _{k=n}^{2n-1}B_k^\mathrm{c})\\
&\leq \prod_{k=n}^{2n-2}\max_{\vec{b}\in{\{0,\ldots,N\}}^\ze,\vec{b}\not\equiv\vec{0}}P(K_{k+1}\cap B_{k+1}^\mathrm{c}|K_k\cap \{I^N_{{k}+j}(m,{k}-m)=\vec{b}(m),\forall m\} )\\
&\leq \prod_{k=n}^{2n-2}(1-a)=(1-a)^{n-1},
\endaligned
\]
where in the last inequality we used \eqref{eq_restart}.
Combining the result above with~\eqref{eq_K_1}, we prove the lemma.

\end{proof}

\section{Spreading Speed}\label{Sec:spread_speed}
Recall the limiting (deterministic) system $(\DI_n(\x),\DR_n(\x))_{n\in \rN,\x\in\ze^2}$ from Section~2. As before, we always assume that $\DR_0\equiv 0$ and focus only on the first quadrant. In this section, we assume further that $\DI_0$ is supported on the origin. To simplify notation, we define $\DD_n(\x)=\DR_{n+1}(\x)=\DI_n(\x)+\DR_n(\x)$, and  $\Dd_n(\x)=\Dd^N_n(\x)=R^N_{n+1}(\x)=I^N_n(\x)+R^N_n(\x)$. Note that $\Dd_n$ represents the number of particles that are infected before or at time $n$ and is increasing in the initial condition $I_0$. We show a weaker version of Theorem~\ref{thm:speed} in this section.

\begin{thm}\label{thm_speed}
For any fixed $\theta\in(0,\infty)$,  the following results hold:
\begin{enumerate}[(i)]
    \item Consider the deterministic system with the initial condition $\DI_0=\gamma\bm{\delta}_{\0}$ for some $\gamma\in (0,1]$.
 For any $\vep>0$, along any sequence $(i_k,j_k)_k \subset \rN\times \rN$ satisfying $i_k+j_k\rightarrow \infty$ and $\arg(i_k,j_k)\rightarrow \phi$ for some $\phi\in[0,\pi/2]$, we have
\begin{equation}\label{eq_speed_D_upper}
    \lim_{k\rightarrow\infty}\DD_{\lfloor (i_k+j_k)(\upsilon(\theta,\phi)^{-1}-\vep)\rfloor}(i_k,j_k)=0;
\end{equation}
 Moreover,  there exists  $\delta=\delta(\theta,\gamma,\vep)>0$ such that
 \begin{equation}\label{eq_speed_D_lower}
    \liminf_{k\rightarrow\infty}\DD_{\lfloor (i_k+j_k)(\upsilon(\theta,\phi)^{-1}+\vep)\rfloor}(i_k,j_k)>\delta.
\end{equation}

    \item Consider the SIR process with the initial condition IC2 \eqref{eq:IC_1}. For any $\vep>0$ and $N$, along any sequence $(i_k,j_k)_k \subset \rN\times \rN$ satisfying $i_k+j_k\rightarrow \infty$ and $\arg(i_k,j_k)\rightarrow \phi$ for some $\phi\in[0,\pi/2]$, we have,
 \begin{equation}\label{eq_speed_R_upper}
 \lim_{k\rightarrow\infty}P(\Dd^N_{\lfloor (i_k+j_k)(\upsilon(\theta,\phi)^{-1}-\vep)\rfloor}(i_k,j_k)>0)=0;   \end{equation}
 Moreover, for any $\vep>0$, there exists  $\delta=\delta(\theta,\vep,\gamma)>0$ such that when $N$ is sufficiently large,
 \begin{equation}\label{eq_speed_R_lower}
    \liminf_{k\rightarrow\infty}P(\Dd^N_{\lfloor (i_k+j_k)(\upsilon(\theta,\phi)^{-1}+\vep)\rfloor}(i_k,j_k)/N\geq \delta)>1-\vep.
\end{equation}
\end{enumerate}
\end{thm}

\begin{remark}\label{rk:speed_strong}
We will show a little bit more than \eqref{eq_speed_R_lower}: we can replace $D^N(i,j)$ by the corresponding number of particles that are infected via only those edges lying in the region $\{(x,y):\ 0< x+y< i+j\}$ except possibly with one endpoint at either $(0,0)$ or $(i,j)$.  We need this stronger version when proving Theorem~\ref{thm:speed}(ii). Note that Theorem \ref{thm:speed}(i) for IC2 is just \eqref{eq_speed_R_upper}, which, by the monotonicity of the recovered process with respect to the initial condition, also implies the conclusion for IC1 conditional on survival.
\end{remark}

\begin{proof}[Proof of Theorem~\ref{thm_speed}]

Similarly to what we did in Section \ref{ssec:conv_front_det}, we will analyze the asymptotics of
\[
g_n(\x)=\left.\frac{\partial \DI_n}{\partial \gamma}\right|_{\gamma=0}(\x).
\]
From the recursive equation, a simple application of chain rule yields
\[
g_{n+1}(m,l)=\frac{1+\theta}{5}\wt{g}_n(m,l).
\]

Similarly as before, by induction, one can show

\ul{Claim~1}:\, \[g_n(m,l)=\left(\frac{1+\theta}{5}\right)^n\times\#\mathrm{~LRW~paths~from~}(0,0)\rightarrow(m,l)\mathrm{~with~}n\mathrm{~steps}.\]
Here ``LRW'' stands for lazy random walk, namely, at  each time, the walker either stays at the same position or moves to one of the four neighbouring site with equal probability $1/5$.

As before, we need to analyze the asymptotic behavior of $g_n(m,l)$. Recall that Simple Random Walk (SRW) is such that at each time, the walker moves to one of the four neighbouring site with equal probability $1/4$. Write $\sharp_{\mathrm{S}}(m,l;n)$($\sharp_{\mathrm{L}}(m,l;n)$, respectively) for the number of SRW (LRW, respectively) paths from $(0,0)$ to $(m,l)$ with $n$ steps. We have
\begin{equation}\label{eq:nbr_srw}
\sharp_{\mathrm{S}}(m,l;n)
=\binom{n}{\frac{n-(m+l)}{2}}\binom{n}{\frac{n-(m-l)}{2}}.
\end{equation}
By considering the number of times that the walker does not move before time~$n$, we get
\begin{equation}\label{eq:L_to_S}
\sharp_{\mathrm{L}}(m,l;n)=\sum_{i:m+l\leq i\leq n, 2|i-m-l}\binom{n}{n-i}\sharp_{\mathrm{S}}(m,l;i).
\end{equation}
By Stirling's formula, we have
\[
\binom{n}{k}\approx \left(\frac{1}{u^u(1-u)^{1-u}}\right)^n, \mathrm{~where~}u=\frac{k}{n},
\]
where for any two sequences $(a_n)$ and $(b_n)$, $a_n\approx b_n$ means that
\[
\lim_{n\rightarrow\infty}\log a_n /\log b_n=1.
\]
It follows that when both $(m+l)/n$ and $m/(m+l)$ converge, we have
\begin{equation}\label{eq:nbr_lrw_asym}
\aligned
\sharp_{\mathrm{L}}(m,l;n)
\approx&\max_{i}\left(\frac{1}{t^t(1-t)^{1-t}}\right)^n\left(\frac{1}{r^r(1-r)^{1-r}}\right)^i\left(\frac{1}{s^s(1-s)^{1-s}}\right)^i\\
\approx&\left(\sup_{t\in(\frac{m+l}{n},1]}\frac{1}{t^t(1-t)^{1-t}r^{tr}(1-r)^{t(1-r)}s^{ts}(1-s)^{t(1-s)}}\right)^n\\
\approx&\left(\frac{1}{\exp(G(\frac{m+l}{n},\frac{m}{m+l}))}\right)^n,
\endaligned
\end{equation}
where $t=i/n,r={(t-(m+l)/n)}/{(2t)},s=(t-(m-l))/n/(2t)$,
\[
\aligned
  h(t)&=t\log t+(1-t)\log(1-t),\\
  g(t,v,a) &= h(t)+t\left(h\left(\frac{1}{2}-\frac{v}{2t}\right)+h\left(\frac{1}{2}-\frac{(1-2a)v}{2t}\right)\right),
\endaligned
\]
and
\begin{equation}\label{eq_def_G}
 G(v,a)=\inf_{t\in[v,1]}g(t,v,a).
\end{equation}
Note that $g(t,v,a)$ is negative, continuous on all variables,  strictly increasing in $v\in(0,1]$ and decreasing in $a\in[0,1/2]$. Hence, $G(v,a)$ has the same property and
$$
G(1,a)=h(a), \q \lim_{v\rightarrow 0}G(v,a)=G(0,a)=\min_{t\in(0,1)}h(t)+2t\, h({1}/{2})=\log (1/5).
$$

Combining the estimates above yields
\[
\sharp_{\mathrm{L}}(m,l;n)\approx \left(\frac{1}{\exp(G(\frac{m+l}{n},\frac{m}{m+l}))}\right)^n,\,
g_n(m,l)\approx \left(\frac{(1+\theta)/5}{\exp(G(\frac{m+l}{n},\frac{m}{m+l}))}\right)^n.
\]
Using the continuity of $g(t,v,a)$, one  gets
\begin{lemma}\label{lm:g_asym}
There exists sequence $c(n)\rightarrow 0$ such that for any integer-valued functions $m=m(n), l=l(n)$ with $|m|+|l|\leq n$, we have
\[
\sharp_{\mathrm{L}}(m,l;n)\in\left[\left(\frac{1}{\exp(G(\frac{|m|+|l|}{n},\frac{|m|}{|m|+|l|}))}-c(n)\right)^n ,\left(\frac{1}{\exp(G(\frac{|m|+|l|}{n},\frac{|m|}{|m|+|l|}))}+c(n)\right)^n\right].
\]
\end{lemma}

When $h(a)\geq\log((1+\theta)/5)$, we can find a unique $\upsilon=\upsilon(\theta,a)\in(0,1]$ such that
\begin{equation}\label{eq_def_v}
    G(\upsilon,a)=\log ((1+\theta)/5).
\end{equation}
When $h(a)<\log((1+\theta)/5)$, we define $\upsilon(\theta,\phi) = 1.$

We now continue the proof of the theorem. We will mainly work with $a=a(\phi)=|\cos(\phi)|/(|\cos(\phi)|+|\sin(\phi)|)$ instead of $\phi$. Note that $\arg(i_k,j_k)\rightarrow \phi$ {implies that $|i_k|/(|i_k|+|j_k|)\rightarrow a(\phi)$}. Without loss of generality, we assume $a\in [0, 1/2]$.

We first show the upper bounds, \eqref{eq_speed_D_upper} and \eqref{eq_speed_R_upper}, which are relatively easier. Note that when $\upsilon(\theta,\phi) =1$, \eqref{eq_speed_D_upper} and \eqref{eq_speed_R_upper} hold trivially because $R_n(m,l)=R^N_n(m,l)=0$ for $n\leq m+l$. Now assume that $\upsilon(\theta,\phi) <1$. By Claim~1, the lemma above and the definition of $\upsilon$,  when $m+l$ is large enough and $|m/(m+l)-a|$ is small enough, we have
$$
g_{k}(m,l)\leq b^{k}, \mathrm{~when~} m+l\le k\leq (m+l)(\upsilon(\theta,a)^{-1}-\vep),
$$
where $b<1$ depends only on $\theta,a(\phi)$ and $\vep$.
It is easy to use induction to show that $\DI_n(\x)\leq g_n(\x)$ and $E(I^N_n(\x)/N)\leq g_n(\x)$. Therefore,
\[\aligned
&\DD_{\lfloor (m+l)(\upsilon^{-1}-\vep)\rfloor}(m,l)=\sum_{k=m+l}^{\lfloor (m+l)(\upsilon^{-1}-\vep)\rfloor}\DI_k(m,l)\\
&\,\leq \sum_{k=m+l}^{\lfloor (m+l)(\upsilon^{-1}-\vep)\rfloor}g_k(m,l)\leq \sum_{k=m+l}^{\lfloor (m+l)(\upsilon^{-1}-\vep)\rfloor} b^k\leq b^{m+l}/(1-b)\stackrel{\mathrm{as~}m+l\rightarrow\infty }{\longrightarrow} 0.
\endaligned\]
The proof of \eqref{eq_speed_D_upper} is complete.
Similarly,
\[\aligned
E(\Dd^N_{\lfloor (m+l)(\upsilon^{-1}-\vep)\rfloor}&(m,l)/N)
=\sum_{k=m+l}^{\lfloor (m+l)(\upsilon^{-1}-\vep)\rfloor}E\left(I_k(m,l)/N\right)\\
&\leq \sum_{k=m+l}^{\lfloor (m+l)(\upsilon^{-1}-\vep)\rfloor}g_k(m,l)\leq b^{m+l}/(1-b)\stackrel{\mathrm{as~}m+l\rightarrow\infty }{\longrightarrow} 0.
\endaligned\]
The proof of \eqref{eq_speed_R_upper} is complete.

Next, we prove \eqref{eq_speed_D_lower} and \eqref{eq_speed_R_lower}. Our method is similar to the proof of \eqref{eq:conv_SIR_front_middle_inside}.  As in the proof of \eqref{eq:conv_SIR_front_middle_inside}, it would be more convenient to rotate the space counterclockwise by $\pi/4$  and work with
$$
\cL = \{(m,n)\in \mathbb{Z}^{2} \, :\, m+n \textrm{ is even } \, n\geq 0\}.
$$
Note that the cone $\{(m,l)\in \rN\times \rN: a(m+l)<m<(1-a)(m+l)\}$ in the first quadrant corresponds to $\{(x,y)\in\cL:|x|<\lambda y\}$ in $\cL$, with $\lambda=(1-2a)$.

We will define an oriented site percolation process, called $\eta_t$-system, as follows. For $t>0$, $\beta_1, \beta_2>0$, $L\in \rN$, define for each $\mathbf{b}=(b_1,b_2)\in \cL$,
\[
\aligned
C(\mathbf{b})&=(b_1\lfloor(\lambda+\beta_1) L\rfloor,b_2L)\in \cL,\\ A(\mathbf{b})&=C(\mathbf{b})+(-\lfloor(\lambda+\beta_1+\beta_2)L\rfloor,\lfloor(\lambda+\beta_1+\beta_2)L\rfloor)\times[0,L].
\endaligned
\]
For some small $\delta_0\in(0,\gamma)$ to be chosen below, assume that there are $\lceil \delta_0 N\rceil$ infected particles at $C(\mathbf{b})$ at time $0$. Consider the SIR process restricted to $A(\mathbf{b})$, namely, a particle inside $A(\mathbf{b})$ cannot infect particles outside. For notational ease, we use the same notation for this restricted SIR process as the original SIR.
We set $\eta_t(\mathbf{b})=1$ if
\begin{equation}\label{cDu84}
 \min\{\Dd_{\lfloor tL\rfloor}(C(\mathbf{b}+(1,1)),\Dd_{\lfloor tL\rfloor}(C(\mathbf{b}+(-1,1)))\}\geq \lceil {\delta_0 N}\rceil,
\end{equation}
and $\eta_t(\mathbf{b})=0$ otherwise.

From this construction, it is easy to see that if $(b_1,b_2)$ is connected to the origin in the $\eta_t$-system, then in the SIR process with initial condition $I_0^N(0,0)=\lceil \delta_0 N\rceil$, we have $\Dd^N_{\lfloor tL\rfloor b_2}(b_1\lfloor(\lambda+\beta_1) L\rfloor, Lb_2) \geq \delta_0 N$. Similarly to Lemma~\ref{lm_rect}, we have the following lemma whose proof will be given later.
\begin{lemma}\label{lm_rect2}
For any $\vep>0$, there exist  $\beta_1,\beta_2>0$ sufficiently small,  $L\in \zz{N}$ sufficiently large and $\delta_0>0$ sufficiently small such that
$$
\lim_{N\rightarrow \infty}\inf_{\lambda\in[0,1]}P(\eta_{\upsilon(\theta,(1-\lambda)/2)^{-1}+\vep}(\b)=1)=1.
$$
\end{lemma}

Note that the $\eta_{t}$-system with $t=\upsilon(\theta,(1-\lambda)/2)^{-1}+\vep$ is  $k$-dependent with $k=2\lceil 2(\beta_1+\beta_2)/\beta_1)\rceil~+~2$ (when $L$ is large). It follows from Lemma \ref{lm_rect2} that the $\eta_{t}$-system percolates with  high probability for all~$N$ sufficiently large. Moreover, similarly to the proof of Lemma~\ref{lm_rect}, from the standard result in two dimensional oriented percolation (\cite{Du84}), we know that when $N$ is sufficiently large,  for any fixed $\vep_1>0$, as long as $|b_1|<(1-\vep_1)b_2$, $\mathbf{b}=(b_1,b_2)$ is connected to the origin in the $\eta_t$-system for $t={\upsilon(\theta, (1-\lambda)/2)^{-1}+\vep}$  with high probability. In other words,  $\Dd^N_{\lfloor (\upsilon^{-1}+\vep)L\rfloor b_2}(b_1\lfloor(\lambda+\beta_1) L\rfloor,Lb_2)\geq \delta_0 N$ with high probability. From this, we see that \eqref{eq_speed_R_lower} holds for $\delta=\delta_0$ when $(i_k,j_k)_k\subset \bigcup\{C(\mathbf{b})\}$.

When $(i_k,j_k)_k\not\subset \bigcup\{C(\mathbf{b})\}$, we use an argument similar to the one in the proof of \eqref{eq:conv_SIR_front_large}. From Proposition~\ref{prop:lim_N}, it is easy to see that when $L$ is fixed, we can find some $\delta_1\leq \delta_0$ such that for the SIR process restricted to $\ze\times[0,2L]$ with initial condition $\lceil\delta_0 N\rceil\bm{\delta}_{\0}$, one has
\begin{equation}\label{eq_l7}
\lim_{N\rightarrow\infty}\sup_{\x\in[-2L,2L]\times[0,L]}P(\Dd_{5L}(\x)<\delta_1 N)\rightarrow 0.
\end{equation}
Similarly to the argument that \eqref{eq_l1}, \eqref{eq_l2} and \eqref{eq_l3} imply \eqref{eq:conv_SIR_front_large}, we have that \eqref{eq_speed_R_lower} for $\delta=\delta_0,(i_k,j_k)_k\subset \bigcup\{C(\mathbf{b})\}$ and \eqref{eq_l7} imply \eqref{eq_speed_R_lower} for $\delta=\delta_1$ and $(i_k,j_k)_k\not\subset \bigcup\{C(\mathbf{b})\}$.

Finally, we prove  the stronger version of  \eqref{eq_speed_R_lower} stated in Remark~\ref{rk:speed_strong}. In order to get the desired result, by a similar argument to \eqref{eq_l7}, we have the following:
\begin{equation}\label{eq_l8}
\lim_{N\rightarrow\infty}\sup_{\x\in[-2L,2L]\times[L,2L]}P(\Dd^-_{5L}(\x)<\delta_1 N)\rightarrow 0,
\end{equation}
where $D_n^-(i,j)$ stands for the number of particles that are infected before or at time $n$ via only those edges in the region {$\{(x,y)\in \cL:\ 0< y < j\}$}
except possibly with one endpoint at either $(0,0)$ or $(i,j)$. To prove the stronger version of  \eqref{eq_speed_R_lower}, for any $(i_k,j_k)\in\cL$, find an $\x\in \bigcup\{C(\mathbf{b})\}$ such that $(i_k,j_k) \in \x+[-2L,2L]\times [L,2L]$.
By \eqref{eq_l8} and the proof above for \eqref{eq_speed_R_lower}, and translating the result back to original orientation, we get the stronger version of \eqref{eq_speed_R_lower} with all edges lying in $\{(x,y)\in \ze^2:\, 0\leq x+y<i_k+j_k\}$
except possibly with one endpoint at $(i,j)$.
We need to further rule out those edges on the line $\{(x,y)\in \ze^2:\ x+y=0\}$. We can do so by changing the starting point to $(1,0)$. Indeed, with a high probability, $I_1(1,0)/N\geq \gamma_0$ for some $\gamma_0$ depending only on~$\gamma$. Then regarding $(1,0)$ as the starting point and using the argument above, we get the desired stronger version.

\end{proof}

\begin{proof}[Proof of Lemma~\ref{lm_rect2}]
The proof is similar to that of Lemma~\ref{lm_rect}. Write $A=A(\0)=(-\lfloor(\lambda+\beta_1+\beta_2)L\rfloor,\lfloor(\lambda+\beta_1+\beta_2)L\rfloor)\times[0,L]$. Let $Y_0(\x)=\gamma \bm{\delta}_{(0,0)}(\x)$ and define $Y_n:\cL\rightarrow \re $ recursively:
\[
\aligned
&Y_n(m,l)\\
=&\left\{ \begin{array}{cc}
    (1-\sum_{0}^{n-1}Y_i(m,l))(1-\exp(-\frac{1+\theta}{5}(\wt{Y}_{n-1}(m,l))) , & \mathrm{if}\,(m,l)\in A;\\
   0,  & \mbox{otherwise}.
\end{array}\right.
\endaligned
\]
Note that in $\cL$, $\wt{f}(m,l)=f(m+1,l+1)+f(m-1,l+1)+f(m+1,l-1)+f(m-1,l-1)$.

As before, we will analyze
\[
g'_n(m,l):=\left.\frac{\partial Y_n}{\partial \gamma}\right|_{\gamma=0}(m,l).
\]
We have {$g{_0}'(0,0)=1$}, and for $(m,l)\in A$,
\[
g'_n(m,l)=\frac{1+\theta}{5}\wt{g'_{n-1}}(m,l).
\]
By induction, one gets

\ul{Claim~1}:
\[
g'_n(m,l)=\left(\frac{1+\theta}{5}\right)^n\times
\sharp_{\mathrm{L}}(m,l;n;A),
\]
where $\sharp_{\mathrm{L}}(m,l;n;A)$ is the number of LRW paths from $(0,0)$ to $(m,l)$ with $n$ steps inside $A$.
Write $\sharp_{\mathrm{{\mathrm{S}}}}(m,l;n;A)$ as the corresponding number for SRW.

Similarly to the proof of Lemma~\ref{lm_rect}, it suffices to show that when $L$ is large, $n=\lfloor(\upsilon^{-1}+\vep)L\rfloor$, for any ${\lambda}\in[0,1]$,
\begin{equation}\label{eq_g_prime}
 g'_n(m,L)>1, \,\mathrm{~with~}m=\lfloor (\lambda+\beta_1)L\rfloor.
\end{equation}

We aim to show $\sharp_{\mathrm{L}}(m,L;n;A)\approx \sharp_{\mathrm{L}}(m,L;n)$ for some small and fixed $\beta_1,\beta_2$. If so, then by the asymptotics of $\sharp_{\mathrm{L}}(m,L;n)$, namely, Lemma~\ref{lm:g_asym}, the above inequality holds for large $L$ and we finish the proof. Obviously, $\sharp_{\mathrm{L}}(m,l;n;A)\leq \sharp_{\mathrm{L}}(m,l;n)$. We need to show the other direction. In principal, $\sharp_{\mathrm{L}}(m,l;n;A)$ can be written as the sum of $\sharp_{\mathrm{L}}$'s without restriction by using the reflection principle. However the process will be too involved. Here we separate it into two steps. We will first consider the number of paths in $A_1=\ze\times [0,L]$.

Using the reflection principle, we have
\[\aligned
\sharp_{\mathrm{{S}}}&(m,l;n;A_1)=\sharp_{\mathrm{S}}(m,l;n)-2\sharp_{\mathrm{S}}(m,l+2;n)+\sharp_{\mathrm{S}}(m,l+4;n)+\sharp_{\mathrm{S}}(m,3l+4;n)\\
&\geq \sharp_{\mathrm{S}}(m,l;n)-2\sharp_{\mathrm{S}}(m,l+2;n)+\sharp_{\mathrm{S}}(m,l+4;n).
\endaligned\]
Note that for the lattice $\cL$, \eqref{eq:nbr_srw} and the aysmptotics in Lemma~\ref{lm:g_asym} become (when $m,l\geq0$)
\[ \sharp_{\mathrm{S}}(m,l;n)=\binom{n}{\frac{n-l}{2}}\binom{n}{\frac{n-m}{2}};
\]
\begin{equation}\label{eq_num_LRW_L}
  \sharp_{\mathrm{L}}(m,l;n) \approx \left(\exp^{-1}(G(\frac{m \vee l}{n}, \frac{1-\frac{m\wedge l}{m\vee l}}{2}))\right)^{n},
\end{equation}
where we write $m\vee l=\max\{m,l\}, m\wedge l=\min\{m,l\}$.

Using this,  we  get
\[\aligned
&\frac{\sharp_{\mathrm{S}}(m,l;n)-2\sharp_{\mathrm{S}}(m,l+2;n)+\sharp_{\mathrm{S}}(m,l+4;n)}{\sharp_{\mathrm{S}}(m,l;n)}\\
&=1-2\frac{n-l}{n+l+2}+\frac{(n-l)(n-l-2)}{(n+l+2)(n+l+4)}\geq \left(\frac{l}{n}\right)^2-O({1}/{n}).
\endaligned\]
Hence,
\[\frac{\sharp_{\mathrm{S}}(m,l;n;A_1)}{\sharp_{\mathrm{S}}(m,l;n)}\geq \left(\frac{l}{n}\right)^2-O({1}/{n}).
\]
Using the relation between the number of LRW and that of SRW, namely, \eqref{eq:L_to_S}
 one  gets
\begin{equation}\label{eq_L_to_R}
   \frac{\sharp_{\mathrm{L}}(m,l;n;A_1)}{\sharp_{\mathrm{L}}(m,l;n)}\geq \left(\frac{l}{n}\right)^2-O({1}/{{n}})= \left(\frac{l}{n}\right)^2-O({1}/{{l}}).
\end{equation}

Note that $\sharp_{\mathrm{L}}(m,l;n)$ is decreasing in $m$ in $[0, n]$. Therefore, with $k=\lfloor(\lambda+\beta_1+\beta_2)L\rfloor, n=\lfloor (\upsilon^{-1}+\vep)L\rfloor$, when $m=\lfloor (\lambda+\beta_1)L\rfloor$,
\[\aligned
&\sharp_{\mathrm{L}}(m,L;n;A)=\sharp_{\mathrm{L}}(m,L;n;A_1)-\sum_{i=\pm m}\sharp_{\mathrm{L}}(2k+i,L;n;A_1)\\
&\quad\quad\quad\quad\quad\quad+\sum_{i=\pm m}\sharp_{\mathrm{L}}(4k+i,L;n;A_1)\\
&\geq \left(\left(\frac{L}{n}\right)^2-O({1}/{L})\right)\sharp_{\mathrm{L}}(m,L;n)-2\sharp_{\mathrm{L}}(2k- m,L;n)\\
&\geq \left(\left(\frac{L}{n}\right)^2-O({1}/{L})\right)\sharp_{\mathrm{L}}(\lfloor (\lambda+\beta_1)L\rfloor,L;n)-2\sharp_{\mathrm{L}}(\lfloor (\lambda+\beta_1+2\beta_2)L\rfloor-2,L;n)\\
&\geq C\left(\exp\left(- G\left(\frac{(\lambda+\beta_1)\vee 1}{\upsilon^{-1}+{\vep}},\frac{1-\frac{(\lambda+\beta_1)\wedge 1}{(\lambda+\beta_1)\vee 1}}{2}\right)\right)-c(n)\right)^n\\
  &\quad\quad\quad-\left( \exp\left(- G\left(\frac{(\lambda+\beta_1+2\beta_2)\vee 1}{\upsilon^{-1}+{\vep}},\frac{1-\frac{(\lambda+\beta_1+2\beta_2)\wedge 1}{(\lambda+\beta_1+2\beta_2)\vee 1}}{2}\right)\right)+c(n)\right)^n,
\endaligned
\]
where in the last step we use Lemma~\ref{lm:g_asym} and \eqref{eq_num_LRW_L}. Note that $G(\upsilon, a)=G(\upsilon, (1-\lambda)/2)\leq \log((1+\theta)/5)$ and $G(v,\cdot)$ is strictly decreasing in $v$. By the continuity of $G(\cdot,\cdot)$ and $\upsilon(\theta, a)$, for any $\eps>0$, we can find $\beta_1,\beta_2$ small enough such that for some $\eps_1>0$ and when $L$ is large,  the above term is bigger than $\left({5}/{(1+\theta)}+\eps_1\right)^{n}$ for all ${\lambda}\in[0,1]$. Therefore, by Claim 1, \eqref{eq_g_prime} holds.
\end{proof}

\section{Proportion of Recovery}\label{Sec:prop_recover}

We prove Theorems \ref{thm:surv_prob}, \ref{thm:speed}(ii), \ref{thm:ult_inf_prop_2} and \ref{thm:ult_inf_prop_1} in this section.

We consider the SIR process and the limiting process simultaneously. The default setting is that their initial conditions are consistent, namely,
\begin{equation}\label{eq:ini_conv_2}
\frac{I^N_0(\x)}{N}\stackrel{}{\longrightarrow} \DI_0(\x),\q\mbox{ for all } {\x}\in\ze^2 \mbox{ as } N\rightarrow\infty.
\end{equation}
We also assume that the initial condition is nontrivial, i.e., $\DI_0\not\equiv 0$. Recall that we write $\DD_n=\DR_{n+1}$ and $\Dd_n=R_{n+1}$, $\DR_\infty(\x)=\lim_{n\rightarrow\infty}\DR_n(\x)$, and $R_\infty(\x)=\lim_{n\rightarrow\infty}R_n(\x)$.

\begin{thm}\label{thm:final_prop}
\begin{enumerate}[(i)]
    \item Consider the deterministic system. For any nontrivial initial condition, $\{\DR_\infty(\x)\}_{\x\in\ze^2}$ satisfies that
    \begin{equation}\label{eq:final_bd}
     \frac{\DR_\infty(\x)-\DI_0(\x)}{1-\DI_0(\x)}>\iota, \mbox{ for all } \x \mbox{ such that }\DI_0(\x)<1;
     \end{equation}
     and
    \begin{equation}\label{eq:final_S}
    \aligned
     &\frac{\DR_\infty(\x)-\DI_0(\x)}{1-\DI_0(\x)}\\
     =&1-\exp\left(-\frac{1+\theta}{5}\wt{\DR}_\infty(\x)\right), \mbox{ for all } \x\mbox{ such that }\DI_0(\x)<1.
     \endaligned
     \end{equation}
Moreover, if the initial condition is either $\DI_0(m,l)=\gamma\cdot\bm{\delta}_{m+l=0}$ or $\DI_0(m,l)=\gamma\cdot\bm{\delta}_{(0,0)}$ for some $\gamma\in(0,1]$, then we have
 \begin{equation}\label{eq:lim_S}
 \lim_{(m,l)\rightarrow \infty:m,l\geq 0}\DR_\infty(m,l)=\iota.
 \end{equation}
    \item Consider the SIR system. For any initial condition satisfying \eqref{eq:ini_conv_2} with $\DI_0\not\equiv 0$, we have
    \begin{equation}\label{eq:D_to_R}
   \frac{R^N_\infty(\x)}{N}\toop  \DR_\infty(\x),  \quad \mbox{ for all } \x\in \ze^2  \mbox{ as } N\rightarrow\infty.
   \end{equation}
     Moreover, if the initial condition is either $I^N_0(m,l)=\lfloor\gamma N\rfloor\bm{\delta}_{m+l=0}$ or $I^N_0(m,l)=\lfloor\gamma N\rfloor\bm{\delta}_{(0,0)}$ for some $\gamma\in(0,1]$, then for any $\vep>0$, there exists~$N_0$ such that for all $N\geq N_0$,
    \begin{equation}\label{eq:R_final}
    \limsup_{(m,l)\rightarrow \infty:m,l\geq 0}P\left(\left|\frac{R^N_\infty(m,l)}{N}-\iota\right|>\vep\right)<\vep.
    \end{equation}

\end{enumerate}
\end{thm}

Note that Theorem~\ref{thm:ult_inf_prop_1} is contained in Theorem~\ref{thm:final_prop}(ii).
The proof of this theorem is divided into several parts. The original recursive formula \eqref{eq:lim_X} is for~$\DI_n$, which is not convenient when addressing $\DR_n$ or $\DD_n$. We start by giving a recursive formula for $\DD_n$.
\begin{lemma}\label{lm:S_recur}
The process $\{\DD_n(\x)\}_{n\in \rN}$ satisfies the following recursive equation:
\begin{equation}\label{eq:S_recur}
    \aligned
     &\frac{\DD_{n+1}(\x)-\DI_0(\x)}{1-\DI_0(\x)}\\
    =&1-\exp\left(-\frac{1+\theta}{5}\wt{\DD}_n(\x)\right), \mbox{ for all } \x \mbox{ satisfying } \DI_0(\x)<1.
    \endaligned
\end{equation}
Letting $n$ go to infinity yields \eqref{eq:final_S}.
\end{lemma}
\begin{proof}
One can show \eqref{eq:S_recur} by induction. Here we give a probabilistic proof, which is more intuitive.

For any fixed $\mathbf{x}$ satisfying $\DI_0(\x)<1$, consider the following adjusted percolation model. There are $N$ particles at each site  $\y\in \ze^2$ except that at $\x$, there are $N+1$ particles. Other settings are the same as before: any two particles have an edge if and only if they are in the same or in the neighbor site(s); each edge is open with probability $P_N=(1+\theta)/(5N)$. It is easy to see that this adjusted percolation model shares the same limiting process as the original one, and Proposition~\ref{prop:lim_N} also holds for the adjusted model.

Note that the underlying graph of the original SIR model can be embedded into the adjusted one, and we can couple the two models in a same probability space. Write $V$ for the extra vertex at $\x$. Write $A_N$ for the set of initial infected particles, $Y_0, Y_1,\dots,Y_4$ for the numbers of particles at  $\x$ and the four neighboring sites that are connected to $A_N$ within $n$ edges, without passing through $V$. Conditionally on the values of $\{Y_i\}$, we have
$$
P(V\mathrm{~is~ connected~to~}A_N\mathrm{~within~}n+1\mathrm{~steps})=1-(1-P_N)^{\sum_{i=0}^4Y_i}.
$$
Letting $N\rightarrow\infty$, by Proposition~\ref{prop:lim_N}, we get the assertion.
\end{proof}

\begin{lemma}\label{lm:IC_extreme}
If the initial condition is $\DI_0(m,l)=\bm{\delta}_{m+l=0}$,
then
\[
\DR_\infty(m,0)
:= \DR_\infty(m)\searrow \iota.
\]
Moreover, if the initial condition is $\DI_0=\gamma\bm{\delta}_{\0}$ for some $\gamma\in{(0,1]}$, then
\[
\lim_{\x\rightarrow \infty}\DR_\infty(\x)=\iota.
\]
\end{lemma}
\begin{proof}
First assume that the initial condition is $\bm{\delta}_{m+l=0}$. Using \eqref{eq:S_recur}, by induction, one can see $\DD_n(m,0)\geq \DD_{n}(m+1,0)$. Letting $n$ go to infinity, we see that~$\DR_\infty(m)$ is decreasing in $m$. Hence, $\lim_{m\rightarrow\infty} \DR_\infty(m)$ exists. We denote it by $s$. Letting $\x\rightarrow \infty$ in \eqref{eq:final_S}, we see that $s$ satisfies the same equation as~\eqref{eq:iota}, the equation for $\iota$. Moreover, by Theorem~\ref{thm_speed}, $s\geq \delta$ for some positive $\delta$. Therefore, $s=\iota$.

The second assertion can be proved in a similar way. Using the recursive equation \eqref{eq:S_recur} and induction, we can show that $\DD_n(m,l)\geq \max\{\DD_n(m+1,l),\DD_n(m,l+1)\}$ when $m,l\geq 0$. Therefore,
\[
\DR_\infty(m,l)\geq \max\{\DR_\infty(m+1,l),\DR_\infty(m,l+1)\}, \mbox{ for all } m,l\geq 0.
\]
Hence, for any $(m,l)\in\ze^2$, $\lim_{k\rightarrow \infty} \DR_\infty(m+k,l+k)$ exists, denoted by $s(m,l)$. Because $\DR_\infty(m+k,l+k)\geq \DR_\infty(m+k+1,l+k)\geq \DR_\infty(m+k+1,l+k+1)$, we get that $s(m,l)=s(m+1,l)$. Similarly, one can get that $s(m,l)=s(m,l+1)$. Therefore, all $s(m,l)$'s are equal, say, to $s$. By the recursive equation, we see that $s$ satisfies the same equation as \eqref{eq:iota}. Theorem~\ref{thm_speed} guarantees that $s\neq 0$,  hence $s$ must be $\iota$. Combining the conclusions above,  we see that $\DR_\infty(m,l)\geq \iota$. On the other hand, $\DR_\infty(m,l)$ is not more than the corresponding $\DR_\infty(m,l)$ when the initial condition is $\bm{\delta}_{x+y=0}$. By the first assertion, we get the upper bound and finish the proof of the second assertion.
\end{proof}

We now prove \eqref{eq:final_bd}.
\begin{proof}[Proof of \eqref{eq:final_bd}]
For any nontrivial initial condition,
by the previous lemma and monotonicity, we have that $\DR_\infty(\x)\geq \iota$ for all $\x\in \ze^2$. Plugging this into \eqref{eq:final_S} yields
\begin{equation}\label{eq:bd_R}
     \frac{\DR_\infty(\x)-\DI_0(\x)}{1-\DI_0(\x)}\geq\iota, \mbox{ for all } \x \mbox{ such that } \DI_0(\x)< 1.
\end{equation}

We need to further show that ``='' cannot hold in \eqref{eq:bd_R}. Assume otherwise  that for some $\x\in \{\x:\DI_0(\x)\neq 1\}$, the above inequality becomes equality. Then by \eqref{eq:final_S} and \eqref{eq:bd_R}, we get that for any {$\y$ with $||\y-\x||_1=1$}, we have
\[
 \DR_\infty(\y)=\iota ,\mbox{ and }\DI_0(\y)=0.
\]
Repeating this argument yields that the above equalities hold for all $\z\in \ze^2$, which contradicts to the assumption that $\DI_0$ is nontrivial. Therefore, \eqref{eq:bd_R} can not take ``$=$'' for any $\x\in\ze^2$.
\end{proof}

\begin{lemma}\label{prop:D_to_R}
If the initial condition satisfies that
\begin{equation}\label{eq:cd_final}
 \inf_{{\x}\in\ze^2}\DR_\infty(\x)>\theta/(1+\theta),
 \end{equation}
then the conclusion \eqref{eq:D_to_R} holds.
\end{lemma}
\begin{remark}
Note that $\iota>\theta/(1+\theta)$. One can see this by showing that $t<1-\exp(-(1+\theta)t)$ for $t=\theta/(1+\theta)$. From this, the proof of \eqref{eq:D_to_R} is complete once we prove the lemma above.
\end{remark}

\begin{proof}[Proof of Lemma~\ref{prop:D_to_R}]
It suffices to show that for any fixed $\x\in \ze^2$, $\vep>0$ and $n_1\in\zz{N}$, there is an $n>n_1$ such that when $N$ is sufficiently large, we have
\begin{equation}\label{eq_P7_1}
    E\left( \frac{R_\infty^N(\x)-\Dd^N_n(\x)}{N}\right) <\vep.
\end{equation}

For any  $K\in\zz{N}$, write $\mathcal{C}_K(\x)=\{\z\in\ze^2:||\x-\z||_1\leq K\}$ and $\partial\mathcal{C}_K(\x)=\{\z\in\ze^2:||\x-\z||_1= K+1\}$.
By \eqref{eq:cd_final}, we can find a $\delta>0$ such that $\inf_{{\x}\in\ze^2}\DR_\infty(\x)>(\theta+\delta)/(1+\theta)$. For any $K$, we can find an $n_0$ such that
\[
\DD_{n_0}(\y)>\frac{\theta+\delta}{1+\theta}, \mbox{ for all } \y\in \mathcal{C}_K(\x).
\]
Therefore, as $N\to\infty$, with a high probability we have that
\[
\frac{\Dd_{n_0}^N(\y)}{N}>\frac{\theta+\delta}{1+\theta},\mbox{ for all } \y\in \mathcal{C}_K(\x).
\]
By \eqref{eqn:X_R}, when $k\geq n_0$ and $ y\in\mathcal{C}_K{(\x)}$,
\[\aligned
E \left(\frac{I^N_{k+1}(\y)}{N}\right)
&=\left(1-\frac{\Dd^N_k(\y)}{N}\right)E\left(1-\left(1-\frac{1+\theta}{5N}\right)^{\wt{I}^N_k(\y)}\right)\\
&\leq \frac{1-\delta}{1+\theta}\frac{1+\theta}{5}\frac{E\wt{I}_k^N(\y)}{N}= \frac{1-\delta}{5}\frac{E\wt{I}_k^N(\y)}{N}.
\endaligned
\]
By induction, we get that when $y\in\mathcal{C}_K{(\x)}$,
\[\aligned
E\left(\frac{I^N_{n_0+m}(\y)}{N}\right)
&\leq\sum_{\mathbf{z}\in\mathcal{C}_K{(\x})}E\left(\frac{I^N_{n_0}(\mathbf{z})}{N}\right)\sharp_{\mathrm{L}}(\y,\z;m;\mathcal{C})\left(\frac{1-\delta}{5}\right)^m\\
&+\sum_{\mathbf{z}\in\partial\mathcal{C}_K{(\x})}\sum_{i=0}^{m}E\left(\frac{I^N_{n_0+i}(\mathbf{z})}{N}\right)\sharp_\mathrm{L}(\y,\z;m-i;\mathcal{C})\left(\frac{1-\delta}{5}\right)^{m-i},
\endaligned
\]
where we write $\sharp_{\mathrm{L}}(\y,\z;j;\mathcal{C})$ for the number of  LRW paths of length $j$ from $\y$ to $\z$ inside $\mathcal{C}_K{(\x)}$  (except the last step when $\z\in\partial\mathcal{C}_K{(\x})$).

Obviously, $\sharp_{\mathrm{L}}(\y,\z;m;\mathcal{C})\leq 5^m$ and for $i\leq K, \mathbf{z}\in\partial\mathcal{C}_K{(\x})$, $\sharp_{\mathrm{L}}(\x,\z;i;\mathcal{C})=0$. Therefore,
\[
\aligned
E\left(\frac{I^N_{n_0+m}(\x)}{N}\right)
&\leq \sum_{\mathbf{z}\in\mathcal{C}_K{(\x})}(1-\delta)^m+\sum_{\mathbf{z}\in\partial\mathcal{C}_K{(\x})}\sum_{i=0}^{m-K} E\left(\frac{I^N_{n_0+i}(\mathbf{z})}{N}\right)(1-\delta)^{m-i}\\
&\leq (5K^2)(1-\delta)^m+\sum_{\mathbf{z}\in\partial\mathcal{C}_K{(\x})}\sum_{i=0}^{m-K} E\left(\frac{I^N_{n_0+i}(\mathbf{z})}{N}\right)(1-\delta)^{m-i}.
\endaligned
\]
It follows that for any $n_2>0,$
\[
\aligned
&E\left(\frac{\sum_{m\geq n_2} I^N_{n_0+m}(\x)}{N}\right)\\
\leq & 5K^2\sum_{m\geq n_2}(1-\delta)^m+\sum_{\mathbf{z}\in\partial\mathcal{C}_K{(\x})}\sum_{m\geq n_2}\sum_{i=0}^{m-K} E\left(\frac{I^N_{n_0+i}(\mathbf{z})}{N}\right)(1-\delta)^{m-i}\\
\leq & 5K^2\frac{(1-\delta)^{n_2}}{\delta}+\sum_{\mathbf{z}\in\partial\mathcal{C}_K{(\x})}\sum_{i\geq0}E\left(\frac{I^N_{n_0+i}(\mathbf{z})}{N}\right)\sum_{m\geq \max(n_2,i+K)}(1-\delta)^{m-i}\\
\leq & 5K^2\frac{(1-\delta)^{n_2}}{\delta}+8K\frac{(1-\delta)^{K}}{\delta}.
\endaligned
\]
Letting $K$ be large enough first and then $n_2$ be large enough, we  get \eqref{eq_P7_1}.
\end{proof}

\begin{proof}[Proof of \eqref{eq:R_final}]
For the upper bound, it is sufficient to consider the initial condition $N\bm{\delta}_{x+y=0}$. By \eqref{eq:lim_S}, we can find an $n_0$ such that $\DR_\infty(n_0,0)<\iota+\vep/2$. It follows from \eqref{eq:D_to_R} that there is an $N_0$ such that when $N\geq N_0$,
\[
P\left(\frac{R_\infty^N(n_0,0)}{N}>\iota+\vep\right)<\vep.
\]
Note that for any $(m,l)$ with $m+l>n_0$, $R_\infty^N(m,l)$ is stochastically dominated by $R_\infty^N(n_0,0)$. Therefore, we have
\[
P\left(\frac{R_\infty^N(m,l)}{N}>\iota+\vep\right)\leq P\left(\frac{R_\infty^N(n_0,0)}{N}>\iota+\vep\right)<\vep.
\]

Next, we prove the lower bound of \eqref{eq:R_final}. For this,  it suffices to consider the initial condition $\lfloor\gamma N\rfloor\bm{\delta}_{\0}$. We consider the corresponding system in the half space $\HS=\{(m,l)\in\ze^2:m+l\geq 0\}$. Let $\SHS_0(\x)=\gamma\bm{\delta}_{(0,0)}(\x)$ and define $\SHS_n:\ze^2\rightarrow \re$ recursively as follows:
\begin{equation}\label{eq_Q_hs}
\SHS_n(m,l)=\left\{ \begin{array}{cc}
    1-\exp(-\frac{1+\theta}{5}\wt{\SHS}_{n-1}(m,l)) , & \mathrm{if}\,(m,l)\in \HS\setminus\{(0,0)\};\\
    \gamma+(1-\gamma)(1-\exp(-\frac{1+\theta}{5}\wt{\SHS}_{n-1}(m,l)) ), & \mathrm{if}\,(m,l)=(0,0);\\
   0,  & \mbox{otherwise}.
\end{array}\right.
\end{equation}
Let $\SHS(m,l)=\lim_{n\rightarrow\infty}\SHS_n(m,l)$.
Analogously to \eqref{eq:middle_inside_cone}, we have
\begin{prop}\label{prop:conv_SHS}
For any $b>1,
\gamma\in (0,1]$,
\begin{equation}\label{eq_Q_SHS}
\lim_{n\rightarrow \infty} \sup_{m:\,-bn\leq m\leq bn}|\SHS(m,n-m)-\iota|=0.
\end{equation}
\end{prop}
The proof of the above proposition is  similar to the one of \eqref{eq:middle_inside_cone} and will be given {after the current proof}.

Consider the $N$-percolation model. In the proof of \eqref{eq_speed_R_lower}, we have shown that under the initial condition  $I^N_0\sim \gamma N  \bm{\delta}_{\0} $, for any $\vep>0$, there exist positive constants $\delta$ and $a$ such that when $N$ is large,
\begin{equation}\label{eq_v1}
  \limsup_{(m,l)\rightarrow\infty:m,l\geq 0}P(\Dd^N_{a(m+l)}(m,l)/N<\delta)<\vep.
\end{equation}
In fact, we have shown the above inequality holds when we replace $\Dd^N_{a(m+l)}(m,l)$ with  $\Dd^{N,-}_{a(m+l)}(m,l)$, the number of particles that are connected to the $\lfloor\gamma N\rfloor$ initial infected particles at the origin via no more than $a(m+l)$ edges inside the lower half space $\{(x,y):x+y < l+m\}$.

On the other hand, by Proposition \ref{prop:conv_SHS} and Proposition \ref{prop:lim_N},  for any $\vep>0$, there exist $n_0$ and $n_1$ such that when $N$ is large enough,
\begin{equation}\label{eq_v2}
P\left(\frac{\Dd^{N,+}_{n_0}({n_1,0})}{N}<\iota-\vep\right)<\vep,
\end{equation}
where $\Dd^{N,+}_{n_0}({n_1,0})$ stands for the number of particles at $(n_1,0)$ that are connected to the $\lceil\delta N\rceil$ initial infected particles at the origin via no more than $n_0$ edges in $\{(x,y):x+y\geq 0\}$.

Combining the results in the last two paragraphs, by connecting paths in the $N$-percolation model, one can get the lower bound of \eqref{eq:R_final}. By symmetry, we assume $m\geq l$. In fact, \eqref{eq_v1} for $\Dd^{N,-}_{a{(m_0+l_0)}}(m_0,l_0)$ says that, with high probability the number of particles at $(m_0,l_0)$ that are connected to the $\lfloor\gamma N\rfloor$  infected particles at the origin via  open edges inside the lower half space $\{(x,y):x+y< m_0+l_0\}$ is at least $\lceil\delta N\rceil$. On the other hand, it follows \eqref{eq_v2} that from  those $\lceil\delta N\rceil$ infected particles at $(m_0,l_0)$, with high probability, there are at least $\lceil(\iota-\vep) N\rceil$ particles at site $(m_0+n_1,l_0)$ that are are connected to the $\lceil\delta N\rceil$ infected particles at $(m_0,n_0)$ via  open edges inside the upper half space $\{(x,y):x+y\geq m_0+l_0\}$. Hence, with high probability,  the total number of the infected particles at site $(m_0+n_1,l_0)$ is at least $\lceil(\iota-\vep) N\rceil$. Note that the two events are independent because they depend on the edges that are lying on different sides of $x+y=m_0+{l_0}$.
\end{proof}

\begin{proof}[Proof of Proposition~\ref{prop:conv_SHS}]
The proof is very similar to the one of \eqref{eq:middle_inside_cone}. Note that by \eqref{eq:lim_S}, we only need to show the lower bound. Moreover, by monotonicity, we can assume that $\gamma$ is sufficiently small. Analogous to Claim~1 in the proof of Proposition \ref{prop:middle_conv_front_origin_IC}, we have

\ul{Claim}: For any $\mathbf{a}\in\HS$ and $n\in \rN$ with
\begin{equation}\label{eq:good_pt_1}
\SHS_n(\mathbf{a})>\SHS_0(0,0)=\gamma \mathrm{~and~}
\end{equation}
\begin{equation}\label{eq:good_pt_2}
1-\exp\left(-\frac{1+\theta}{5}\SHS_n(\mathbf{a}-(1,0))\right)\geq \gamma,
\end{equation}
we have
\begin{equation}\label{eq:good_rs}
 \SHS_{n+k}(\x+\mathbf{a})>\SHS_k(\x), \mbox{ for any }\x\in \HS \mbox{ and }k\in\rN.
\end{equation}


As before, in order to find some $\mathbf{a}\in\HS$ and $n\in \rN$ satisfying the  conditions in the claim, we need to analyze the derivative respect to $\gamma$.  Note that
\[
\left.\frac{\partial }{\partial \gamma}\right|_{\gamma=0}\left(1-\exp(-\frac{1+\theta}{5}\SHS_n((i,m-i))\right)=\frac{1+\theta}{5}\left.\frac{\partial\SHS_n(i,m-i)}{\partial\gamma}\right|_{\gamma=0}.
\]
Hence, it suffices to analyze $\partial\SHS_n/\partial\gamma$.  It turns out that it has a similar recursive formula as $g'_n$ in the proof of Lemma~\ref{lm_rect2}:
\[
\left.\frac{\partial \SHS_n}{\partial\gamma}\right|_{\gamma=0}(m,l)=\left\{ \begin{array}{cc}
    \frac{1+\theta}{5}\left.\frac{\partial \wt{\SHS_{n-1}}}{\partial\gamma}\right|_{\gamma=0}(m,l) , & \mathrm{if}\,(m,l)\in \HS\setminus\{(0,0)\};\\
    1+\frac{1+\theta}{5}\left.\frac{\partial \wt{\SHS_{n-1}}}{\partial\gamma}\right|_{\gamma=0}(m,l), & \mathrm{if}\,(m,l)=(0,0);\\
   0,  & \mbox{otherwise}.
\end{array}\right.
\]
Therefore, we get that
\[\aligned
\left.\frac{\partial}{\partial\gamma}\right|_{\gamma=0}&\SHS_n(m,l)\geq\left(\frac{1+\theta}{5}\right)^n\sharp_\mathrm{L}(m,l;n;\HS)\\
&=\left(\frac{1+\theta}{5}\right)^n(\sharp_\mathrm{L}(m,l;n)-\sharp_\mathrm{L}(m+1,l+1;n))\\
&\geq C(b)\left(\frac{1+\theta}{5}\right)^n\sharp_\mathrm{L}(m,l;n),
\endaligned
\]
{where we} write $\sharp_\mathrm{L}(m,l;n;\HS)$ for the number of LRW paths from $(0,0)$ to $(m,l)$ with $n$ steps inside $\HS$,  in the second line we use the reflection principle,
and~$C(b)$ is a constant that is independent of $(m,l)$ as long as  $|m|\leq b|m+l|$ {and $1\leq n/(m+l)$ is bounded from above. The last inequality can be proved in a very similar way to \eqref{eq_L_to_R}}.

By Lemma~\ref{lm:g_asym}, we can find  $n_0$ and $m_0$ large enough such that when $n=n_0$ and $m\in[m_0,3m_0]$,
\[
\inf_{i:|i|\leq bm}\left.\frac{\partial\SHS_n(i,m-i)}{\partial\gamma}\right|_{\gamma=0}\gg 1, \mathrm{~and~}\inf_{i:|i|\leq bm}\left.\frac{\partial\SHS_n(i-1,m-i)}{\partial\gamma}\right|_{\gamma=0}\gg1.
\]
Therefore, there exists a positive $\gamma$ such that \eqref{eq:good_pt_1} and \eqref{eq:good_pt_2} hold for all $\mathbf{a}\in \{(i,j)||i|\leq b(i+j),\,m_0\leq i+j\leq 2m_0+1\}:= A$ and $n=n_0$. By the Claim,~\eqref{eq:good_rs} holds for such $\mathbf{a}$ and $n_0$. Letting $k\rightarrow\infty$ in \eqref{eq:good_rs} we get
\begin{equation}\label{eq_SHS_mono}
  \SHS(\x+\mathbf{a})\geq \SHS(\x), \,\mbox{ for all } \x\in\HS, \mathbf{a}\in A.
\end{equation}
Hence, we can define,  for any $\x\in\HS,\mathbf{a}\in A$,
\[
\lim_{n\rightarrow\infty}\SHS(\x+n\mathbf{a}):= h(\x,\mathbf{a}).
\]
For $\mathbf{a}_1,\mathbf{a}_2\in A$ with $\mathbf{a}:=\mathbf{a}_1+\mathbf{a}_2\in A$,  \eqref{eq_SHS_mono} implies
\[
h(\x,\mathbf{a})\leq h(\x+\mathbf{a}_1,\mathbf{a})\leq h(\x+\mathbf{a}_1+\mathbf{a}_2,\mathbf{a})= h(\x,\mathbf{a}).
\]
Hence,
\[
h(\x,\mathbf{a})=h(\x+\mathbf{a}_1,\mathbf{a}).
\]
Similarly we can get that
\[
h(\x,\mathbf{a})=h(\x+\mathbf{a}_2,\mathbf{a}).
\]
By choosing $\{\mathbf{a}_1,\mathbf{a}_2\}$ so that they generate $\ze^2$, we see that for $h(\x,\mathbf{a})$ is independent of $\x$. Letting $n$ in {\eqref{eq_Q_hs}} go to infinity first and then {$\x=(m,l)$} go to infinity, we get that $h(\x,\mathbf{a})=\iota$. From this, one can get the lower bound, in a similar way that \eqref{eq:side_1} implies the lower bound of \eqref{eq:middle_inside_cone}.
\end{proof}

We give the following proposition about the uniqueness of the solution to equation \eqref{eq:final_S}.
\begin{prop}\label{prop:uniq_sol_eq}
For  any nontrivial $\DI_0:\ze^2\rightarrow[0,1]$, there exists a unique solution $f:\ze^2\rightarrow [0, 1]$ to the following difference equation
\begin{equation}\label{eq:f_recur}
     f(\x)=\DI_0(\x)+(1-\DI_0(\x))\left(1-\exp\left(-\frac{1+\theta}{5}\wt{f}_n(\x)\right)\right), \mbox{for all }  \x\in \ze^2.
\end{equation}
Moreover, for such solution $f$, we have
\begin{equation}\label{eq:relation_solu_iota}
    \frac{f(\x)-\DI_0(\x)}{1-\DI_0(\x)}>\iota,\, \mbox{for all } \x\in\{\x:\DI_0(\x)\neq 1\}, \mbox{ and }\lim_{\x:\mathbf{d}(\x,\mathrm{Supp}(\DI_0))\rightarrow \infty}f(\x)=\iota,
\end{equation}
where $\mathrm{Supp}(\DI_0)$ represents the support of $\DI_0$, namely, the set $\{\x:\DI_0(\x)\neq 0\}$, and for any $\x\in\ze^2$ and $A\subseteq \ze^2,$ $\mathbf{d}(\x,A):=\inf_{\y\in A}||\x-\y||_1$.
\end{prop}
\begin{proof}
Define $u_0=\DI_0, v_0\equiv 1$ and recursively,
\[
\aligned
u_{n+1}(\x)&=\DI_0(\x)+(1-\DI_0(\x))\left(1-\exp(-\frac{1+\theta}{5}\wt{u_n}(\x))\right);\\
v_{n+1}(\x)&=\DI_0(\x)+(1-\DI_0(\x))\left(1-\exp(-\frac{1+\theta}{5}\wt{v_n}(\x))\right).
\endaligned
\]
By induction, it is easy to show that $u_n$  increases  to the minimal  solution to~\eqref{eq:f_recur}, denoted by $u$. Similarly, $v_n$ decreases to the maximal solution to \eqref{eq:f_recur}, denoted by $v$.

We first argue that
\[
(u(\x)-\DI_0(\x))/(1-\DI_0(\x))>\iota.
\]
From \eqref{eq:S_recur}, we see that $\DR_n=u_n$ and hence $\DR_\infty=u$. By \eqref{eq:final_bd}, we get the above assertion.

Note that when $a>b\in[\iota,1]$, we have
\[
(1-\exp(-(1+\theta)a))-(1-\exp(-(1+\theta)b))\leq (a-b).
\]
Therefore, we have
\[
\aligned
v(\x)-u(\x)&\leq \frac{v(\x)-\DI_0(x)}{1-\DI_0(\x)}-\frac{u(\x)-\DI_0(x)}{1-\DI_0(\x)}\\
&= \left(1-\exp(-\frac{1+\theta}{5}\wt{v}(\x))\right)-\left(1-\exp(-\frac{1+\theta}{5}\wt{u}(\x))\right)\\
&\leq \frac{\wt{v}(x)}{5}-\frac{\wt{u}(x)}{5}.
\endaligned
\]
This implies that $v-u$ is a (discrete) subharmonic function. On the other hand, $v-u$ is bounded. The classical result in subharmonic functions states  that a (two dimensional continuous) subharmonic function bounded from above must be constant. This result also holds for discrete subharmonic functions in $\ze^2$; see, e.g., \cite{RSV97}. Hence, $v-u$ is constant, say~$c$, and the above display must take equality.  In order for the above display to take equality, $c$ must be $0$. Now we finish the proof of the proposition except for the last assertion in \eqref{eq:relation_solu_iota}.
We proceed to prove the last assertion in \eqref{eq:relation_solu_iota}. Write~$w_n$  for the solution $f$ when $\DI_0(x,y)=\bm{\delta}_{|x|+|y|\geq n}$. By monotonicity, we get that when $\mathbf{d}(\x,\mathrm{Supp}(\DI_0))\geq n$, $f(\x)\leq w_n(0)$. It suffices to show that
\begin{equation}\label{eq_l_3}
    \liminf_{n\rightarrow\infty}w_n(\0)\leq \iota.
\end{equation}
By monotonicity, $w_n(\x)$ decreases in $n$.  Hence, we define
\[
w(\x)=\lim_{n\rightarrow\infty}w_n(\x).
\]
Note that $w_n$ satisfies \eqref{eq:f_recur}. Letting $n$ go to infinity, we get that $w$ satisfies~\eqref{eq:f_recur} for all $\x\in\ze^2$ with $\DI_0\equiv0$. Now for this trivial initial condition, we analyze the maximal solution. As before,  define $v_0\equiv1$,
\[
v_{n+1}(\x)=\DI_0(\x)+(1-\DI_0(\x))\left(1-\exp(-\frac{1+\theta}{5}\wt{v_n}(\x))\right).
\]
One can see that $v_n$ is constant and so is the maximal solution $v:=\lim_{n\rightarrow\infty}v_n$. The constant solution to \eqref{eq:f_recur} is either $\iota$ or $0$. Hence we get that $v\equiv \iota$ and therefore \eqref{eq_l_3}.
\end{proof}

{We now prove the result about $\sum_{i=1}^\infty \ell^{(i)}$ mentioned in the Introduction.}
\begin{lemma}\label{lm:sum_ell} {When $\theta>1.5$,}  the $\ell^{(i)}$'s defined in \eqref{eq:ell_i} satisfy that
\[
\iota=\sum_{i=1}^\infty \ell^{(i)}.
\]
\end{lemma}
\begin{proof}
Consider the deterministic system with initial condition $\DI_0=\bm{\delta}_{x+y=0}$. By Proposition \ref{prop:conv_front_layers_const_IC}, we have that $\lim_{n\rightarrow\infty}\DD_k(n,0)=\sum_{i=1}^{k}\ell^{(i)}$. Letting $\x=(n,0)$ with $n\rightarrow\infty$ in \eqref{eq:S_recur}, we get that $s_k=\sum_{i=1}^{k}\ell^{(i)}$ satisfies
\[
s_{k+1}=1-\exp\left(-\frac{1+\theta}{5}(2s_{k+1}+s_k+2s_{k-1})\right).
\]
Letting $k\rightarrow \infty$, we get that $s=\lim_{k\rightarrow\infty}s_k=\sum_{i=1}^{\infty}\ell^{(i)}$ satisfies
\[
s=1-\exp(-(1+\theta)s),
\]
which is the same equation as \eqref{eq:iota}, the equation for $\iota$. Note that because $\theta>1.5$, $s\geq \ell^{(1)}>0$. Hence, $s$ must equal $\iota$.
\end{proof}

We are ready to prove Theorems \ref{thm:surv_prob}, \ref{thm:speed}(ii) and \ref{thm:ult_inf_prop_2}, starting with Theorem~\ref{thm:surv_prob}.

\begin{proof}[Proof of Theorem~\ref{thm:surv_prob}]
Note that Theorem~\ref{thm:speed}(ii) guarantees that $q_N\rightarrow 1$ for IC2. Hence we assume that the initial condition is IC1.

Clearly, $q_N$ is not more than the survival probability of the branching process with offspring distribution $\mathrm{Bin}(5N,(1+\theta)/(5N))$. It is easy to show that the last survival probability converges to $\iota$. This gives the upper bound.

We now show the lower bound. Write $\mathcal{L}(k)=\{(m,l)\in \ze^2:m+l=k\}\times [N]$. We first consider the SIR process with initial condition $N\bm{\delta}_{m+l=0}$. For any $m,l\in\ze$,  we have
\[
E\left(\frac{R_\infty^N(m,l)}{N}\right)=P^N(((m,l),1)\mathrm{~is~connected~to~}\mathcal{L}(0)),
\]
where $((m,l),1)$ represents the first particle at site $(m,l)$ and we use superscript $N$ to indicate that the village size is $N$.
By \eqref{eq:R_final}, for any $\vep>0$, there is an $N_0$ such that when $N>N_0$, we have
\[
\limsup_{n\rightarrow \infty} P^N(((n,0),1)\mathrm{~is~connected~to~}\mathcal{L}(0))>\iota-\vep.
\]
Note that the probability above is decreasing in $n$, hence we get
\[
 \inf_{n\in\rN}P^N(((n,0),1)\mathrm{~is~connected~to~}\mathcal{L}(0))>\iota-\vep.
\]
By symmetry, we get
\[
\inf_{n\in\rN}P^N(({(0,0)},1))\mathrm{~is~connected~to~}\mathcal{L}(n))>\iota-\vep.
\]
Therefore, we have $q_N>\iota-\vep$.
\end{proof}

Next, we show Theorem~\ref{thm:speed}(ii).
\begin{proof}[Proof of Theorem~\ref{thm:speed}(ii)]
By symmetry, we assume all $(i_k,j_k)$ satisfy $i_k\geq j_k\geq0$. We first show the upper bound for IC1 conditional on survival and IC2:
\[
\limsup_{k\to\infty}P\left(\frac{R_{\lfloor (i_k+j_k)(\upsilon(\theta,\phi)^{-1}+\eps)\rfloor}(i_k,j_k)}{N}>\iota+\eps\right)<\eps.
\]
Note that in order to show the inequality above for IC1 conditional on survival, it suffices to show the inequality unconditionally, and by monotonicity, it suffices to do so for IC2. This follows from  \eqref{eq:R_final} because $R_n\leq R_\infty$ for any $n$.

Now we turn to the lower bound:
\begin{equation}\label{eq:speed_lower}
 \limsup_{k\to\infty}P\left(\frac{R_{\lfloor (i_k+j_k)(\upsilon(\theta,\phi)^{-1}+\eps)\rfloor}(i_k,j_k)}{N}<\iota-\eps\right)<\eps.
\end{equation}
We first consider the initial condition IC2. We will show a stronger version:
\begin{equation}\label{eq:speed_lower_strong}
 \limsup_{k\to\infty}P\left(\frac{R^+_{\lfloor (i_k+j_k)(\upsilon(\theta,\phi)^{-1}+\eps)\rfloor}(i_k,j_k)}{N}<\iota-\eps\right)<\eps,
\end{equation}
where we write $R^+(i,j)$ for the number of particles that are infected via edges in the upper half plane $\{(x,y)|x+y > 0\}$.

Our method is similar to {the one used in the proof of \eqref{eq:lim_S}}. By Remark \ref{rk:speed_strong}, there exists $\delta_1>0$ such that when $N$ is sufficiently large, for any $(\y_k)$ satisfying $||\y_k||_1\to\infty$ and $\arg(\y_k)\to \phi$,
\begin{equation}\label{eq_speed_v1}
 \limsup_{k\to\infty}P\left(\frac{R'_{\lfloor ||\y_k||_1(\upsilon(\theta,\phi)^{-1}+0.5\eps)\rfloor}(\y_k)}{N}<\delta_1\right)<\eps/2,
\end{equation}
where we write $R'(i,j)$ for the  number of particles that are infected via  edges  in the region $\{(x,y):\ 0< x+y < i+j\}$.

On the other hand, similarly to \eqref{eq_v2}, there exist $n_0$ and $n_1$ such that with the initial condition $I_0(\x)=\lceil\delta_1 N\rceil\bm{\delta}_{\0}$, when $N$ is large enough,
\begin{equation}\label{eq_speed_v2}
P\left(\frac{\Dd^{N,+}_{n_0}({n_1,0})}{N}<\iota-\vep\right)<\vep/2,
\end{equation}
where $\Dd^{N,+}_{n_0}({n_1,0})$ stands for the number of particles at $(n_1,0)$ that are connected to the $\lceil\delta_1 N\rceil$ initial infected particles at the origin via no more than $n_0$ edges in the upper half space $\{(x,y):x+y\geq 0\}$.

Combining the last two displays, we can show \eqref{eq:speed_lower}. Indeed, by \eqref{eq_speed_v1}, when~$k$ is large, with probability at least $ 1-\eps/2$, $R'_{\lfloor ||\y_k||_1(\upsilon(\theta,\phi)^{-1}+0.5\eps)\rfloor}(\y_k) \geq {\lceil\delta_1 N\rceil}$ for $\y_k=(i_k-n_1,j_k)$. On the other hand, it follows from \eqref{eq_speed_v2} that starting from  those $\lceil\delta_1 N\rceil$ infected particles at $\y_k$, with  probability at least $1-\eps/2$, after $n_0$ unit of time, there are at least $\lceil(\iota-\vep) N\rceil$ particles at site $\y_k+(n_1,0)=(i_k,j_k)$ that are are connected to the $\lceil\delta_1 N\rceil$ infected particles at $\y_k$ via  open edges inside the upper half space {$\{(x,y):x+y\geq||\y_k||_1\}$}.  Note that the two events above are independent because they depend on the edges that are on different sides of $x+y=||\y_k||_1$. Therefore, with probability at least $1-\eps$,  the total number of recovered particles at site $(i_k,j_k)$ is at least $\lceil(\iota-\vep_0) N\rceil$ at time $||\y_k||_1(\upsilon(\theta,\phi)^{-1}+0.5\eps)+n_0\leq (i_k+j_k)(\upsilon(\theta,\phi)^{-1}+\eps)$. This proves \eqref{eq:speed_lower_strong}.

It remains to prove \eqref{eq:speed_lower} for IC1 conditional on survival. As before, it suffices {to} show that under IC1,
\begin{equation}\label{eq_speed_lower_IC1}
 \limsup_{k\to\infty}P\left(\frac{R_{\lfloor (i_k+j_k)(\upsilon(\theta,\phi)^{-1}+\eps)\rfloor}(i_k,j_k)}{N}<\iota-\eps, (\0,1)\leftrightarrow \infty\right)<\eps,
\end{equation}
where we write $(\0,1)\leftrightarrow\infty$ for the event that $(\0,1)$ belongs to the infinite cluster, or equivalently, the epidemic lasts forever under IC1.

Our method is similar to that used in Section \ref{ssec:front_1_infect}, namely, to reduce IC1 to IC2. We fixed some $\gamma_0\in(0,1]$, say, $0.1$. Consider the $N$-percolation. For any $n,l\in\rN,m\in\ze$, define the event
\[
B(n,l,m)=
\{ R^-_{n+l}(m,n-m)\geq \gamma_0N\},
\]
where for any $i,j\in\zz{Z},$  $R^-_{n}(i,j)$ stands for the number of particles on site $(i,j)$ that are connected to $(\0,1)$ via at
most $n$ open edges in the lower half plane $\{(x,y):x+y\leq i+j\}$.
Define
\[
B(n,l)=\cup_{m\in\ze} B(n,l,m),\,
B(n)=\cup_{l\in\rN}B(n,l),\mbox{ and } A(k)=\cup_{n=k}^{2k-1}B(n).
\]
We omit the proof of the following lemma because it can be proved in a very similar way to  Lemma~\ref{lm:reduction}.
\begin{lemma}\label{lm:reduction2}For any $N\in\zz{N}$,
\[
    \lim_{n\rightarrow\infty}P^N(A^{\mathrm{c}}(n), (\0,1)\leftrightarrow\infty)=0.
\]
\end{lemma}

For simplicity, we write $\Box$ for the event in \eqref{eq:speed_lower}. We need to show that $\Box\cap\{(\0,1)\leftrightarrow\infty\}$ has a small probability when $N,k$ are large. For any $\vep>0$ and any fixed $N$, by Lemma \ref{lm:reduction2}, we can find a $n_0$ large enough such that
\begin{equation}\label{eq_s1}
P(A^{\mathrm{c}}(n_0), (\0,1)\leftrightarrow\infty)<\vep.
\end{equation}
By the definition of $A(n_0)$, we can find an $l_0$ such that
\begin{equation} \label{eq_s3}
   P[A(n_0)\setminus\left(\cup_{n_0\leq n\leq 2n_0-1}\cup_{l\leq l_0}B(n,l)\right)]<\vep.
\end{equation}
We then get that
\[
\aligned
P&(\Box,(\0,1)\leftrightarrow\infty)\leq P(\Box, A(n_0))+P(\Box, A^{\mathrm{c}}(n_0),(\0,1)\leftrightarrow\infty)\\
&\leq P(\Box, \cup_{n_0\leq n\leq 2n_0-1}\cup_{l\leq l_0}B(n,l))+2\vep.
\endaligned
\]
Note that $B(n,l)=\cup_{m\in\ze}B(n,l,m)=\cup_{-l\leq m\leq n+l}B(n,l,m)$. By monotonicity, one can get that when $\min\{i_k+j_k, (i_k+j_k)(\upsilon(\theta,\phi)^{-1}+\eps) \}> 2n_0+l_0$,
\begin{equation}\label{eq_s2}
   \aligned
P&(\Box, \cup_{n_0\leq n\leq 2n_0-1}\cup_{l\leq l_0}B(n,l))\leq\\
&\max_{n,l,m:n_0\leq n\leq 2n_0-1,-l_0\leq m\leq 2n_0+l_0,l\leq l_0}
P(\Box|B'(n,l,m)),
\endaligned
\end{equation}
where we write $P(\cdot| B'(n,l,m))$ for the probability law that the initial time is $t_0=n+l$ and the initial condition is $I_{t_0}(\x)=\lceil\gamma_0 N\rceil\bm{\delta}_{(m,n-m)}, R_{t_0}= (N-I_{t_0})\bm{\delta}_{x+y\leq n}$.

Combining the  results above and \eqref{eq:speed_lower_strong}, we can show  \eqref{eq_speed_lower_IC1} as follows. For any $\vep_0>0$, by \eqref{eq:speed_lower_strong}, we can find a $\delta>0$ such that when $N$ is large enough, say $N\geq N_0$, \eqref{eq:speed_lower_strong} holds with $\vep=0.1\vep_0$. For any $N\geq N_0$ fixed, we find some $n_0,l_0$ such that \eqref{eq_s1}, \eqref{eq_s3} and \eqref{eq_s2} hold with $\vep=0.1\vep_0$. By  monotonicity and \eqref{eq:speed_lower_strong}, we obtain that when $k$ is large enough, the right hand side of \eqref{eq_s2} is no more than $0.2\vep_0$. Summing up, we see that \eqref{eq_speed_lower_IC1} is true with $\vep=\vep_0$.
\end{proof}

Finally, we show Theorem~\ref{thm:ult_inf_prop_2}.
\begin{proof}[Proof of Theorem~\ref{thm:ult_inf_prop_2}]
In the proof below, we use $P^N$ to denote the corresponding probability distribution of the $N$-percolation model, and $E^N$ for the expectation. We note that it suffices to prove the following:
\begin{equation}\label{eq_V1_V2}
\aligned
    \liminf_{N\rightarrow\infty}\inf_{\x\in \ze^2}P^N((\x,1)\leftrightarrow\infty,(\mathbf{0},2)\leftrightarrow\infty)\geq \iota^2,\\ \limsup_{N\rightarrow\infty}\sup_{\x\in\ze^2}P^N((\x,1)\leftrightarrow\infty,(\x,2)\leftrightarrow\infty, (\mathbf{0},3)\leftrightarrow\infty)\leq\iota^3,
\endaligned
\end{equation}
where, recall that, $(\x,i)$ represents the $i$-th particle at site $\x$ and we use
$V\leftrightarrow \infty$  for the event that $V$
belongs to the infinite cluster. By standard results in percolation theory, there exists almost surely a unique infinite cluster when $P^N((\0,1)\leftrightarrow\infty)>0$; see, e.g., Section~8.2 of \cite{Grimmett99}. We first show that \eqref{eq_V1_V2} implies \eqref{eq:R_final_IC1}. In fact, by symmetry, we have
\[
E^N\left(\frac{\sum_{i=1}^N1_{(\x,i)\leftrightarrow\infty}}{N}|(\mathbf{0},1)\leftrightarrow\infty\right)=\frac{\sum_{i=1}^NP^N((\x,i),(\mathbf{0},1)\leftrightarrow\infty)}{NP^N((\mathbf{0},1)\leftrightarrow\infty)}\geq \iota-o(1);
\]
and
\[
\aligned
&E^N\left(\left(\frac{\sum_{i=1}^N1_{(\x,i)\leftrightarrow\infty}}{N}\right)^2|(\mathbf{0},1)\leftrightarrow\infty\right)=\frac{\sum_{i,j=1}^NP^N((\x,i),(\x,j),(\mathbf{0},1)\leftrightarrow\infty)}{N^2P^N((\mathbf{0},1)\leftrightarrow\infty)}\\
\leq&\frac{N+N(N-1)P^N((\x,2)\leftrightarrow\infty,(\x,3)\leftrightarrow\infty,(\mathbf{0},1)\leftrightarrow\infty)+2N1_{\x=\mathbf{0}}}{N^2P^N((\mathbf{0},1)\leftrightarrow\infty)}\\
\leq&\iota^2+o(1).
\endaligned
\]
Therefore,
\[
\aligned
&\sup_{\x\in\ze^2} E\left(\left(\left.\frac{R^N_\infty(\x)}{N}\right|(\mathbf{0},1)\leftrightarrow\infty\right)-\iota\right)^2\\
\leq & \sup_{\x\in\ze^2}E\left(\frac{R^N_\infty(\x)}{N}|(\mathbf{0},1)\leftrightarrow\infty\right)^2
-2\iota \inf_{\x\in\ze^2}E\left(\left.\frac{R^N_\infty(\x)}{N}\right|(\mathbf{0},1)\leftrightarrow\infty\right)+\iota^2\\
\leq &(\iota^2+o(1))-2\iota(\iota-o(1))+\iota^2=o(1).
\endaligned
\]
The conclusion \eqref{eq:R_final_IC1} follows.

It remains to prove \eqref{eq_V1_V2}. By the FKG inequality (Theorem~2.4 of \cite{Grimmett99}) and Theorem \ref{thm:surv_prob}, we have
\[
P^N((\x,1)\leftrightarrow\infty,(\mathbf{0},2)\leftrightarrow\infty)\geq P^N((\x,1)\leftrightarrow\infty)P^N((\mathbf{0},2)\leftrightarrow\infty)\rightarrow \iota^2.
\]
On the other hand, for any fixed $\vep>0$, by Theorem~\ref{thm:final_prop}, we can find some $k\geq 1$ such that when the initial condition is $\DI_0=\bm{\delta}_{x+y=-k}$, we have
\[
\DR_\infty(i,0)<\iota+\vep,\quad \mbox{ for all } i\in\{0,\dots,4k\}.
\]
By \eqref{eq:D_to_R}, we can find some $N_0$ such that when $N\geq N_0$, under the initial condition  $I_0=N\bm{\delta}_{x+y=-k}$,
\begin{equation}\label{eq_v3}
P\left(\frac{R_\infty^N(\y)}{N}\leq \iota+\vep, \mbox{ for all } \y \in\{\y\in \ze^2:||\y-(k,k)||_1\leq 2k\}\right)>1-\vep.
\end{equation}
Write ${B}=\{(x,y)\in\ze^2:x+y=-k\}\times[N]$ and $A=\{\y\in \ze^2:||\y-(k,k)||_1\leq 2k\}$. Note that by symmetry, conditionally on the event inside the probability in \eqref{eq_v3}, for any $\y_1,\y_2\in A$, we have
\[\aligned
P^N((\y_1,1)\leftrightarrow B)&\leq \iota+\vep, \quad P^N((\y_1,1),(\y_1,2)\leftrightarrow B)\leq ({\iota}+\vep)^2;\\
&P^N((\y_1,1),(\y_1,2),(\y_2,3)\leftrightarrow B)\leq (\iota+\vep)^3.
\endaligned
\]
Therefore, by \eqref{eq_v3}, unconditionally,  when $\y_1,\y_2\in A$, we have
\begin{equation}\label{eq_V1_V2_L}
\aligned
P^N((\y_1,1)\leftrightarrow B)&\leq \iota+2\vep, \quad P^N((\y_1,1),(\y_1,2)\leftrightarrow B)\leq (\iota+\vep)^2+\vep;\\
&P^N((\y_1,1),(\y_1,2),(\y_2,3)\leftrightarrow B)\leq (\iota+\vep)^3+\vep.
\endaligned
\end{equation}
By symmetry and monotonicity, we get that for all $\y_1\in\{(x,y):x+y\geq0\}$,
\begin{equation}\label{eq_V_L}
   P^N((\y_1,1)\leftrightarrow B)\leq \iota+2\vep, P^N((\y_1,1), (\y_1,2)\leftrightarrow B)\leq (\iota+\vep)^2+\vep.
\end{equation}

Now we can show the second assertion of \eqref{eq_V1_V2}. Without loss of generality,  we replace $\mathbf{0}$ by $\x_0=(k,k)$. When $\x\in A$, by \eqref{eq_V1_V2_L},
\[
P^N((\x,1),(\x,2), (\x_0,3)\leftrightarrow\infty)\leq P^N((\x,1),(\x,2),(\x_0,3)\leftrightarrow B)\leq (\iota+\vep)^3+\vep.
\]
Note that because there is a unique infinite cluster and this cluster intersects~$B$, $V\leftrightarrow\infty$ implies  $V\leftrightarrow B$.

On the other hand, when $\x\notin A$, we can find a line $l$ with the form $\bm{\delta}_{x+y=c}$ or $\bm{\delta}_{x-y=c}$ such that both $\x$ and $\x_0$ are on different sides of $l$ and that $\x$ and $\x_0$ are at  distance at least $k$ away from $l$. Let $L=l\times[N]$. We have
\[\aligned
P^N&((\x,1),(\x,2), (\x_0,3)\leftrightarrow\infty)\leq P^N((\x,1),(\x,2),(\x_0,3)\leftrightarrow L)\\
&=P^N((\x,1),(\x,2)\leftrightarrow L)P((\x_0,3)\leftrightarrow L)\leq((\iota+\vep)^2+\vep)(\iota+2\vep),
\endaligned
\]
where the middle equality follows from the fact that the events are independent because $\x$ and $\x_0$ lie on  different sides of $L$, and the last inequality is due to~\eqref{eq_V_L}.
Letting $\vep$ go to zero, we finish the proof.
\end{proof}

\section*{Acknowledgements}
We thank Eyal Neuman for many helpful discussions. Research is partially supported by the HKUST IAS Postdoctoral Fellowship and RGC grant GRF 16304019 of the HKSAR.


\begin{thebibliography}{15}
\newcommand{\enquote}[1]{``#1''}
\expandafter\ifx\csname natexlab\endcsname\relax\def\natexlab#1{#1}\fi

\bibitem[{Athreya and Ney(1972)}]{athreya72}
Athreya, K.~B. and Ney, P.~E. (1972), \textit{Branching processes}, New York:
  Springer-Verlag, die Grundlehren der mathematischen Wissenschaften, Band 196.

\bibitem[{Biggins(1976)}]{Biggins76}
Biggins, J.~D. (1976), \enquote{The first- and last-birth problems for a
  multitype age-dependent branching process,} \textit{Advances in Appl.
  Probability}, 8, 446--459.

\bibitem[{Bramson(1978)}]{Bramson78}
Bramson, M.~D. (1978), \enquote{Minimal displacement of branching random walk,}
  \textit{Z. Wahrsch. Verw. Gebiete}, 45, 89--108.

\bibitem[{Cox and Durrett(1988)}]{CD88}
Cox, J.~T. and Durrett, R. (1988), \enquote{Limit theorems for the spread of
  epidemics and forest fires,} \textit{Stochastic Process. Appl.}, 30,
  171--191.

\bibitem[{Durrett(1984)}]{Du84}
Durrett, R. (1984), \enquote{Oriented percolation in two dimensions,}
  \textit{Ann. Probab.}, 12, 999--1040.

\bibitem[{Durrett and Liggett(1981)}]{DL81}
Durrett, R. and Liggett, T.~M. (1981), \enquote{The shape of the limit set in
  {R}ichardson's growth model,} \textit{Ann. Probab.}, 9, 186--193.

\bibitem[{Grimmett(1999)}]{Grimmett99}
Grimmett, G. (1999), \textit{Percolation}, vol. 321 of \textit{Grundlehren der
  mathematischen Wissenschaften [Fundamental Principles of Mathematical
  Sciences]}, Springer-Verlag, Berlin, 2nd ed.

\bibitem[{Hammersley(1974)}]{Hammersley74}
Hammersley, J.~M. (1974), \enquote{Postulates for subadditive processes,}
  \textit{Ann. Probability}, 2, 652--680.

\bibitem[{Kermack and McKendrick(1927)}]{KM27}
Kermack, W. O.; McKendrick, A. G. (1927), \enquote{A Contribution to the Mathematical Theory of Epidemics,}
  \textit{Proceedings of the Royal Society of London. Series A, Containing Papers of a Mathematical and Physical Character}, 115 (772), 700--721.


\bibitem[{Kingman(1975)}]{Kingman75}
Kingman, J. F.~C. (1975), \enquote{The first birth problem for an age-dependent
  branching process,} \textit{Ann. Probability}, 3, 790--801.

\bibitem[{Lalley(2003)}]{Lalley03}
Lalley, S.~P. (2003), \enquote{Strict convexity of the limit shape in
  first-passage percolation,} \textit{Electron. Comm. Probab.}, 8, 135--141.

\bibitem[{Lalley(2009)}]{lalley09}
--- (2009), \enquote{Spatial epidemics: critical behavior in one dimension,}
  \textit{Probab. Theory Related Fields}, 144, 429--469.

\bibitem[{Lalley, Perkins, and Zheng(2014)}]{LPZ14}
Lalley, S.~P., Perkins, E.~A., and Zheng, X. (2014), \enquote{A phase
  transition for measure-valued {SIR} epidemic processes,} \textit{Ann.
  Probab.}, 42, 237--310.

\bibitem[{Lalley and Zheng(2010)}]{lz10}
Lalley, S.~P. and Zheng, X. (2010), \enquote{Spatial epidemics and local times
  for critical branching random walks in dimensions 2 and 3,} \textit{Probab.
  Theory Related Fields}, 148, 527--566.

\bibitem[{Liggett, Schonmann, and Stacey(1997)}]{LSS97}
Liggett, T. M. and Schonmann, R. H. and Stacey, A. M. (1997), \enquote{Domination by product measures}, \textit{Ann. Probab.}, 25,  71--95.



\bibitem[{Richardson(1973)}]{Richardson73}
Richardson, D. (1973), \enquote{Random growth in a tessellation,} \textit{Proc. Cambridge Philos. Soc.}, 74, 515--528.


\bibitem[{Rigoli, Salvatori, and Vignati(1997)}]{RSV97}
Rigoli, M., Salvatori, M., and Vignati, M. (1997), \enquote{Subharmonic
  functions on graphs,} \textit{Israel J. Math.}, 99, 1--27.

\bibitem[{Zhang(1993)}]{ZY93}
Zhang, Y. (1993), \enquote{A shape theorem for epidemics and forest fires with
  finite range interactions,} \textit{Ann. Probab.}, 21, 1755--1781.

\end{thebibliography}

\end{document}